\documentclass[11pt]{amsart}

\usepackage{amsmath}
\usepackage{amsfonts}
\usepackage{amssymb}
\usepackage{graphicx}
\usepackage[dvipsnames]{xcolor}
\usepackage{amsthm}
\usepackage[english]{babel}
\usepackage{enumerate}
\usepackage{fontawesome5}
\usepackage{hyperref} 
\usepackage{enumitem}
\usepackage[margin=2cm]{geometry}

\usepackage{hyperref}
\usepackage{orcidlink}

\newtheorem{theorem}{Theorem}[section]

\newtheorem{definition}[theorem]{Definition}

\newtheorem{proposition}[theorem]{Proposition}
\newtheorem{lemma}[theorem]{Lemma}
\newtheorem{remark}[theorem]{Remark}

\def\R {\mathbb{R}}
\def\M{H_r^\mu(\R^N)}
\def\H{H^\mu(\R^N)}

\numberwithin{equation}{section}

\title[Hartree type equation and Pohozaev Identity]{Ground state solutions for Hartree type equations driven by superposition operators and Pohozaev Identity}

\author[A. Cannone]{Alessandro Cannone \orcidlink{0009-0006-4073-8298}
}
\address{Alessandro Cannone \newline
	Dipartimento di Matematica, Universit\`{a} degli Studi di Bari Aldo Moro,\newline
	Via Orabona 4, 70125 Bari, Italy.}
\email{alessandro.cannone@uniba.it}

\author[S. Cingolani]{Silvia Cingolani \orcidlink{0000-0002-3680-9106}
}
\address{Silvia Cingolani \newline
	Dipartimento di Matematica, Universit\`{a} degli Studi di Bari Aldo Moro,\newline
	Via Orabona 4, 70125 Bari, Italy.}
\email{silvia.cingolani@uniba.it}

\author[S. Dipierro]{Serena Dipierro 
	\orcidlink{0000-0003-4386-4485}
}
\address{Serena Dipierro \newline
	Department of Mathematics and Statistics, University of Western Australia, \newline 35 Stirling Highway, 6009 Crawley, Australia.}
\email{serena.dipierro@uwa.edu.au}

\begin{document}

\maketitle
\begin{abstract}
We investigate Hartree-type equations driven by a nonlocal operator $\mathcal{L}_\mu$, defined as a superposition of fractional Laplacians through a signed Borel measure $\mu$. Under Berestycki-Lions type assumptions, we prove the existence of a Mountain Pass solution and show that its energy level coincides with the minimum on the Pohozaev manifold. 
\noindent
We also establish the boundedness of non-negative solutions. The proof of this fact requires a careful use of the Sobolev embedding in the
iterative argument and a delicate treatment of the integrals involved in the estimates, as well as a Kato-type inequality in our general setting. 
\noindent
Finally, we establish a general Pohozaev identity for solutions under a suitable summability assumption.
\end{abstract}

\smallskip
\noindent
\textbf{MSC2020:} 
35A01,  35B38, 35J20, 35Q40, 35R11

\smallskip
\noindent
\textbf{Key words:} 
Hartree type equations; superposition operators; fractional laplacian; Pohozaev identity; ground state solutions

\tableofcontents

\section{Introduction}

In this paper, we study the existence of Mountain Pass solutions for a class of Hartree-type equations involving a (possibly nonlocal) operator defined as a superposition of fractional Laplacians of different orders. Specifically, following the setting introduced in~\cite{DPSV},
let $\mu^+$ and $\mu^-$ be two nonnegative finite Borel measures on $[0,1]$, and let $\mu := \mu^+ - \mu^-$ denote the corresponding signed measure. 

The main operator under investigation is then defined as 
\begin{equation}\label{def. superposition operator}
    \mathcal{L}_\mu u := \int_{[0,1]} (-\Delta)^s u \, d\mu(s),
\end{equation} 
where, for any $s \in (0,1)$, $(-\Delta)^s$ denotes the fractional Laplacian, acting on a function $u$ as
\begin{equation}\label{fractional laplacian}
    (-\Delta)^s u(x) := C_{N,s} \int_{\mathbb{R}^N} \frac{2u(x) - u(x+y) - u(x-y)}{|y|^{N+2s}} \, dy.
\end{equation}
The positive normalization constant $C_{N,s}$ is explicitly given by
\begin{equation}\label{constant_definition}
    C_{N,s} = \frac{2^{2s} s \Gamma \left(\frac{N+2s}{2}\right)}{\pi^{N/2} \Gamma(1-s)},
\end{equation}
where $\Gamma$ denotes the Euler Gamma function. This constant is chosen in
such a way that, for any smooth and rapidly decaying function $u$, the Fourier transform of $(-\Delta)^s u$ is given by $(2\pi|\xi|)^{2s}\widehat{u}(\xi)$, see e.g.~\cite[Proposition~3.3]{DPV}.
Furthermore, this definition is consistent with the limit cases $s \to 1$ and $s \to 0$, in the sense that
\begin{equation*}
\lim_{s \to 1} (-\Delta)^s u = -\Delta u \quad \text{and} \quad \lim_{s \to 0} (-\Delta)^s u = u,
\end{equation*}
see e.g.~\cite[Proposition~4.4]{DPV}.

Special cases of the operator $\mathcal{L}_\mu$ include the standard Laplacian (when $\mu$ is the Dirac measure $\delta_1$), the fractional Laplacian $(-\Delta)^{s_\star}$ for some $s_\star \in (0,1)$ (when $\mu = \delta_{s_\star}$), and mixed operators of the form~$(-\Delta)^{s_1} + (-\Delta)^{s_2}$ (when $\mu = \delta_{s_1} + \delta_{s_2}$).
The setting that we consider in this paper also allows us to consider operators of the form~$(-\Delta)^{s_1} -\alpha (-\Delta)^{s_2}$ for some~$\alpha>0$, namely we can have terms with the ``wrong sign''. 
Moreover a ``continuous'' superposition of operators of different fractional orders can be taken into account.
A list of interesting cases that are comprised by this setting is provided
on pages~6-7 of~\cite{DPSV} (and also Section~5 there).

Operators that emerge from superpositions of local and nonlocal operators have been extensively studied in recent research from a variety of perspectives. This includes regularity theory (see~\cite{BCCI12, CDV22, BDVV22, GK22, SVWZ,
GL23, DFM, SVWZ2, AAM26}), existence and nonexistence results (see~\cite{AC21, BDVV25, DPLSV25, DSVZ}),
symmetry results (see~\cite{CS16, BDVV}), 
geometric and variational inequalities (see~\cite{BDVV23, BDVVa}). Furthermore, operators of
this type naturally arise in concrete applications, for instance they
model the dispersal of a biological population whose individuals are subject to different kinds
of diffusive strategies (see~\cite{DV21, DPLV1,DPLV2, DPLSV}).

 \medskip

In our setting, we assume that there exist $\bar s \in (0,1]$ and $\gamma \geq 0$ such that the following conditions hold:
\begin{eqnarray}
   && \mu^+([\bar s, 1])  > 0, \label{positivity property 1} \\
   && \mu^-|_{[\bar s, 1]}  = 0, \label{positivity property 2} \\
  &&  \mu^-([0,\bar s]) \leq \gamma \mu^+([\bar s,1]). \label{positivity property 3}
\end{eqnarray}
Conditions \eqref{positivity property 1} and \eqref{positivity property 2} say that the component of the signed measure~$\mu$ supported on higher fractional exponents is strictly positive. Meanwhile, condition \eqref{positivity property 3} prescribes that the negative part of~$\mu$ is appropriately balanced, or ``reabsorbed'', by the positive part.

We note that, by assumption \eqref{positivity property 1}, there exists $s_\sharp \in [\bar s, 1]$ such that 
\begin{equation} \label{positivity measure 4}
    \mu^+([s_\sharp, 1]) > 0.
\end{equation}
We point out that the choice of $s_\sharp$ is somehow arbitrary: as a matter of fact, $s_\sharp$ will
play the role of a critical exponent and therefore the idea is that one would like to pick $s_\sharp$
``as large as possible'' but still satisfying~\eqref{positivity measure 4}.

The main purpose of this work is to investigate the existence of Mountain Pass type solutions to the following problem involving\footnote{As customary, the symbol~$ \ast$ represents the convolution in~$\R^N$, namely
$$ (g\ast h)(x)=\int_{\R^N} g(y)h(x-y)\,dy.$$} the operator $\mathcal{L}_\mu$:
\begin{equation} \label{problem}
    \mathcal{L}_\mu u + \beta u = (I_\alpha \ast F(u))f(u) \quad \text{in } \mathbb{R}^N,
\end{equation}
where $\beta > 0$, $N \geq 3$, $\alpha \in (0,N)$, and 
$$I_\alpha(x) := C_\alpha |x|^{-(N-\alpha)},$$ for some~$C_\alpha>0$, is the Riesz potential. 

Here $$F(t) := \int_0^t f(s) \, ds$$ and the nonlinearity $f: \mathbb{R} \to \mathbb{R}$ is assumed to satisfy conditions of Berestycki-Lions type (see~\cite{B-L}):
\begin{itemize}
\item [$(F_1)$] $f\in C(\mathbb{R}, \mathbb{R});$
\item[$(F_2)$] there exists $C>0$ such that, for every $t \in \mathbb{R}$, 
\begin{align*} |tf(t)|\leq C\Big(|t|^{\frac{N+\alpha}{N}}+|t|^{\frac{N+\alpha}{N-2s_{\sharp}}}\Big);
\end{align*}
    \item [$(F_3)$]$F$ satisfies
    \begin{align*}
        i)\lim\limits_{t\to 0^+} \frac{F(t)}{|t|^{\frac{N+ \alpha}{N}}}=0, \qquad\qquad ii)\lim\limits_{t\to +\infty} \frac{F(t)}{|t|^{\frac{N+\alpha}{N-2s_{\sharp}}}}=0;
    \end{align*}
    \item[$(F_4)$] there exists $t_0 \in \mathbb{R}\setminus\{0\}$ such that $F(t_0)\not=0.$
\end{itemize}
We recall that the value of~$s_\sharp$ is determined according to condition \eqref{positivity measure 4}.

Throughout the rest of the paper, unless otherwise specified, we will assume that the assumptions in~$(F_1)$--$(F_4)$ are in force.

When~$\mu$ is the Dirac mass concentrated at $s=1$, and therefore the operator~$\mathcal{L}_\mu$
boils down to the classical Laplacian~$-\Delta$, the 
equation in~\eqref{problem} has been dealt with in~\cite{MV}. Moreover,
when $\mu$ is the Dirac mass concentrated at some $s_\star \in (0,1)$, and thus~$\mathcal{L}_\mu$ reduces to the
fractionan Laplacian~$(-\Delta)^{s_\star}$, the equation in~\eqref{problem} has been taken into account
in~\cite{CGT}.
\medskip

The investigation of nonlocal models such as equation \eqref{problem} allows for the analysis of complex systems characterized by long-range interactions and non-standard diffusion regimes. The distributed order operator $\mathcal{L}_\mu$ serves as a natural and effective generalization of the fractional Laplacian $(-\Delta)^s$ by accounting for a superposition of various spatial scales through the measure $\mu(s)$. Whenever this measure is concentrated at a single point $s \in (0,1)$, the model recovers the standard fractional Laplacian structure, which has been extensively documented in the literature to describe anomalous diffusion processes. For instance, such operators arise in population dynamics \cite{CDV, DPLV1, DPLSV}, where a range of diffusion exponents allows for the modeling of specific ecological environments. 

In the quantum realm, the Choquard component is essential for analyzing the mean-field limit of weakly interacting molecules \cite{AM, DSS, HMT} and their orbital stability \cite{CFHMT}, further extending to the physics of graphene \cite{LMM} to characterize particle interactions. This theoretical versatility ultimately culminates in astrophysical applications concerning the study of exotic stars, such as pseudorelativistic boson stars and white dwarfs \cite{ES, FJL1, FJL2, FL, HL, L1, L2, LL, LY1, HLLS}.
In this framework, the descriptive power of equation \eqref{problem} is fully realized when the measure $\mu$ is concentrated in a three-dimensional setting ($N=3$) at $s=1/2$. In this configuration, the model recovers the massless boson stars equation
\begin{equation*}
    \sqrt{-\Delta} u + \beta u = \left( \frac{1}{4\pi |x|} \ast u^2 \right)u \quad \text{in } \mathbb{R}^3,
\end{equation*}
where the pseudorelativistic operator $\sqrt{-\Delta+m}$ collapses into the square root of the Laplacian \cite{FrankLenzmann, HL, LL}. The analysis of ground state solutions for such a system not only ensures an understanding of gravitational equilibrium states but also provides vital insights into critical initial conditions and the Chandrasekhar limiting mass~\cite{L1}.
\medskip

The main contributions of the present paper are summarized in the following statements. Our first result establishes the existence of a solution of Mountain Pass type.

\begin{theorem}\label{MP solution} 
There exists~$\gamma_0>0$, depending on~$N$, $\alpha$, $\beta$, $\mu^+$ and $\bar s$,
such that for all~$\gamma\in[0,\gamma_0]$ the following statement holds true.

Let~$\mu$ satisfy~\eqref{positivity property 1}, \eqref{positivity property 2}, and~\eqref{positivity property 3}. 
Then, there exists a non trivial, radially symmetric weak solution~$\bar{u}$ to problem~\eqref{problem}
which is of Mountain Pass type and satisfies the Pohozaev identity
\begin{eqnarray*}&&
\int_{[0,1]}\left( \frac{(N-2s)C_{N,s}}{2}\iint_{\R^{2N}}\frac{(\bar u(x)-\bar u(y))^2}{|x-y|^{N+2s}}\,dx\,dy\right)\, d\mu(s)  +\frac{\beta N}{2}\int_{\R^N}\bar{u}^2(x)\,dx\\&&\qquad\qquad=\frac{N+\alpha}{2}\int_{\R^N} (I_\alpha\ast F(\bar u))F(\bar u)\,dx.
\end{eqnarray*}
\end{theorem}

We refer to Section~\ref{Variational Setting and existence result} for the precise definition of weak solution (see in particular Definition~\ref{weak solution}). 

What is more, we can characterize the energy level of the solution found in Theorem~\ref{MP solution}, showing that
it coincides with the minimum level on the Pohozaev manifold
(the technical definitions of Mountain Pass level and Pohozaev minimum level will be provided in formulas~\eqref{mplevel00} and~\eqref{poholevel00} below). As a consequence, we have the following: 

\begin{theorem}\label{p=l}
The solution~$\bar u$ found in Theorem~\ref{MP solution}
minimizes the energy among all the solutions satisfying the Pohozaev identity.
\end{theorem}

We point out that the presence of operators with the ``wrong sign'' in~$\mathcal{L}_\mu$ makes it
more complicated to establish the existence of a Mountain Pass solution. Indeed, 
the smallness assumption on the parameter~$\gamma$ plays a role in ``reabsorbing'' the negative part of the energy coming
from the operator, see
Propositions~\ref{assorb measure} and~\ref{assorb measureBIS}.
We also remark that this reabsorbing property is different from the one provided in~\cite[Proposition~2.3]{DPSV},
since it is unavoidable in our setting that the~$L^2$-norm of the function pops up in the right-hand side
of the inequality, and this requires an additional step to reabsorb all the negative terms. 

In the forthcoming Proposition~\ref{qwedfrgthbncmx98ty8yi4erwyt8} we will also establish that there exists a constant sign solution
to~\eqref{problem} under the stronger assumption that~$\mu^-=0$.
We mention here that the assumption~$\mu^-\equiv0$ is technical,
and we will further investigate the possibility of finding a solution with constant sign without this stronger condition (but still assuming some sort of ``smallness'' of the parameter~$\gamma$) in a future work.

This provides a motivation to investigate qualitative properties of non-negative solutions. In this direction, we provide
a preliminary step towards regularity, namely we show boundedness of non-negative solutions in the case in which~$\mu^-=0$.

\begin{theorem}\label{THM:LINFTY} Assume that~$\mu^-=0$.
Suppose that there exists~$\overline{s}_\sharp\in[s_\sharp, 1]$ such that
\begin{equation}\label{ass9567daskjytrTRYJ}\mu([s_\sharp,\overline{s}_\sharp])>0.\end{equation}

Let $u\in H^{\mu}(\R^N)$ be a non-negative solution of \eqref{problem}. Then $u \in L^{\infty}(\R^N)$. 
\end{theorem}

We point out that the space~$\H$ appearing in Theorem~\ref{THM:LINFTY}
is the energy space associated with the problem in~\eqref{problem}
and will be defined here below in formula~\eqref{spazioenerg}.

We remark that the assumption in~\eqref{ass9567daskjytrTRYJ} is satisfied
if the operator contains the Dirac's delta at~$s_\sharp$, and in this case~$\mu(\{s_\sharp\})>0$ and the proof of Theorem~\ref{THM:LINFTY} would be more similar to the case of the (possibly fractional) Laplacian, see footnote~\ref{foriuermnv876543}.

The advantage of~\eqref{ass9567daskjytrTRYJ} is that it allows us to consider also the case in which~$\mu(\{s_\sharp\})=0$ (for example when~$\mu$ is the Lebesgue measure), provided that we can find some ``mass'' near~$s_\sharp$. In this general setting, the proof of Theorem~\ref{THM:LINFTY} is not a mere adaptation
of the classical iterative argument (see e.g.~\cite[Proposition~2.2]{Barrios} or~\cite[Proposition~5.1.1]{DMV} for the proof of boundedness  of solutions in the fractional case), but it requires a more careful use of the Sobolev embedding and a delicate treatment of the integrals involved in the estimates.  

En passant, we also mention that we provide a detailed proof of a Kato-type inequality (in the weak sense)
in the general setting treated in this paper, see Proposition~\ref{lem:kato} below. 

Finally, we provide a general Pohozaev identity for solutions of~\eqref{problem}, provided they possess enough ``regularity''.

 \begin{theorem}\label{pohozaev theorem}
 Assume that~$\mu_-=0$. 
Let~$u \in \H $ be such that, for all~$R>0$,
\begin{equation*}
\iiint_{[0,1]\times B_R\times\R^N} C_{N,s}\frac{|2u(x)-u(x+y)-u(x-y)|}{|y|^{N+2s}}\,d\mu(s)\,dx\,dy<+\infty.
\end{equation*}
Suppose also that~$u\in W^{1,p}_{loc}(\R^N)$ for some~$p\in(1,+\infty)$ and that~$
\mathcal{L}_{\mu}u\in L^q_{loc}(\R^N)$,
being~$q$ the conjugate exponent of~$p$, namely~$\frac1p+\frac1q=1$.

 Suppose that  
 \begin{equation*}    \mathcal{L}_\mu u + \beta u = (I_\alpha \ast F(u))f(u) \,\,\, \textit{a.e. in}\,\,\, \R^N.
\end{equation*}

Then $u$ satisfies the Pohozaev identity \begin{eqnarray*}&&
\int_{[0,1]}\left( \frac{(N-2s)C_{N,s}}{2}\iint_{\R^{2N}}\frac{(u(x)-u(y))^2}{|x-y|^{N+2s}}\,dx\,dy\right)\, d\mu(s)  
+\frac{\beta N}{2}\int_{\R^N}{u}^2(x)\,dx\\&&\qquad\qquad=\frac{N+\alpha}{2}\int_{\R^N} (I_\alpha\ast F( u))F( u)\,dx.
\end{eqnarray*}
\end{theorem}

We point out that the regularity assumptions on the solution~$u$ in Theorem~\ref{pohozaev theorem} are instrumental to make sense
of the quantities involved in the Pohozaev identity and in specific cases can be guaranteed by a suitable fractional elliptic regularity theory
depending on the specific problem considered, see e.g.~\cite{BDVV22, SVWZ, SVWZ2}.

\medskip

The rest of the paper is organized as follows. In Section \ref{Preliminaries}, we introduce the functional framework, defining the Hilbert space associated with the superposition operator $\mathcal{L}_\mu$ and recalling the fundamental properties of the Riesz potential and the Hardy-Littlewood-Sobolev inequality. 

Section \ref{Variational Setting and existence result} deals with the variational framework of \eqref{problem}, focusing on the Mountain Pass geometry and the existence of ground state solutions via radial compactness. Namely, here we prove the existence statement in Theorem~\ref{MP solution}.
Moreover, we establish the equivalence between the Mountain Pass level and the least energy level attained on the Pohozaev manifold, thus proving Theorem~\ref{p=l}.

In Section \ref{Regularity}, we address the global $L^\infty$ bound of Theorem~\ref{THM:LINFTY}.
Finally, Section \ref{sec_Pohozaev} is concerned with the derivation of the general Pohozaev identity for the operator $\mathcal{L}_\mu$, as stated in Theorem~\ref{pohozaev theorem}.

\section{Preliminaries}\label{Preliminaries}

Let $N\geq 3$  and $s\in (0,1)$. We consider the fractional Sobolev space
\begin{equation*}
    H^{s}(\R^N)=\{ u \in L^2(\R^N)\, | \, (-\Delta)^{\frac{s}{2}} u \in L^2(\R^N) \}.
\end{equation*}

In the rest of the paper, we use the short notation~$\|\cdot\|_p:=\|\cdot\|_{L^p(\R^N)}$.

We point out that the following relation with the Gagliardo seminorm holds (see e.g.~\cite[Proposition~3.6]{DPV}):
 \begin{equation}
    \label{seminorm_s}
    \|(-\Delta)^{\frac{s}{2}}u\|_{2}=[u]_{s}:=\left(C_{N,s} \iint_{\R^{2N}} \frac{|u(x)-u(y)|^2}{|x-y|^{N+2s}}\, dx\,dy\right)^{\frac{1}{2}}.
\end{equation} 

Also, we recall the continuous embedding (see~\cite[Theorem~6.5]{DPV}) 
\begin{equation}\label{embedding cont Hs}
    H^{s}(\R^N) \hookrightarrow L^p(\R^N) \qquad{\mbox{for every }}p\in [2, 2^*_{s}],
\end{equation}
where
$$2^*_{s}:=\frac{2N}{N-2s}$$ is the fractional critical Sobolev exponent, and the compact embedding (see~\cite{Lions}) \begin{equation}\label{compact embedding Hs}
    H^{s}_r(\R^N)\hookrightarrow \hookrightarrow L^p(\R^N)\qquad
    {\mbox{for every }}p\in (2,2_{s}^*),
\end{equation}  where~$H^{s}_r(\R^N)$ denotes
the space of radially symmetric functions in~$H^{s}(\R^N)$.

In addition, we have that
 \begin{equation}\label{embendind homogeneous space}
   \|u\|_{2_s^*}\leq C \|(-\Delta)^{\frac{s}{2}} u\|_{2},
\end{equation} for some $C>0$, depending on~$N$ and~$s$
(see~\cite[Theorem 6.5]{DPV}).

Now we introduce the functional space 
needed to deal with our setting as the closure of the space of smooth and compactly supported functions
in~$\R^N$ with respect to the norm
\begin{equation*}
    \|u\|_{\mu}:=\left(\int_{[0,1]} \|(-\Delta)^{\frac{s}{2}}u\|_{2}^2\, d\mu^+(s)+ \|u\|^2_{2}\right)^{\frac12}.
\end{equation*}
We denote this space by~$\H$ and we remark that it can also be written as
 \begin{equation}\label{spazioenerg}
    \H=\left\{u \in L^2(\R^N)\,\, \Big|\,\, \int_{[0,1]} (-\Delta)^{\frac{s}{2}}u\, d\mu^+(s) \in L^2(\R^N)\right\}
\end{equation}
(see, for instance~\cite[Theorem~2.1]{DPLSV}).

The space~$\H$ is continuously embedded in $H^{s_\sharp}(\R^N)$, where $s_\sharp$ is as in~\eqref{positivity measure 4}, according to the following statement (see also~\cite[Proposition~2.4]{DPSV}
for a similar result in the setting of functions in the energy space satisfying a zero Dirichlet condition outside
a given bounded domain~$\Omega$):

\begin{proposition}\label{Prop H mu in H s}
Let~$s_\sharp$ be as in~\eqref{positivity measure 4}.

Then, for any $u \in \H$,    \begin{equation*}
   [u]^2_{s_\sharp}\leq \|u\|_2^2+ \frac1{\mu^+([s_\sharp,1])}
        \int_{[0,1]} [u]^2_s\, d\mu^+(s).
    \end{equation*}
\end{proposition}

The proof of Proposition~\ref{Prop H mu in H s} is based on a simple observation, that we provide here in details since it will be used multiple times in the rest of this paper:

\begin{lemma}\label{fractional immersion}
Let $s_1$, $s_2 \in [0, 1]$ be s.t. $0\leq s_1\leq s_2\leq 1$ and $u \in H^{s_2}(\R^N)$.
Then,
\begin{equation*}
     [u]^2_{s_1}\leq \|u\|_2^2 + [u]^2_{s_2}.
\end{equation*}
\end{lemma}

\begin{proof}
By density, it is enough to perform the proof for~$u \in C^\infty_0(\R^N)$. 
Thanks to~\cite[Proposition~3.4]{DPV}), we can write
\begin{equation*}\begin{split}
     [u]^2_{s_1}=&\int_{\R^N} |2\pi \xi|^{2s_1} |\hat u(\xi)|^2\, d\xi\\   
     =&\int_{B_{\frac{1}{2\pi}} } |2\pi \xi|^{2s_1} |\hat u(\xi)|^2d\xi + \int_{\R^N\smallsetminus B_{\frac{1}{2\pi}} } |2\pi \xi|^{2s_1} |\hat u(\xi)|^2\, d\xi\\
     \leq &\int_{\R^N} |\hat u(\xi)|^2\, d\xi+ \int_{\R^N} |2\pi \xi|^{2s_2} |\hat u(\xi)|^2\, d\xi\\ 
     =&\|u\|_2^2 + [u]^2_{s_2},
\end{split}\end{equation*} as desired.
\end{proof}

\begin{proof}[Proof of Proposition~\ref{Prop H mu in H s}]
Using Lemma~\ref{fractional immersion} with~$s_1:=s_\sharp$ ad $s_2:=s \in [s_\sharp, 1]$, we find that
$$  [u]^2_{s_\sharp}\le \|u\|_2^2 + [u]^2_{s}.$$
Thus, integrating in~$[s_\sharp, 1]$, we obtain that
\begin{eqnarray*}
\mu^+([s_\sharp, 1])[u]^2_{s_\sharp} &\leq&  \mu^+([s_\sharp,1])\|u\|_2^2+\int_{[s_\sharp,1]}[u]^2_s\, d\mu^+(s)\\ &\leq& \mu^+([s_\sharp,1])\|u\|_2^2+\int_{[0,1]}[u]^2_s\, d\mu^+(s),
\end{eqnarray*} 
from which the desired inequality follows.
\end{proof}

By Proposition~\ref{Prop H mu in H s} and the embeddings in~\eqref{embedding cont Hs}
and~\eqref{compact embedding Hs}, we deduce that
\begin{equation}\label{aggcompatto00}
\H \hookrightarrow L^p(\R^N) \qquad {\mbox{for every }}p \in [2,2^*_{s_\sharp}]\end{equation} and 
\begin{equation}\label{aggcompatto}
H^\mu_r(\R^N) \hookrightarrow\hookrightarrow L^p(\R^N)\qquad{\mbox{for every }}p \in (2,2^*_{s_\sharp}).
\end{equation}

Finally, we recall the following standard estimates for the Riesz pontential (see~\cite[Theorem~4.3]{LiebLoss}).

\begin{proposition}[Hardy-Littlewood-Sobolev inequality] \label{prop:HLS}
Let $\alpha \in(0,N)$, and let $r$, $q \in (1, +\infty)$ be such that $\frac{1}{r}-\frac{1}{q}=\frac{\alpha}{N}$. Then, the map
\begin{equation*}
    f \in L^r(\R^N) \mapsto I_\alpha * f \in L^q(\R^N)
\end{equation*} is continuous. In particular, if $r$, $t \in (1, +\infty)$ verify $\frac{1}{r}+\frac{1}{t}=\frac{N+\alpha}{N}$, then there exists a constant $C=C(N, \alpha, r, t)>0$ such that \begin{equation*}
    \left|\int_{\R^N} (I_\alpha * g)h\, dx\right|\leq C\|g\|_{r}\|h\|_{t}
\end{equation*} for all $g \in L^{r}(\R^N)$ and $h \in L^t(\R^N)$.
\end{proposition}

Proposition~\ref{prop:HLS} is useful because it will guarantee that the functional associated with our problem
in~\eqref{problem} are ``well-defined'' in the functional setting introduced here.
To check this fact, we provide the following observation:

\begin{lemma}\label{regularity of f}
 Let $u \in H^{s_\sharp}(\mathbb{R}^N)$ and assume $(F_2)$. Then \begin{align*}
    f(u) \in L^{\frac{2N}{\alpha}}(\mathbb{R}^N) + L^{\frac{2N}{\alpha+2s_\sharp}}(\mathbb{R}^N),\,\,\,\,\,  F(u)\in L^{\frac{2N}{N+\alpha}}(\mathbb{R}^N),
\end{align*}
\begin{align*}
    I_{\alpha} * F(u) \in L^2(\R^N)+ L^{\frac{2N}{N+2s_\sharp}}(\R^N),\,\,\,\,\,\,\, (I_\alpha * F(u))f(u) \in L^2(\mathbb{R}^N) + L^{\frac{2N}{N+2s_\sharp}}(\mathbb{R}^N).
\end{align*}
 \end{lemma}
 
 \begin{proof}
     Let $u \in H^{s_\sharp}(\mathbb{R}^N)\subset L^2(\mathbb{R}^N)\cap L^{2^*_{s_\sharp}}(\mathbb{R}^N)$.
     By $(F_2)$, we have
     \begin{align*}
    f(u) &\in \left(L^{\frac{2N}{\alpha}}(\mathbb{R}^N)\cap L^{\frac{N}{\alpha}\frac{2N}{N-2s_\sharp}}(\mathbb{R}^N) \right)+ \left(L^{2\frac{N-2s_\sharp}{\alpha+2s_\sharp}}(\mathbb{R}^N)\cap L^{\frac{2N}{\alpha+2s_\sharp}}(\mathbb{R}^N)\right)\\ &\subset L^{\frac{2N}{\alpha}}(\mathbb{R}^N) + L^{\frac{2N}{\alpha+2s_\sharp}}(\mathbb{R}^N), 
\end{align*} and
\begin{align*}
    F(u)&\in \left(L^{\frac{2N}{N+\alpha}}(\mathbb{R}^N)\cap L^{\frac{N}{N+\alpha}\frac{2N}{N-2s_\sharp}}(\mathbb{R}^N)\right)
    +\left( L^{2\frac{N-2s_\sharp}{N+\alpha}}(\mathbb{R}^N) \cap L^{\frac{2N}{N+\alpha}}(\mathbb{R}^N)\right)\\ & \subset L^{\frac{2N}{N+\alpha}}(\mathbb{R}^N).
\end{align*}Thus, by the Hardy-Littlewood-Sobolev inequality in Proposition~\ref{prop:HLS}, we obtain that
\begin{align*}
    I_{\alpha} * F(u) & \in \left(
    L^{\frac{2N}{N-\alpha}}(\mathbb{R}^N)\cap L^{\frac{2N^2}{N^2-(\alpha+2s_\sharp)N-2s_\sharp \alpha}}(\mathbb{R}^N)\right)
    +\left(L^{\frac{2N(N-2s)}{N^2-\alpha N +4 s_\sharp\alpha}}(\mathbb{R}^N)\cap L^{\frac{2N}{N-\alpha}}(\mathbb{R}^N)\right)\\ &\subset  L^2(\R^N)+ L^{\frac{2N}{N+2s_\sharp}}(\R^N).
\end{align*}Then, by the H\"older inequality,
\begin{align*}
    (I_\alpha * F(u))f(u) &\in \left(
    L^2(\mathbb{R}^N) \cap L^{\frac{2N^2}{N^2 - 2s_\sharp\alpha}}(\mathbb{R}^N)\right) 
    + \left(L^{\frac{2N(N-2s_\sharp)}{N^2 + 2s_\sharp\alpha}}(\mathbb{R}^N) \cap L^{\frac{2N}{N+2s_\sharp}}(\mathbb{R}^N)\right)\\ & \subset L^2(\mathbb{R}^N) + L^{\frac{2N}{N+2s_\sharp}}(\mathbb{R}^N).
\end{align*}
These considerations establish Lemma~\ref{regularity of f}.
\end{proof}
 
 \begin{remark}
{\rm{We note that in general $(I_\alpha * F(u))f(u)$ does not lie in $L^2(\R^N)$. On the other hand, if~$\varphi \in H^{s_\sharp}(\R^N)\subset L^2(\R^N) \cap L^{2^*_{s_\sharp}}(\R^N)$, the summability property of $(I_\alpha * F(u))f(u)$ 
determined by Lemma~\ref{regularity of f}, gives that~$(I_\alpha * F(u))f(u)\varphi\in L^1(\R^N)$.
}}\end{remark}
 
\section{Variational setting and proofs of Theorems \ref{MP solution}
and~\ref{p=l}}\label{Variational Setting and existence result}

\subsection{Variational setting}
In this section we search for solution of the problem~\eqref{problem} via variational techniques restricted to the space of radially symmetric functions $$\M=\{ u \in \H\, |\, u(x):=u(|x|)\, \textit{a.e.}\, x \in \R^N     \}.$$
For this, we will need to cast our problem into the appropriate variational setting.

The first statement will be used to deal with the terms in the energy functional related to the negative part of the measure, namely we would like to ``reabsorb'' this term in the one coming from the positive part in order to be able
to check that the Mountain Pass structure of the functional is fulfilled. An analogous statement is contained in~\cite[Proposition~2.3]{DPSV}).
 
\begin{proposition}\label{assorb measure}
  Let~$\bar{s}$ be as in~\eqref{positivity property 1}, \eqref{positivity property 2}, and~\eqref{positivity property 3}. 
  
Then, for any $u\in \H$,
$$        \int_{[0,\bar s]} [u]_s^2 \, d\mu^-(s) \leq  \gamma \left( 
        2 \mu^+([\bar s, 1]) \|u\|^2_2+ \int_{[\bar s,1]} [u]_s^2\, d\mu^+(s)\right).
   $$\end{proposition}

\begin{proof}
We exploit Lemma~\ref{fractional immersion} with $s_1:=s\in [0,\bar s]$ and~$s_2:=\bar s$ and we obtain that~$
    [u]^2_{s} \leq \|u\|_2^2+ [u]^2_{\bar s}$.
     Thus, integrating over~$[0,\bar s]$ with respect to~$\mu^-$,
     \begin{equation}\label{bvcnkxiwqou8tr}
     \int_{[0,\bar s]} [u]^2_{s}\, d\mu^-(s)\leq \mu^-([0,\bar s])\big(\|u \|_2^2+ [u]^2_{\bar s}\big).
     \end{equation}

Moreover, using Lemma~\ref{fractional immersion}
with~$s_1:=\bar s$ and~$s_2:=s\in [\bar s,1]$ we obtain that~$ [u]^2_{\bar s}\leq \|u\|_2^2+ [u]^2_{ s}$, and so,
integrating over~$ [\bar s,1]$ with respect to~$\mu^+$,
\begin{equation}\label{bvcnkxiwqou8tr2}
\mu^+([\bar s,1]) [u]^2_{\bar s}\leq \mu^+([\bar s,1]) \|u\|_2^2
+\int_{[\bar s,1]} [u]^2_{ s}\,\mu^+(s).
\end{equation}

From~\eqref{bvcnkxiwqou8tr} and~\eqref{bvcnkxiwqou8tr2},
and recalling also~\eqref{positivity property 2} and~\eqref{positivity property 3}, we conclude that
     \begin{eqnarray*}
        & &\int_{[0,\bar s]} [u]^2_{s}\, d\mu^-(s)\leq  \mu^-([0,\bar s])\big(\|u \|_2^2+ [u]^2_{\bar s}\big)
      \leq\gamma  \mu^+([\bar s, 1])\big(\|u\|_2^2+[u]^2_{\bar s}\big)\\
       &&\qquad \leq\gamma\mu^+([\bar s, 1])\|u\|_2^2+ \gamma \mu^+([\bar s, 1]) \|u\|^2_2+\gamma \int_{[\bar s,1]} [u]_s^2\, d\mu^+(s)\\ &&\qquad\le 2 \gamma\mu^+([\bar s, 1]) \|u\|^2_2+\gamma\int_{[\bar s,1]} [u]_s^2\, d\mu^+(s),
     \end{eqnarray*} as desired.
\end{proof}

For our purposes, we will also need a variant of Proposition~\ref{assorb measure}, that takes care of the ``reabsorbing'' property for the Pohozaev functional (that will be defined here below in formula~\eqref{Pohozaev functional}).

\begin{proposition}
\label{assorb measureBIS}
  Let~$\bar{s}$ be as in~\eqref{positivity property 1}, \eqref{positivity property 2}, and~\eqref{positivity property 3}. 
  
Then, there exists a constant~$\bar{c}>0$, depending on~$N$ and~$\mu^+([\bar s,1])$, such that,
for any $u\in \H$,
$$        \int_{[0,\bar s]}\frac{N-2s}2 [u]_s^2 \, d\mu^-(s) \leq  \gamma \bar{c}\left( 
 \|u\|^2_2+ \int_{[\bar s, 1]} \frac{N-2s}2 [u]_s^2\, d\mu^+(s)\right).
   $$\end{proposition}

\begin{proof}
We exploit Lemma~\ref{fractional immersion} with $s_1:=s\in [0,\bar s]$ and~$s_2:=\bar s$ and we obtain that~$
    [u]^2_{s} \leq \|u\|_2^2+ [u]^2_{\bar s}$.
     Thus, multiplying by~$(N-2s)/2$ and integrating over~$[0,\bar s]$ with respect to~$\mu^-$, we find that
     \begin{equation}\label{bvcnkxiwqou8tr3}\begin{split}
     \int_{[0,\bar s]}\frac{N-2s}2 [u]^2_{s}\, d\mu^-(s)&\leq
        \int_{[0,\bar s]}\frac{N-2s}2\, d\mu^-(s)  \big(\|u\|_2^2+ [u]^2_{\bar s}\big)\\
     &\leq \frac{N}2 \mu^-([0,\bar s])\big(\|u \|_2^2+ [u]^2_{\bar s}\big).
     \end{split}\end{equation}

Moreover, using Lemma~\ref{fractional immersion}
with~$s_1:=\bar s$ and~$s_2:=s\in [\bar s,1]$ we obtain that~$ [u]^2_{\bar s}\leq \|u\|_2^2+ [u]^2_{ s}$, and so,
multiplying by~$(N-2s)/2$ and
integrating over~$ [\bar s,1]$ with respect to~$\mu^+$,
\begin{equation}\label{bvcnkxiwqou8tr4}\begin{split}
\frac{N-2}2\mu^+([\bar s,1]) [u]^2_{\bar s}&\leq
\int_{[\bar s,1]}\frac{N-2s}2\, d\mu^+(s) [u]^2_{\bar s}\\
&\leq \int_{[\bar s,1]}\frac{N-2s}2 \big( \|u\|_2^2+ [u]^2_{ s}\big)\, d\mu^+(s)
\\&\leq \frac{N}2\mu^+([\bar s,1])\|u\|_2^2+ \int_{[\bar s,1]}\frac{N-2s}2[u]^2_{ s}\, d\mu^+(s)
.
\end{split}\end{equation}

From~\eqref{bvcnkxiwqou8tr3} and~\eqref{bvcnkxiwqou8tr4},
and recalling also~\eqref{positivity property 2} and~\eqref{positivity property 3}, we conclude that
     \begin{eqnarray*}
&&\int_{[0,\bar s]}\frac{N-2s}2 [u]^2_{s}\, d\mu^-(s)\leq \frac{N}2 \mu^-([0,\bar s])\big(\|u \|_2^2+ [u]^2_{\bar s}\big)\\&&\qquad
      \leq\frac{\gamma N}2  \mu^+([\bar s, 1])\big(\|u\|_2^2+[u]^2_{\bar s}\big)\\
&&\qquad\leq
\frac{\gamma N}2\mu^+([\bar s, 1])\|u\|_2^2+
\frac{\gamma N}{N-2}\left(\frac{N}2\mu^+([\bar s,1])\|u\|_2^2+ \int_{[\bar s,1]}\frac{N-2s}2[u]^2_{ s}\, d\mu^+(s)\right)
\\ &&\qquad= \gamma
\left(\frac{ N}2+\frac{  N^2}{2(N-2)}\right)\mu^+([\bar s,1])\|u\|_2^2
+
\frac{\gamma N}{N-2} \int_{[\bar s,1]}\frac{N-2s}2[u]^2_{ s}\, d\mu^+(s) ,
     \end{eqnarray*} as desired.
\end{proof}

 Now we introduce the definition of a weak solution of \eqref{problem}.
 
\begin{definition}\label{weak solution}
We say that $u \in H_r^{\mu}(\mathbb{R}^N)$ is a weak solution
of~\eqref{problem} if, for all $v \in H_r^{\mu}(\mathbb{R}^N)$,
    \begin{align*}
     \int_{[0,1]} &\Big( C_{N,s} \iint_{\mathbb{R}^{2N}}\frac{(u(x)-u(y))(v(x)-v(y))}{|x-y|^{N+2s}}dx\,dy\Big)d\mu^+(s)\\ &-\int_{[0, \bar s]} \Big(C_{N,s}\iint_{\mathbb{R}^{2N}}\frac{(u(x)-u(y))(v(x)-v(y))}{|x-y|^{N+2s}}dx\,dy\Big)d\mu^-(s)\\&= -\beta \int_{\mathbb{R}^N}u\, v\, dx +\int_{\mathbb{R}^N} (I_{\alpha} * F(u))f(u) v\, dx.
    \end{align*}
\end{definition}

The variational functional associated with~\eqref{problem} is
$\mathcal{J}_\beta: H_r^{\mu}(\mathbb{R}^N)\to \mathbb{R}$ defined as
\begin{equation}    \label{functional associated to the problem}
    \mathcal J_{\beta}(u):= \frac{1}{2}\int_{[0,1]} \|(-\Delta)^{\frac{s}{2}}u\|^2_2\, d\mu^+(s)-\frac{1}{2}\int_{[0,\bar s]} \|(-\Delta)^{\frac{s}{2}}u\|^2_2\, d\mu^-(s)  + \frac{\beta}{2}\|u\|^2_2- \frac{1}{2}\mathcal{D}(u),
\end{equation}
where \begin{equation}\label{chewiut9743767676u}
\mathcal{D}(u):=\int_{\R^N} (I_\alpha * F(u))F(u)\, dx.\end{equation}
We point out that
the functional $\mathcal J_\beta$ is of class $C^1$
and its critical points correspond to weak solutions of~\eqref{problem} according to Definition~\ref{weak solution}.

   
   \subsection{The Mountain Pass structure}

Aim of this subsection is to check that the functional~$\mathcal{J}_\beta$
in~\eqref{functional associated to the problem} has a Mountain Pass structure.

For this, we provide an estimate for the term~$\mathcal{D}(u)$ in~\eqref{chewiut9743767676u}.

\begin{lemma}\label{bjkewyr8392659843659jnhbgf}
For all~$u\in H_r^{\mu}(\mathbb{R}^N)$, we have that
$$ |\mathcal{D}(u)|\le C\left(\|u\|_\mu^{2\frac{N+\alpha}N}
+\|u\|_{\mu}^{2\frac{N+\alpha}{N-2s_\sharp}}\right)
,$$ for some~$\widetilde C>0$, depending on~$N$, $\alpha$, $s_\sharp$, and the constant~$C$ appearing in~$(F_2)$.
\end{lemma}

\begin{proof}
From~$(F_2)$ we know that
$$ |F(u)|\le C \left( |u|^{\frac{N+\alpha}{N}}+
|u|^{\frac{N+\alpha}{N-2s_\sharp}}\right).$$
Therefore, using the Hardy-Littlewood-Sobolev inequality in Proposition~\ref{prop:HLS}, we see that
\begin{eqnarray*}
|\mathcal{D}(u)|=\left|\int_{\R^N} (I_\alpha * F(u))F(u)\, dx\right|\le \widetilde C\left( \|u\|_2^{2\frac{N+\alpha}N}
+\|u\|_{2^*_{s_\sharp}}^{2\frac{N+\alpha}{N-2s_\sharp}}
\right),
\end{eqnarray*}
for some~$\widetilde C>0$, depending on~$N$, $\alpha$, $s_\sharp$, and~$C$.

From this and~\eqref{aggcompatto00}, we obtain the desired estimate.
\end{proof}

This observation is useful to check that the functional is positive for small values of the norm, as pointed 
out in the next result:

\begin{lemma}\label{bjhkewr793798436780jfs7y6t5relk0}
There exist~$\gamma_0>0$ (depending on~$\beta$, $\mu^+$, and~$\bar s$), 
$\rho\in(0,1]$ and $c>0$ (depending on~$N$, $\alpha$, $\beta$,
$s_\sharp$, and the constant~$C$ appearing in~$(F_2)$)
such that for all~$\gamma\in[0,\gamma_0]$
we have that~$\mathcal J_{\beta}(u)\geq c$ for any~$u\in H_r^{\mu}(\mathbb{R}^N)$
with~$\|u\|_{\mu}=\rho$.
\end{lemma}

\begin{proof}
By Proposition~\ref{assorb measure}, we have that  \begin{align*}
    \mathcal J_\beta(u)\geq
 \frac{1-\gamma}{2}\int_{[0,1]} \|(-\Delta)^{\frac{s}{2}}u\|^2_2\, d\mu^+(s)-\gamma\mu^+([\bar s, 1]) \|u\|^2_2
+ \frac{\beta}{2} \|u\|_{2}^2 -\frac{1}{2}\mathcal{D}(u).
    \end{align*} 
Hence, by taking~$\gamma  $ sufficiently small, possibly in dependence of~$\beta$, $\mu^+$ and~$\bar s$,
we conclude that
   \begin{eqnarray*} \mathcal J_\beta(u)&\geq& 
   \frac{1}4 \int_{[0,1]}   \|(-\Delta)^{\frac{s}{2}} u\|_2^2\, d\mu^+(s)
   + \frac{\beta }4\|u\|^2_2 -\frac{1}{2}\mathcal{D}(u)\\&\geq&
   \frac14 \min\{1,\beta\}   \|u\|_{\mu}^2
   -\frac{1}{2}\mathcal{D}(u)
    .\end{eqnarray*}
    
    Thus, using the estimate in Lemma~\ref{bjkewyr8392659843659jnhbgf}, if~$\|u\|_\mu\le1$,
    \begin{eqnarray*}
    \mathcal J_\beta(u)\geq  \frac14 \min\{1,\beta\}   \|u\|_{\mu}^2- C \|u\|_\mu^{2+\frac{2\alpha}N}=
   \|u\|_{\mu}^2\left(  \frac14 \min\{1,\beta\}  - C \|u\|_\mu^{\frac{2\alpha}N}\right).
    \end{eqnarray*}
    Now we pick
    $$ \rho:=\min\left\{\left(\frac1{8C} \min\{1,\beta\}\right)^{\frac{N}{2\alpha}}, 1\right\} $$
    and we find that, if~$\|u\|_\mu=\rho$,
    $$  \mathcal J_\beta(u)\geq\frac18 \min\{1,\beta\}\rho^2.$$
    Thus, choosing $$c:=\frac18 \min\{1,\beta\}\rho^2,$$ we obtain the desired result.
\end{proof}


\begin{lemma}\label{bjhkewr793798436780jfs7y6t5relk}
There exists a function~$v\in H_r^{\mu}(\mathbb{R}^N)\setminus\{0\}$ such that~$\|v\|_\mu\geq 2$,
$$ \mathcal{D}(v)>0\qquad{\mbox{and}}\qquad \mathcal J_{\beta}(v)<0.$$
\end{lemma}

\begin{proof}
Let~$t_0$ be as in~$(F_4)$ and\footnote{As customary, for any set~$A\subseteq\R^N$, we use the notation $$\chi_A(x):=\begin{cases}
1, {\mbox{ if }} x\in A,\\
0, {\mbox{ if }} x\in\R^N\smallsetminus  A.
\end{cases}$$} take~$\widetilde w:=t_0\chi_{B_1}$. In this way, we see that
\begin{eqnarray*}&&\int_{\R^N}  (I_\alpha \ast F(\widetilde w)(x))F(\widetilde w(x))\,dx=
\iint_{\R^{2N}}  I_\alpha(x-y) F(\widetilde w(y))F(\widetilde w(x))\,dx\,dy\\&&\qquad
=(F(t_0))^2\iint_{B_1\times B_1}  I_\alpha(x-y)\,dx\,dy>0
.
\end{eqnarray*} By~$(F_2)$, the first term in the above formula is continuous in~$L^2(\R^N)\cap L^{\frac{2N}{N-2s_\sharp}}(\R^N)$, therefore, by density,
there exists a function~$w\in H_r^{\mu}(\mathbb{R}^N)$ such that
$$ \mathcal{D}(w)=\int_{\R^N}  (I_\alpha \ast F(w)(x))F(w(x))\,dx>0.$$

Now, for any~$t\ge0$, we consider the function~$\psi(t)\in H_r^{\mu}(\mathbb{R}^N)$ defined as
\begin{equation*} \psi(t):=\begin{cases}
w\left(\frac{\cdot}{t}\right), &{\mbox{ if }} t>0,\\
0, &{\mbox{ if }} t=0\end{cases}\end{equation*}
and we have that
\begin{align*}
    \mathcal J_\beta(\psi(t))&\leq 
    \frac{1}{2} \int_{[0,1]} \|(-\Delta)^{\frac{s}{2}} \psi(t)\|_{2}^2\, d\mu^+(s)+ \frac{\beta}{2} \|\psi(t)\|_{2}^2-\frac{1}{2}\mathcal{D}(\psi(t))\\ &= \int_{[0,1]} \frac{t^{N-2s}}{2} \|(-\Delta)^{\frac{s}{2}} w\|_{2}^2\, d\mu^+(s)+ \frac{\beta t^N}{2} \|w\|_{2}^2-\frac{t^{N+\alpha}}{2}\mathcal{D}(w)\\
     &=\frac{t^{N+\alpha}}{2} \left( \int_{[0,1]} t^{-2s-\alpha} \|(-\Delta)^{\frac{s}{2}} w\|_{2}^2\, d\mu^+(s)+ \beta t^{-\alpha}\|w\|_{2}^2-\mathcal{D}(w)\right).
\end{align*}
Therefore, if $t\geq 1$, we infer that
\begin{equation}\label{bncxiwi5uy98430thewsd}
    \mathcal J_\beta(\psi(t))
    \leq \frac{t^{N+\alpha}}{2} \left[ t^{-\alpha}\left(\int_{[0,1]} \|(-\Delta)^{\frac{s}{2}} w\|_{2}^2\, d\mu^+(s)+ \beta \|w\|_{2}^2\right)-\mathcal{D}(w)\right].\end{equation}
    
We now pick
$$ t_\star:=\max \left\{\left(\frac{\displaystyle 2\int_{[0,1]} \|(-\Delta)^{\frac{s}{2}} w\|_2^2\, d\mu^+(s)+2 \beta \|w\|_2^2}{D(w)}\right)^{\frac1{\alpha}},
\left(\frac{4}{\|w\|^2_\mu}\right)^{\frac1{N-2}}\right\} +1$$ and we conclude that~$\mathcal J_\beta(\psi(t_\star)) < 0$.

Moreover,
\begin{eqnarray*}
\|\psi(t_\star)\|_\mu^2&=& \int_{[0,1]} \|(-\Delta)^{\frac{s}{2}}\psi(t_\star)\|_{2}^2\, d\mu^+(s)+ \|\psi(t_\star)\|^2_{2}\\&
=&  \int_{[0,1]} t_\star^{N-2s} \|(-\Delta)^{\frac{s}{2}} w\|_{2}^2\, d\mu^+(s)+  t^N_\star \|w\|_{2}^2\\&\geq&
t_\star^{N-2}
\left( \int_{[0,1]} \|(-\Delta)^{\frac{s}{2}} w\|_{2}^2\, d\mu^+(s)+  t^2_\star \|w\|_{2}^2\right)\\&\geq&
t_\star^{N-2}\|w\|_\mu^2\\&\geq &4.
\end{eqnarray*}
The proof is thereby complete by taking~$v:=\psi(t_\star)$.
\end{proof}

We now introduce the set of paths
\begin{equation}\label{cjkwhoufguo7654}
    \Gamma_\beta:=\Big\{\psi \in C([0,1], H_r^{\mu}(\mathbb{R}^N))\,\, |\,\,  \psi(0)=0,\, \, \, \,\mathcal{J}_\beta(\psi(1))<0\Big\}
\end{equation}
and we point out that~$\Gamma_\beta\neq\varnothing$
thanks to Lemma~\ref{bjhkewr793798436780jfs7y6t5relk}.

We define the Mountain Pass value as\begin{equation}\label{mplevel00}
    l(\beta):= \inf\limits_{\psi\in \Gamma_\beta}\max\limits_{t\in[0,1]} \mathcal J_{\beta}(\psi(t)).
\end{equation}

We also introduce the Pohozaev functional $\mathcal{P}_\beta : H_r^{\mu}(\mathbb{R}^N)\to \mathbb{R}$ defined as
\begin{equation}
    \label{Pohozaev functional}
    \begin{split}
    \mathcal{P}_\beta(u):=&\int_{[0,1]} \frac{N-2s}{2} \|(-\Delta)^{\frac{s}{2}}u \|^2_2\, d\mu^+(s)-\int_{[0,1]} \frac{N-2s}{2} \|(-\Delta)^{\frac{s}{2}}u \|^2_2\, d\mu^-(s)\\
    &\qquad+\frac{\beta N}{2}\|u\|^2_{2}  -\frac{N+\alpha}{2}\mathcal{D}(u)
 \end{split}
\end{equation}
and we define
the least energy of $\mathcal J_{\beta}$ on the Pohozaev set as \begin{equation}\label{poholevel00}
    p(\beta):=\inf\big\{\mathcal J_\beta(u)\,\, | \,\, u\in H_r^{\mu}(\mathbb{R}^N)\setminus\{0\},\,\, \mathcal{P}_\beta(u)=0\big\}
\end{equation}We now  check that our functional~$\mathcal{J}_\beta$ defined in~\eqref{functional associated to the problem}
satisfies the
Palais-Smale condition. 
To meet this goal, we provide a result that guarantees a suitable convergence of the term involving the negative part of~$\mu$.

\begin{lemma}\label{lemma:agg0547543ygf}
Let~$u_n$ be a bounded  sequence in~$H^\mu_r(\R^N)$.

Then, there exists~$u:\R^N\to\R$ such that, up to a subsequence, for every~$v\in H^\mu_r(\R^N)$,
\begin{equation}\label{9er4kjzsdvjhawrygu3wgfruaws}\begin{split}& \lim_{n\to+\infty}
\int_{[0,1]} \iint_{\R^{2N}} C_{N,s}\frac{(u_n(x)-u_n(y))(v(x)-v(y))}{|x-y|^{N+2s}}\, dx\, dy\, d\mu^-(s)\\&\qquad
=\int_{[0,1]} \iint_{\R^{2N}} C_{N,s}\frac{(u(x)-u(y)) (v(x)-v(y))}{|x-y|^{N+2s}}\, dx\, dy\, d\mu^-(s).\end{split}\end{equation}
Also,
\begin{equation}\label{9er4kjzsdvjhawrygu3wgfruaws2}
{\mbox{$u_n$ converges to~$u$ a.e. in~$\R^N$ as~$n\to+\infty$.}}\end{equation}
\end{lemma}

\begin{proof}
We notice that if~$\mu^-$ is the null measure, then~\eqref{9er4kjzsdvjhawrygu3wgfruaws} is void and~\eqref{9er4kjzsdvjhawrygu3wgfruaws2} follows from the compact embedding in~\eqref{aggcompatto}.

If instead~$\mu^-$ is the Dirac measure at~$0$, then the embedding in~\eqref{aggcompatto00}
gives the existence of~$u$ such that~\eqref{9er4kjzsdvjhawrygu3wgfruaws} is satisfied.
Moreover,
from~\eqref{aggcompatto} we know that
\begin{equation}\label{dslkhgvlkhgds345678mnvxzfiewpihtihgriep}
{\mbox{$u_n$ converges  to some~$v$ in~$L^p_{\rm{loc}}(\R^N)$ for all~$p\in[1,2^*_{s_\sharp})$}}\end{equation}
and therefore almost everywhere in~$\R^N$.

We now claim that
\begin{equation}\label{ncr9043iglsjhgldjlv} u=v.\end{equation}
Indeed, let~$\varphi\in C^\infty_0(\R^N)$ and suppose that the support of~$\varphi$ is contained in~$B_R$ for some~$R>0$.
Then, we know from~\eqref{9er4kjzsdvjhawrygu3wgfruaws} that
\begin{equation}\label{cdsf8iu43y8toiew9854kDFGHJKjvdkjhds000}
\lim_{n\to+\infty}\int_{\R^N} u_n(x)\varphi(x)\,dx= \int_{\R^N} u(x)\varphi(x)\,dx. \end{equation}

Also,
\begin{equation}\label{cdsf8iu43y8toiew9854kDFGHJKjvdkjhds}
\int_{\R^N} u_n(x)\varphi(x)\,dx= \int_{B_R} u_n(x)\varphi(x)\,dx=
\int_{B_R} (u_n(x)-v(x)) \varphi(x)\,dx+\int_{B_R} v(x) \varphi(x)\,dx.
\end{equation}
By the H\"older inequality, we have that
\begin{eqnarray*}\left|
\int_{B_R} (u_n(x)-v(x)) \varphi(x)\,dx\right|\le \left(\int_{B_R} |u_n(x)-v(x)|^2\,dx\right)^{\frac12}
\left(\int_{B_R}|\varphi(x)|^2\,dx\right)^{\frac12}\le C\|u_n-v\|_{L^2(B_R)}.
\end{eqnarray*} This and~\eqref{dslkhgvlkhgds345678mnvxzfiewpihtihgriep} imply that
$$ \lim_{n\to+\infty} \int_{B_R} (u_n(x)-v(x)) \varphi(x)\,dx=0.$$
From this and~\eqref{cdsf8iu43y8toiew9854kDFGHJKjvdkjhds} we conclude that
$$ \lim_{n\to+\infty} \int_{\R^N} u_n(x)\varphi(x)\,dx= \int_{B_R} v(x) \varphi(x)\,dx= \int_{\R^N} v(x) \varphi(x)\,dx. $$
This, together with~\eqref{cdsf8iu43y8toiew9854kDFGHJKjvdkjhds000}, entails~\eqref{ncr9043iglsjhgldjlv}, and completes the proof of~\eqref{9er4kjzsdvjhawrygu3wgfruaws2} in this case.
Hence, from now on we suppose that~$\mu^-((0,\bar s])>0$.
We observe that, thanks to Proposition~\ref{assorb measure},
\begin{eqnarray*}&&
\int_{[0,1]} [u_n]_s^2 \, d\mu^-(s) = \int_{[0,\bar s]} [u_n]_s^2 \, d\mu^-(s) \leq  \gamma \left( 
        2 \mu^+([\bar s, 1]) \|u_n\|^2_2+ \int_{[\bar s,1]} [u_n]_s^2\, d\mu^+(s)\right)\\&&\qquad\le
        \gamma \left( 
        2 \mu^+([\bar s, 1]) \|u_n\|^2_2+ \int_{[0,1]} [u_n]_s^2\, d\mu^+(s)\right)
        \end{eqnarray*} and this quantity is bounded uniformly in~$n$.
As a consequence of this and the Banach-Alaoglu Theorem
we obtain that there exist~$u$, $w$ such that, for all~$v\in  H^\mu_r(\R^N)$,
\begin{equation}\label{saru34yetukwhtikwetuio00}\begin{split}& \lim_{n\to+\infty}
\int_{[0,1]} \iint_{\R^{2N}} C_{N,s}\frac{(u_n(x)-u_n(y))(v(x)-v(y))}{|x-y|^{N+2s}}\, dx\, dy\, d\mu^-(s)
+\int_{\R^N} u_n(x)v(x)\,dx
\\&\qquad
=\int_{[0,1]} \iint_{\R^{2N}} C_{N,s}\frac{(u(x)-u(y)) (v(x)-v(y))}{|x-y|^{N+2s}}\, dx\, dy\, d\mu^-(s)+\int_{\R^N} u(x)v(x)\,dx
\end{split}\end{equation}
and, for all~$v\in L^2(\R^N)$,
\begin{equation}\label{saru34yetukwhtikwetuio} \lim_{n\to+\infty}
\int_{\R^N} u_n(x)v(x)\,dx
=\int_{\R^N} w(x)v(x)\,dx
.\end{equation}

Now, by the Dominated Convergence Theorem, we have that
$$ \lim_{\varepsilon\searrow0}\mu^-([\varepsilon,\bar s])=\mu^-((0,\bar s])>0.$$
Hence, there exists~$\varepsilon_0\in(0,\bar s]$ such that~$\mu^-([\varepsilon_0,\bar s])>0$.

From this and Lemma~\ref{fractional immersion} we thus find that, for all~$s\in[\varepsilon_0,\bar s]$,
\begin{equation*}
[u_n]^2_{\varepsilon_0}\leq \|u_n\|_2^2 + [u_n]^2_{s}
\end{equation*} and therefore, integrating in~$\mu^-$ over~$[\varepsilon_0,\bar s]$,
$$\mu^-([\varepsilon_0,\bar s]) [u_n]^2_{\varepsilon_0}\leq \mu^-([\varepsilon_0,\bar s]) \|u_n\|_2^2 + \int_{[\varepsilon_0,\bar s]}[u_n]^2_{s}\,d\mu^-(s),$$
which is bounded uniformly in~$n$, in light of
Proposition~\ref{assorb measure}.

This and the compactness result for fractional Sobolev spaces (see e.g.~\cite{DPV}) give that
\begin{equation}\label{saru34yetukwhtikwetuio0076}
{\mbox{$u_n$ converges to~$u$ in~$L^2_{\rm{loc}}(\R^N)$
as~$n\to+\infty$.}}\end{equation}
Accordingly, if~$v\in C^\infty_0(\R^N)$, we deduce from~\eqref{saru34yetukwhtikwetuio} that
\begin{eqnarray*}&& \int_{\R^N} w(x)v(x)\,dx
= \lim_{n\to+\infty}
\int_{\R^N} u_n(x)v(x)\,dx
\\&&\qquad= \lim_{n\to+\infty}
\int_{\R^N} \big(u_n(x)- u(x)\big)v(x)\,dx+
\int_{\R^N}  u(x)v(x)\,dx= 
\int_{\R^N}  u(x)v(x)\,dx\end{eqnarray*}
and therefore~$u=w$.
From this information, \eqref{saru34yetukwhtikwetuio00} and~\eqref{saru34yetukwhtikwetuio}
we obtain~\eqref{9er4kjzsdvjhawrygu3wgfruaws}.
Furthermore, the claim in~\eqref{9er4kjzsdvjhawrygu3wgfruaws2} is a consequence of~\eqref{saru34yetukwhtikwetuio0076}.
\end{proof}

\begin{proposition}\label{palais-smale-prop}
There exists~$\gamma_0>0$, depending on~$N$, $\alpha$, $\beta$, $\mu^+$, and~$\bar s$, such that for all~$\gamma\in[0,\gamma_0]$ the following statement holds true.

The functional $\mathcal{J}_\beta$ satisfies the Palais-Smale condition at every level $b \in \mathbb{R}$ in the following sense: any sequence $\{u_n\} \subset H_r^\mu(\mathbb{R}^N)$ satisfying
\begin{align}
    &\mathcal{J}_\beta (u_n) \to b, \label{1-PSP} \\
    &\mathcal{J}'_\beta(u_n) \to 0 \quad \text{strongly in } (H_r^\mu(\mathbb{R}^N))^*, \label{2-PSP} \\
    &\mathcal{P}_\beta(u_n) \to 0, \label{3-PSP}
\end{align} as~$n\to+\infty$,
admits a strongly convergent subsequence in $H_r^\mu(\mathbb{R}^N)$.
\end{proposition}

\begin{proof}
Let~$b\in \mathbb{R}$ and~$u_n$ be a sequence in $H_r^\mu(\mathbb{R}^N)$ satisfying \eqref{1-PSP}-\eqref{3-PSP}. Using conditions~\eqref{1-PSP} and~\eqref{3-PSP}, we obtain that
  \begin{eqnarray*}
      && (N+\alpha)b+o(1)=(N+\alpha)\mathcal J_\beta(u_n)-\mathcal{P}_\beta(u_n)\\ &&\qquad
      = \frac{N+\alpha}{2}\int_{[0,1]} \|(-\Delta)^{\frac{s}{2}}u_n\|^2_{2}\, d\mu^+(s)-\frac{N+\alpha}{2}
      \int_{[0,\bar s]} \|(-\Delta)^{\frac{s}{2}}u_n\|^2_{2}\, d\mu^-(s)\\ &&\qquad\qquad
      + \frac{\beta(N+\alpha)}{2}\|u_n\|^2_{2}-\frac{N+\alpha}{2}\mathcal{D}(u_n)-\int_{[0,1]} \frac{N-2s}{2} \|(-\Delta)^{\frac{s}{2}}u_n \|^2_{2}\, d\mu^+(s)\\ &&\qquad\qquad
       +\int_{[0,1]} \frac{N-2s}{2}\|(-\Delta)^{\frac{s}{2}}u_n \|^2_{2}\, d\mu^-(s)\nonumber --\frac{\beta N}{2}\|u_n\|^2_{2} +\frac{N+\alpha}{2}\mathcal{D}(u_n)\\ &&\qquad
\geq \frac{N+\alpha}{2}\int_{[0,1]} \|(-\Delta)^{\frac{s}{2}}u_n\|^2_{2}\, d\mu^+(s)-\frac{N+\alpha}{2}\int_{[0,\bar s]} \|(-\Delta)^{\frac{s}{2}}u_n\|^2_{2}\, d\mu^-(s) \\&&\qquad\qquad -\int_{[0,1]} \frac{N-2s}{2} \|(-\Delta)^{\frac{s}{2}}u_n \|^2_{2}\, d\mu^+(s) +\frac{\alpha \beta}{2}\|u_n\|^2_{2}.\end{eqnarray*}
    
   We observe that, in light of Proposition~\ref{assorb measure}, 
$$        \int_{[0,\bar s]} \|(-\Delta)^{\frac{s}{2}}u_n \|^2_{2} \, d\mu^-(s) \leq  \gamma \left( 
        2 \mu^+([\bar s, 1]) \|u_n\|^2_2+ \int_{[0,1]} \|(-\Delta)^{\frac{s}{2}}u_n \|^2_{2}\, d\mu^+(s)\right),   $$
   and therefore
  \begin{eqnarray*}
 (N+\alpha)b+o(1)&\geq&\frac{N+\alpha}{2}(1-\gamma )\int_{[0,1]} \|(-\Delta)^{\frac{s}{2}}u_n\|^2_{2}\, d\mu^+(s)-(N+\alpha) \gamma \mu^+([\bar s, 1])
  \|u\|^2_2 \\ &&\qquad -\int_{[0,1]} \frac{N-2s}{2} \|(-\Delta)^{\frac{s}{2}}u_n \|^2_{2}\, d\mu^+(s) +\frac{\alpha \beta}{2}\|u_n\|^2_{2}\\&\geq&
\frac{N+\alpha}{2}(1-\gamma )\int_{[0,1]} \|(-\Delta)^{\frac{s}{2}}u_n\|^2_{2}\, d\mu^+(s)-(N+\alpha) \gamma \mu^+([\bar s, 1]) \|u\|^2_2\\ &&\qquad -\frac{N}{2} \int_{[0,1]} \|(-\Delta)^{\frac{s}{2}}u_n \|^2_{2}\, d\mu^+(s) +\frac{\alpha \beta}{2}\|u_n\|^2_{2}  \\
 &=&\left( \frac{\alpha}{2}(1-\gamma) -\frac{N}{2} \gamma\right)\int_{[0,1]} \|(-\Delta)^{\frac{s}{2}}u\|^2_{2}\, d\mu^+(s) +\left( \frac{\alpha \beta}{2}-(N+\alpha)\gamma\mu^+([\bar s, 1])\right)\|u_n\|^2_{2}.
    \end{eqnarray*} 
    
Accordingly, taking~$\gamma$ sufficiently small, possibly in dependence of~$N$, $\alpha$, $\beta$, $\mu^+$,
and~$\bar s$, we find that
 $$ (N+\alpha)b+o(1)
      \geq  \frac{\alpha}{4} \int_{[0,1]} \|(-\Delta)^{\frac{s}{2}}u_n\|^2_{2}\, d\mu^+(s)+ \frac{\alpha \beta}{4}\|u_n\|^2_{2}.
 $$
 
Hence, the sequence  $u_n$ is bounded in $\M$. Thus, there exists a subsequence, still denoted by~$u_n$, such that $u_n \rightharpoonup u_0$ in $ H_r^{\mu}(\mathbb{R}^N)$. 
By the compact embedding in~\eqref{aggcompatto}, we deduce that~$u_n \to u_0$ in~$
L^p(\R^N)$ for every~$p \in (2,2^*_{s_\sharp})$
and a.e. in~$\R^N$.

Thanks to the assumptions $(F_1)-(F_4)$ we also obtain that \begin{equation}\label{cxwe84fkdshgfheroiyt9854ihgewhek}\begin{split}&
    \lim_{n\to+\infty}  \int_{\mathbb{R}^N} (I_{\alpha} * F(u_n))f(u_n)u_0\, dx=  \int_{\mathbb{R}^N} (I_{\alpha} * F(u_0))f(u_0)u_0\, dx
 \\{\mbox{and }}\quad&  \lim_{n\to+\infty}
      \int_{\mathbb{R}^N} (I_{\alpha} * F(u_n))f(u_n)u_n\, dx =\int_{\mathbb{R}^N} (I_{\alpha} * F(u_0))f(u_0)u_0\, dx.
  \end{split}\end{equation}
  
Moreover,
we know by assumptions that~$\langle \mathcal J'_\beta(u_n), u_n\rangle \to 0$ and~$\langle \mathcal J'_\beta(u_n), u_0\rangle \to 0$
as~$n\to+\infty$,
or equivalently \begin{align*}
   &  \lim_{n\to+\infty}\int_{[0,1]}\|(-\Delta)^{s/2}u_n\|^2_2\, d\mu^+(s)
   -\int_{[0,1]}\|(-\Delta)^{s/2}u_n\|^2_2\, d\mu^-(s)+\beta\|u_n\|_2^2\\ &\qquad
   - \int_{\mathbb{R}^N} (I_{\alpha} * F(u_n))f(u_n)u_n\, dx= 0
\end{align*} and  \begin{align*}
    &\lim_{n\to+\infty} \int_{[0,1]}\iint_{\R^{2N}}C_{N,s}\frac{(u_n(x)-u_n(y))(u_0(x)-u_0(y))}{|x-y|^{N+2s}}dx\,dy\,d\mu^+(s)\\ &\qquad-\int_{[0,1]}\iint_{\R^{2N}}C_{N,s}\frac{(u_n(x)-u_n(y))(u_0(x)-u_0(y))}{|x-y|^{N+2s}}dx\,dy\,d\mu^-(s)\\ &\qquad+ \beta \int_{\R^N} u_n u_0\, dx-\int_{\mathbb{R}^N} (I_{\alpha} * F(u_n))f(u_n)u_0\, dx=0.
\end{align*}

As a result, we find that~$\langle\mathcal{J}_\beta(u_n), u_n-u_0\rangle\to 0$, hence, as~$n\to+\infty$,
\begin{align*}
&\int_{[0,1]}\|(-\Delta)^{s/2}u_n\|^2_2\, d\mu^+(s)-\int_{[0,1]}\|(-\Delta)^{s/2}u_n\|^2_2\, d\mu^-(s)+\beta \|u_n\|_2^2- \int_{\mathbb{R}^N} (I_{\alpha} * F(u_n))f(u_n)u_n\, dx\\ &\qquad
-\int_{[0,1]}\iint_{\R^{2N}}C_{N,s}\frac{(u_n(x)-u_n(y))(u_0(x)-u_0(y))}{|x-y|^{N+2s}}dx\,dy\,d\mu^+(s)\\ &\qquad
+\int_{[0,1]}\iint_{\R^{2N}}C_{N,s}\frac{(u_n(x)-u_n(y))(u_0(x)-u_0(y))}{|x-y|^{N+2s}}dx\,dy\,d\mu^-(s)\\ &\qquad
-\beta  \int_{\R^N} u_n u\, dx+\int_{\mathbb{R}^N} (I_{\alpha} * F(u_n))f(u_n)u_0\, dx=o(1).
\end{align*}
Namely,
\begin{equation}\label{bcwr83297tfiuwagfiuASDFG}\begin{split}
   & o(1)+\int_{\mathbb{R}^N} (I_{\alpha} * F(u_n))f(u_n)u_n\, dx-\int_{\mathbb{R}^N} (I_{\alpha} * F(u_n))f(u_n)u_0\, dx\\ 
   &\quad=\int_{[0,1]} \iint_{\R^{2N}}C_{N,s}\frac{(u_n(x)-u_n(y))\big((u_n(x)-u_n(y))-(u_0(x)-u_0(y))\big)}{|x-y|^{N+2s}}dx\,dy\,d\mu^+(s)\\ &\qquad+\int_{[0,1]}\iint_{\R^{2N}}C_{N,s}\frac{(u_n(x)-u_n(y))\big((u_0(x)-u_0(y))-(u_n(x)-u_n(y))\big)}{|x-y|^{N+2s}}dx\,dy\,d\mu^-(s)
   \\ &\qquad+\beta  \int_{\R^N} u_n(u_n- u_0)\, dx.\end{split}\end{equation}

We now set
\begin{eqnarray*}
A&:=&\int_{[0,1]}\iint_{\R^{2N}}C_{N,s}\frac{(u_n(x)-u_n(y))\big((u_n(x)-u_n(y))-(u_0(x)-u_0(y))\big)}{|x-y|^{N+2s}}dx\,dy\,d\mu^+(s)
\\&&\qquad+\beta  \int_{\R^N} u_n(u_n- u_0)\, dx
\\{\mbox{and }} \quad
B&:=&\int_{[0,1]}\iint_{\R^{2N}}C_{N,s}\frac{(u_n(x)-u_n(y))\big((u_0(x)-u_0(y))-(u_n(x)-u_n(y))\big)}{|x-y|^{N+2s}}dx\,dy\,d\mu^-(s)
\end{eqnarray*}

Concerning the term~$A$, we have that
\begin{equation*}\begin{split}
       & A=\int_{[0,1]}\iint_{\R^{2N}}
       C_{N,s} \frac{|(u_n(x)-u_n(y))-(u_0(x)-u_0(y))|^2}{|x-y|^{N+2s}}\, d\mu^+(s)\\ &\qquad
       + \int_{[0,1]} \iint_{\R^{2N}} C_{N,s}\frac{(u_0(x)-u_0(y))\big((u_n(x)-u_n(y))-(u_0(x)-u_0(y))\big)}{|x-y|^{N+2s}}\, dx\, dy\, d\mu^+(s)\\ & \qquad +\beta\int_{\R^N}|u_n-u_0|^2\, dx+\beta\int_{\R^N} u_0(u_n-u_0)\, dx .
   \end{split}\end{equation*}
   Similarly, for the term~$B$,
\begin{align*}
&B= - \int_{[0,1]}\iint_{\R^{2N}}C_{N,s} \frac{|(u_n(x)-u_n(y))-(u_0(x)-u_0(y))|^2}{|x-y|^{N+2s}}\, dx\, dy\, d\mu^-(s)\\ &\qquad+ 
\int_{[0,1]} \iint_{\R^{2N}} C_{N,s}\frac{(u_0(x)-u_0(y))\big((u_0(x)-u_0(y))-(u_n(x)-u_n(y))\big)}{|x-y|^{N+2s}}\, dx\, dy\, d\mu^-(s).  
\end{align*}

Now we use Proposition \ref{assorb measure} to see that
\begin{equation*}\begin{split}
&B\geq-\gamma \int_{[0,1]}\iint_{\R^{2N}}
C_{N,s} \frac{|(u_n(x)-u_n(y))-(u_0(x)-u_0(y))|^2}{|x-y|^{N-2s}}\, dx\, dy\, d\mu^+(s)\\&\qquad-2\mu^+([\bar{s},1])
\gamma \int_{\R^N}|u_n-u_0|^2\, dx\\ 
&\qquad+ \int_{[0,1]} \iint_{\R^{2N}} C_{N,s}\frac{(u_0(x)-u_0(y))\big((u_0(x)-u_0(y))-(u_n(x)-u_n(y))\big)}{|x-y|^{N-2s}}\, dx\, dy\, d\mu^-(s).\end{split}\end{equation*}

Gathering these pieces of information and using them into~\eqref{bcwr83297tfiuwagfiuASDFG}, we thus conclude that
\begin{equation*}\begin{split}
  & o(1)+\int_{\mathbb{R}^N} (I_{\alpha} * F(u_n))f(u_n)u_n\, dx-\int_{\mathbb{R}^N} (I_{\alpha} * F(u_n))f(u_n)u_0\, dx\\ 
   &\quad=A+B\\
   &\quad\geq (1-\gamma)\int_{[0,1]}\iint_{\R^{2N}}
       C_{N,s} \frac{|(u_n(x)-u_n(y))-(u_0(x)-u_0(y))|^2}{|x-y|^{N+2s}}\, dx\, dy\, d\mu^+(s)\\ &\qquad
       + \int_{[0,1]} \iint_{\R^{2N}} C_{N,s}\frac{(u_0(x)-u_0(y))\big((u_n(x)-u_n(y))-(u_0(x)-u_0(y))\big)}{|x-y|^{N+2s}}\, dx\, dy\, d\mu^+(s)\\ & \qquad +\big(\beta-2\mu^+([\bar{s},1])
\gamma \big)\int_{\R^N}|u_n-u_0|^2\, dx
 +\beta\int_{\R^N} u_0(u_n-u_0)\, dx
\\ 
&\qquad+ \int_{[0,1]} \iint_{\R^{2N}} C_{N,s}\frac{(u_0(x)-u_0(y))\big((u_0(x)-u_0(y))-(u_n(x)-u_n(y))\big)}{|x-y|^{N-2s}}\, dx\, dy\, d\mu^-(s).\end{split}
\end{equation*}

Furthermore, thanks to Lemma~\ref{lemma:agg0547543ygf}, 
\begin{equation}\label{weak convergence mu meno}
  \lim_{n\to+\infty}  \int_{[0,1]}\int_{\R^{2N}} C_{N,s}\frac{(u_0(x)-u_0(y))\big((u_0(x)-u_0(y))-(u_n(x)-u_n(y))\big)}{|x-y|^{N-2s}}\, dx\, dy\, d\mu^-(s)=0.
\end{equation}
From this and~\eqref{cxwe84fkdshgfheroiyt9854ihgewhek}, we thus gather that
\begin{equation*}\begin{split}
  &0=\lim_{n\to+\infty} (1-\gamma)\int_{[0,1]}\iint_{\R^{2N}}
       C_{N,s} \frac{|(u_n(x)-u_n(y))-(u_0(x)-u_0(y))|^2}{|x-y|^{N+2s}}\, dx\, dy\, d\mu^+(s)\\ &\qquad\qquad\qquad
      +\big(\beta-2\mu^+([\bar{s},1])
\gamma \big)\int_{\R^N}|u_n-u_0|^2\, dx
.\end{split}
\end{equation*}
As a result, by possibly taking~$\gamma$ smaller,
we obtain that
\begin{equation*}
0\ge  \lim_{n\to+\infty}\int_{[0,1]}\iint_{\R^{2N}}
       C_{N,s} \frac{|(u_n(x)-u_n(y))-(u_0(x)-u_0(y))|^2}{|x-y|^{N+2s}}\, dx\, dy\, d\mu^+(s)      +\int_{\R^N}|u_n-u_0|^2\, dx,
\end{equation*} namely~$\|u_n-u_0\|_\mu\to 0$ as~$n\to+\infty$, as desired.
\end{proof}

\subsection{Proof of Theorem~\ref{MP solution}}

We can now complete the proof of Theorem \ref{MP solution}.
For this, we also recall the following abstract deformation lemma, whose proof can be performed as in~\cite[Proposition~4.5]{HT}, (see also~\cite[Proposition~3.1 and Corollary~4.3]{IT}.

 \begin{lemma}\label{defomation Lemma}For all~$b\in\R$, let 
      \begin{equation*} 
      K_b:=\big\{ u \in H_r^\mu (\mathbb{R}^N\, | \, \mathcal J_\beta(u)=b\,\,, \mathcal J'_\beta(u)=0\,\,, \mathcal{P}_\beta(u)=0\big\}
  \end{equation*}be the set of critical points of $\mathcal J_\beta $ at level~$b$
  satisfying the Pohozaev identity.
  
  Then, for any $\bar \varepsilon >0$ and $U$ open neighborhood of $K_b$, there exist~$\varepsilon \in (0, \bar \varepsilon)$  and a continuous map~$\eta :[0,1] \times H_r^\mu (\mathbb{R}^N) \to H_r^\mu(\mathbb{R}^N)$ such that \begin{enumerate}[label=(\roman*)]
      \item $\eta(0,u)=u$ for all~$u \in \M$;
      \item $\eta(t,u)=u$ for all~$ (t,u)\in[0,1] \times \{J_\beta \leq b - \varepsilon \}$;
      \item $J_\beta(\eta(t,u)) \leq J_\beta (u)$ for all~$ (t,u) \in [0,1] \times \M;$
      \item $\eta(1, \{J_\beta \leq b + \varepsilon \}\smallsetminus U) \subset \{J_\beta \leq b- \varepsilon \}$;
      \item $\eta (1, \{J_\beta \leq b+ \varepsilon \}\smallsetminus U) \subset \{J_\beta \leq b- \varepsilon\}\cup U$;
      \item if $K_b=\varnothing, $ then $ \eta(1, \{J_\beta \leq b + \varepsilon \}) \subset \{J_\beta \leq b- \varepsilon\}$.
  \end{enumerate}
 \end{lemma}

\begin{proof}[Proof of Theorem \ref{MP solution}]
We observe that the set~$K_{l(\beta)}$ is compact,
thanks to Proposition~\ref{palais-smale-prop}.
Moreover, in light of Lemma~\ref{bjhkewr793798436780jfs7y6t5relk0} and~\ref{bjhkewr793798436780jfs7y6t5relk},
we have that the Mountain Pass value~$l(\beta)$, defined in~\eqref{mplevel00}, is bounded by above and is strictly positive. In particular, this gives that~$0 \notin K_{l(\beta)}$.
    
Thus, 
by applying the deformation result in Lemma~\ref{defomation Lemma} at level~$b := l(\beta) > 0$, we obtain the existence of a Mountain Pass solution. Moreover, $u \in K_{l(\beta)}$ by construction, thus $u \not\equiv 0$ and~$\mathcal{P}_\beta(u) = 0$. \end{proof}

\subsection{Proof of Theorem \ref{p=l}}
We prove now that the solution obtained in Theorem~\ref{MP solution}
actually minimizes the energy among all the solutions satisfying the Pohozaev identity.

\begin{proof}[Proof of Theorem \ref{p=l}]
We observe that the claim of Theorem \ref{p=l} will follow once we establish that~$l(\beta)=p(\beta).$

Moreover, since the solution found in Theorem~\ref{MP solution}
lies on the Pohozaev manifold, we have that
\begin{equation}\label{nbvc8urgewsiuytr83w7654}
p(\beta)\leq l(\beta).\end{equation}

We now claim that
\begin{equation}\label{bvcxwqr439765toliew} l(\beta)\leq  p(\beta).\end{equation}
For this, let $u \in \M \smallsetminus \{0\}$ be such that $\mathcal{P}
    _\beta(u) = 0$, namely, recalling~\eqref{Pohozaev functional},
    $$ 
    \int_{[0,1]} \frac{N-2s}{2} \|(-\Delta)^{\frac{s}{2}}u \|^2_2\, d\mu^+(s)-\int_{[0,1]} \frac{N-2s}{2} \|(-\Delta)^{\frac{s}{2}}u \|^2_2\, d\mu^-(s)+\frac{\beta N}{2}\|u\|^2_{2}  =\frac{N+\alpha}{2}\mathcal{D}(u)
    .$$
  Thus, exploiting Proposition~\ref{assorb measureBIS},
    \begin{equation}\label{bdr97t6hliksat804w8765487654}
    \frac{N+\alpha}{2} \mathcal{D}(u)\ge (1-\gamma\bar c)
    \int_{[0,1]} \frac{N-2s}{2} \|(-\Delta)^{\frac{s}{2}}u \|^2_2\, d\mu^+(s)+
\left(\frac{\beta N}{2}-\gamma \bar{c}\right)\|u\|^2_{2} 
     >0,
    \end{equation}provided that $\gamma$ is taken small enough, possibly in dependence of~$N$, $\beta$, $\mu^+$, and~$\bar s$.
    
    We now consider the function~$\psi(t)\in\H$ defined as
\begin{equation}\label{fufhgfbvn4935} \psi(t):=\begin{cases}
u\left(\frac{\cdot}{t}\right), &{\mbox{ if }} t>0,\\
0, &{\mbox{ if }} \tau=0\end{cases}\end{equation}
  and we know that
  \begin{equation}\label{nbvctrh6547uy65eu650}
  {\mbox{$J_\beta (\psi(t))<0$ for large values of $t\ge 1$}}\end{equation} 
  (indeed, one can use~\eqref{bncxiwi5uy98430thewsd} with~$u$ in place of~$w$, and notice that we are allowed in lieu of~\eqref{bdr97t6hliksat804w8765487654}).
    
Moreover, we claim that
\begin{equation}\label{nbvctrh6547uy65eu65}
J_\beta (\psi(t)) \text{ attains its maximum at } t = 1.
\end{equation}
For this, let $\bar{t} \in (0,+\infty)$ be a critical point of the map $t \mapsto \mathcal{J}_\beta(u(\frac{\cdot}{t}))$, that is
\begin{align*}
    \frac{\partial}{\partial t} \mathcal{J}_\beta(\psi(t))|_{\bar t} &= 
    \int_{[0,1]} \frac{N-2s}{2} \bar t^{N-2s-1} \|(-\Delta)^{\frac{s}{2}}u \|^2_2\, d\mu^+(s) \\
    &\qquad- \int_{[0,1]} \frac{N-2s}{2} \bar t^{N-2s-1} \|(-\Delta)^{\frac{s}{2}}u \|^2_2\, d\mu^-(s) \\
    &\qquad+ \frac{\bar t^{N-1}\beta N}{2}\|u\|^2_{2}  - \frac{N+\alpha}{2}\bar t^{N+\alpha -1}\mathcal{D}(u)\\&=0.
\end{align*}

We will show that
\begin{equation}\label{csdmtgui4379t4ot}
\frac{\partial^2}{\partial t^2} \mathcal{J}_\beta \left(u\left(\frac{\cdot}{\bar{t}}\right)\right) < 0.\end{equation}
Indeed,
\begin{equation*}\begin{split}
 \frac{\partial^2}{\partial t^2} \mathcal{J}_\beta (u(\tfrac{\cdot}{t}))|_{\bar t} &= \frac{\partial^2}{\partial t^2} \mathcal{J}_\beta (u(\tfrac{\cdot}{t}))|_{\bar t} - \frac{(N-1+\alpha)}{\bar{t}} \frac{\partial}{\partial t}\mathcal{J}_\beta(u(\tfrac{\cdot}{t}))|_{\bar t} \\
    &= \int_{[0,1]} \frac{N-2s}{2}(N-1-2s) \bar{t}^{N-2s-2} \|(-\Delta)^{\frac{s}{2}}u \|^2_2\, d\mu^+(s) \\
&\qquad- \int_{[0,1]} \frac{N-2s}{2}(N-1-2s) \bar{t}^{N-2s-2} \|(-\Delta)^{\frac{s}{2}}u \|^2_2\, d\mu^-(s) \\
&\qquad+ (N-1)\frac{\bar{t}^{N-2}\beta N}{2}\|u\|^2_{2}  - (N-1+\alpha)\frac{N+\alpha}{2}\bar{t}^{N+\alpha -2}\mathcal{D}(u) \\
&\qquad- \int_{[0,1]} \frac{N-2s}{2}(N-1+\alpha) \bar{t}^{N-2s-2} \|(-\Delta)^{\frac{s}{2}}u \|^2_2\, d\mu^+(s) \\
   &\qquad+ \int_{[0,1]} \frac{N-2s}{2}(N-1+\alpha) \bar{t}^{N-2s-2} \|(-\Delta)^{\frac{s}{2}}u \|^2_2\, d\mu^-(s) \\
&\qquad- (N-1+\alpha)\frac{\bar{t}^{N-2}\beta N}{2}\|u\|^2_{2}  + (N-1+\alpha)\frac{N+\alpha}{2}\bar{t}^{N+\alpha -2}\mathcal{D}(u) \\
&= \int_{[0,1]} \frac{N-2s}{2}(-2s-\alpha) \bar{t}^{N-2s-2} \|(-\Delta)^{\frac{s}{2}}u \|^2_2\, d\mu^+(s) \\
    &\qquad+ \int_{[0,1]} \frac{N-2s}{2}(2s+\alpha) \bar{t}^{N-2s-2} \|(-\Delta)^{\frac{s}{2}}u \|^2_2\, d\mu^-(s) - \alpha\frac{\bar{t}^{N-2}\beta N}{2}\|u\|^2_{2}.
\end{split}\end{equation*} 
Thus,
exploiting Proposition~\ref{assorb measureBIS}, we obtain that
\begin{equation*}
\begin{split}
&    \frac{\partial^2}{\partial t^2} \mathcal{J}_\beta (u(\tfrac{\cdot}{t}))|_{\bar t}\\ &\qquad\leq -(1-\gamma) \int_{[0,1]} \frac{N-2s}{2}(2s+\alpha) \bar{t}^{N-2s-2} \|(-\Delta)^{\frac{s}{2}}u \|^2_2\, d\mu^+(s)- (\alpha - \bar{c} \gamma)\frac{\bar{t}^{N-2}\beta N}{2}\|u\|^2_{2} < 0,
\end{split}\end{equation*}
provided that $\gamma$ is sufficiently small, possibly depending on~$\alpha$, $\mu^+$, and $\bar{s}$. This completes the proof of~\eqref{csdmtgui4379t4ot}.

Hence, from~\eqref{csdmtgui4379t4ot} we infer that~$\bar{t}$ is a local maximum.

Furthermore, we have that
\begin{align*}
    \lim\limits_{t\to 0^+} \frac{\mathcal{J}_\beta(u(\frac{\cdot}{t}))}{t^{N-1}} > 0 \qquad \text{and} \qquad \lim\limits_{t \to +\infty} \mathcal{J}_\beta(u(\tfrac{\cdot}{t})) = -\infty.
\end{align*}
These considerations
show that the map $t \mapsto \mathcal{J}_{\beta}(u(\frac{\cdot}{t}))$ has a unique local maximum in~$\bar t$.
In particular, since~$u \in \M$ is such that $\mathcal{P}_\beta(u)=0$, then~$\bar t=1$, and the claim in~\eqref{nbvctrh6547uy65eu65} is established.
 
 After a suitable rescaling, we have that~$\psi \in \Gamma_\beta$ (where~$\Gamma_\beta$ is defined in~\eqref{cjkwhoufguo7654}), and thus, exploiting~\eqref{nbvctrh6547uy65eu65},
 \begin{equation}\label{max over parametrizations}
    \mathcal J_\beta(u)=\max\limits_{t \in [0,1]} \mathcal J_\beta(\psi(t))\geq l(\beta).
\end{equation}
Taking the infimum in \eqref{max over parametrizations} over all the functions~$u \in \M \smallsetminus \{0\}$
such that $\mathcal{P}_\beta(u) = 0$,
we get~\eqref{bvcxwqr439765toliew}, as desired.

From~\eqref{nbvc8urgewsiuytr83w7654} and~\eqref{bvcxwqr439765toliew} we obtain that~$p(\beta)= l(\beta)$, as desired.
\end{proof}

We now check that, in general, Pohozaev minima are solutions of~\eqref{problem}. 

\begin{proposition}\label{Pohozaev minima solution of problem}
There exists $\gamma_0 > 0$, depending on~$N$, $\alpha$, $\beta$, $\mu^+$, and~$\bar s$, such that for
all~$\gamma \in [0,\gamma_0]$ the following statements hold true.

Every Pohozaev minimum is a solution of \eqref{problem}, i.e., if
\begin{align*}
    \mathcal J_\beta(u)=p(\beta)\qquad \textit{and}\qquad\mathcal{P}_\beta(u)=0
\end{align*}then
\begin{align*}
    \mathcal J'_\beta(u)=0.
\end{align*} As a consequence,
$$
    p(\beta)=\inf \big\{ \mathcal J_\beta(u)\, | \, u\in H_r^\mu(\R^N)\smallsetminus \{0\},\,\, \mathcal{P}_\beta(u)=0,\,\, \mathcal J'_\beta(u)=0 \big\}.
$$\end{proposition}

\begin{proof}
 Let $u$ be such that $\mathcal J_\beta(u) = p(\beta)$ and $\mathcal{P}_\beta(u) = 0$.
We consider the function~$\psi(t)$ defined in~\eqref{fufhgfbvn4935} and we
recall~\eqref{nbvctrh6547uy65eu650}
and~\eqref{nbvctrh6547uy65eu65}
to say that~$\mathcal J_\beta(\psi(t))$ is negative for large values of $t$
and its maximum value~$p(\beta)$ is attained at~$t = 1$, provided that~$\gamma$ is sufficiently small.

Now we assume by contradiction that $u$ is not critical. Let $I := [1 - \delta, 1 + \delta]$ be such that $\psi(I) \cap K_{p(\beta)} = \varnothing$
and set
\[
\bar{\varepsilon} := p(\beta) - \max_{t \notin I} \mathcal J_\beta(\psi(t)) > 0.
\]
Let also~$U$ be a neighborhood of $K_{p(\beta)}$ such that $\psi(I) \cap U = \varnothing$. Then, by Lemma \ref{defomation Lemma}, we have that
there exists a continuous map~$
\eta : [0, 1] \times H_r^\mu(\mathbb{R}^N) \to H_r^\mu(\mathbb{R}^N)$
at the level $p(\beta) \in \mathbb{R}$ satisfying the properties~$(i)-(vi)$. 

We now define the deformed path
\[
\widetilde{\psi}(t) := \eta(1, \psi(t))
\] and we observe that,
after a suitable rescaling, $\widetilde{\psi} \in \Gamma_\beta$, thanks to property~$(iii)$.

Now, if $t \notin I$, we have that~$\mathcal J_\beta(\psi(t)) < p(\beta) - \bar{\varepsilon}$, and thus by property $(ii)$ we obtain
\begin{equation}
\mathcal J_\beta(\widetilde{\psi}(t)) = \mathcal J_\beta(\psi(t)) < p(\beta) - \bar{\varepsilon}.
\label{eq:outsideI}
\end{equation} If instead~$t \in I$, since $\psi(t) \notin U$ and $\mathcal J_\beta(\psi(t)) \leq p(\beta) \leq p(\beta) + \bar\varepsilon$, by property $(iv)$ we obtain that
\begin{equation}
\mathcal J_\beta(\widetilde{\psi}(t)) \leq p(\beta) - \bar\varepsilon.
\label{eq:insideI}
\end{equation}

By combining~\eqref{eq:outsideI} and \eqref{eq:insideI}, we conclude that
\[
\max_{t \geq 0} \mathcal J_\beta(\widetilde{\psi}(t)) < p(\beta) = l(\beta),
\]
which provides the desired contradiction.
\end{proof}

In the case in which~$\mu^-\equiv0$,
we can show that there exists a solution of~\eqref{problem}
with constant sign.

\begin{proposition}\label{qwedfrgthbncmx98ty8yi4erwyt8}
Assume $\mu_-=0$ and  $F\not\equiv 0$ on $(0,+\infty)$ (i.e., $t_0$ in assumption $(F_4)$ can
be chosen positive). Then there exists a non-negative solution of \eqref{problem}, which is a minimum over all the non-negative functions on the Pohozaev manifold.
\end{proposition}

\begin{proof}
Let $g := \chi_{(0,+\infty)} f$. We have that $g$ still satisfies assumptions $(F_1)-(F_4)$. Therefore, by Theorem~\ref{MP solution}, there exists a solution $u$ of the equation
\begin{equation}\label{problem3}
\mathcal{L}_\mu u + \beta u = (I_\alpha * G(u))g(u) \quad \text{in } \mathbb{R}^N,
\end{equation}
where
\[
G(t) := \int_0^t g(\tau) \, d\tau.
\] In light of Theorem \ref{p=l}, this solution is a minimum on the Pohozaev manifold associated with~\eqref{problem3}.

We now show that $u$ is non-negative. For this, we observe that, by~\eqref{seminorm_s},
\begin{align*}
&\int_{[0,1]}\|(-\Delta)^{s/2} |u|\|_{2}^2\, d\mu(s)\\ &= \int_{[0,1]}C_{N, s} \int_{\mathbb{R}^{2N}} \frac{(|u(x)| - |u(y)|)^2}{|x - y|^{N+2s}} \, dx\,dy\,d\mu(s) \\
&= \int_{[0,1]}C_{N, s} \int_{\mathbb{R}^{2N}} \frac{|u(x)|^2 + |u(y)|^2 - 2|u(x)||u(y)|}{|x - y|^{N+2s}} \, dx\,dy \, d\mu(s)\\
&\leq\int_{[0,1]} C_{N, s} \int_{\mathbb{R}^{2N}} \frac{u^2(x) + u^2(y) - 2u(x)u(y)}{|x - y|^{N+2s}} \, dx\,dy\, d\mu(s) \\
&= \int_{[0,1]}C_{N, s} \int_{\mathbb{R}^{2N}} \frac{(u(x) - u(y))^2}{|x - y|^{N+2s}} \, dx\,dy\,d\mu(s) \\
&= \int_{[0,1]}\|(-\Delta)^{s/2} u\|_{2}^2\,d\mu(s).
\end{align*}
Therefore, $|u| \in \M $  and, writing\footnote{Here we are using the notation~$u^+:=\max\{0,u\}$ and~$u^-:=\max\{0,-u\}$.} $u = u^+ - u^-$,  we have that
\[
u^- = \frac{|u| - u}{2} \in \M.
\]

As a consequence, $u^-$ can be used as a test  function in the weak formulation of~\eqref{problem3}, which gives that
\begin{equation*}\begin{split}
&\int_{[0,1]} \left( C_{N,s} \iint_{\mathbb{R}^{2N}}\frac{(u(x)-u(y))(u^-(x)-u^-(y))}{|x-y|^{N+2s}}\,dx\,dy\right)d\mu(s)
+\beta \int_{\mathbb{R}^N}u\, u^-\, dx \\ 
&\qquad= \int_{\mathbb{R}^N} (I_{\alpha} * G(u))g(u) u^-\, dx
.
\end{split}\end{equation*}
By the definition of $g$, we thus have that
\begin{equation}\label{bvncxwqu4980328658t3uyg987654}
\int_{[0,1]} \left( C_{N,s} \iint_{\mathbb{R}^{2N}}\frac{(u(x)-u(y))(u^-(x)-u^-(y))}{|x-y|^{N+2s}}dx\,dy\right)d\mu(s)
+\beta \int_{\mathbb{R}^N}u\, u^-\, dx =
0.
\end{equation}

We observe that
\begin{align*}
& \iint_{\mathbb{R}^{2N}} \frac{(u(x) - u(y))(u^-(x) - u^-(y))}{|x - y|^{N+2s}} \, dx\,dy
\\&= - \int_{\{u(x) \geq 0\} \times \{u(y) < 0\}} \frac{(u^+(x) + u^-(y))u^-(y)}{|x - y|^{N+2s}} \, dx\,dy \\
&\qquad -\int_{\{u(x) < 0\} \times \{u(y) \geq 0\}} \frac{(u^-(x) + u^+(y))u^-(x)}{|x - y|^{N+2s}} \, dx\,dy  \\
&\qquad -  \int_{\{u(x) < 0\} \times \{u(y) < 0\}} \frac{(u^-(x) - u^-(y))^2}{|x - y|^{N+2s}} \, dx\,dy\\
&\leq 0 .
\end{align*}
Using this information into~\eqref{bvncxwqu4980328658t3uyg987654}, we obtain that
$$ 0\leq \beta \int_{\mathbb{R}^N}u\, u^-\, dx=-\beta \int_{\mathbb{R}^N}( u^-)^2\, dx,$$
which implies that $u^- \equiv 0$, i.e., $u \geq 0$. Hence, $g(u) = f(u)$ and $G(u) = F(u)$, which entails that $u$ is a (non-negative) solution of equation \eqref{problem}.
\end{proof}

\section{Boundedness and proof of Theorem~\ref{THM:LINFTY}}\label{Regularity}

Aim of this section is to present the proof of Theorem~\ref{THM:LINFTY}.

\subsection{Some preliminary estimates}
We recall the following statements:

\begin{lemma}\cite[Lemma 4.6]{CGT}\label{nonlocal_estimate} Let $N\geq 2$, $s \in (0,1)$ and $\alpha \in (0,N)$. Let~$\theta \in (\frac{\alpha}{N}, 2-\frac{\alpha}{N})$ and $H,K \in L^{\frac{2N}{\alpha}}(\R^N)+ L^{\frac{2N}{\alpha+2s}}(\R^N)$.

Then, for every $\varepsilon >0$ there exists $C_{\varepsilon, \theta}>0$ such that, for every~$u \in H^{s}(\R^N)$,
    \begin{equation*}
\int_{\mathbb{R}^N} \left( I_\alpha * \left( H |u|^\theta \right) \right) K |u|^{2 - \theta} \, dx 
\leq \varepsilon^2 \| (-\Delta)^{{s}/{2}} u \|_2^2 + C_{\varepsilon, \theta} \|u\|_2^2.
\end{equation*} 
\end{lemma}

\begin{lemma}\cite[Lemma 3.5]{GG}\label{truncation_inequality}
    Let \( a, b \in \mathbb{R} \), \( r \geq 2 \), and \( k \geq 0 \). Define the truncation function \( T_k : \mathbb{R} \to [-k, k] \) by
\[
T_k(t) :=
\begin{cases}
-k & \text{if } t \leq -k, \\
t & \text{if } -k < t < k, \\
k & \text{if } t \geq k.
\end{cases}
\]
Set \( a_k := T_k(a) \) and \( b_k := T_k(b) \). 

Then,\begin{equation*}
\frac{4(r - 1)}{r^2} \left( |a_k|^{r/2} - |b_k|^{r/2} \right)^2 
\leq (a - b) \left( a_k |a_k|^{r - 2} - b_k |b_k|^{r - 2} \right).
\end{equation*}
\end{lemma}

In the following we show that $u$ belongs to some $L^r(\R^N)$ with $r>2^*_{s_\sharp}=\frac{2N}{N-2s_\sharp}$,
being~$s_\sharp$ as in~\eqref{positivity measure 4}.

\begin{proposition}\label{L^r regularity}
   Assume that~$\mu^-=0$.
    Let $H, K \in L^{\frac{2N}{\alpha}}(\R^N)+L^{\frac{2N}{\alpha+2s_\sharp}}(\R^N)$. 
    Let $u\in \H$ be a weak solution of  \begin{equation}\label{problem2}
        \mathcal{L}_\mu u + u = (I_\alpha * (Hu))K  \,\,\,\,\, \text{in} \,\,\,\, \R^N.
    \end{equation}
    
    Then, $u\in L^r(\R^N)$ for all $r \in [2, \frac{N}{2}\frac{2N}{N-2s_\sharp})$ and \begin{align*}
        \|u\|_{r}\leq C_r\|u\|_{2}
    \end{align*}with $C_r>0$ not depending on $u$.
\end{proposition}

\begin{proof}
By Lemma \ref{nonlocal_estimate}
(used here with~$\theta:=1$ and~$\varepsilon:=\sqrt{\frac{\mu^+([s_\sharp,1])}{2 } }$), there exists~$\lambda > 0$
such that
\begin{equation} \label{eq:stima-lemma42}
 \int_{\mathbb{R}^N} \left( I_\alpha * H|u| \right) K|u| \, dx \leq \frac{\mu([s_\sharp,1])}{2} 
 \|(-\Delta)^{s_\sharp/2} u\|_{2}^2 + \frac{\lambda}{2}\|u\|_{2}^2.
\end{equation}
By possibly taking~$\lambda$ larger, we can suppose that
\begin{equation}\label{second condition of lambda}
    \lambda > \max\left\{1, 2\mu([\bar s,1])\right\}.
\end{equation}

Let us set
\[
H_n := H \chi_{\{|H| \leq n\}}, \quad K_n := K \chi_{\{|K| \leq n\}}, \quad \text{for } n \in \mathbb{N},
\]
and observe that
\[
H_n, K_n \in L^{\frac{2N}{\alpha}}(\mathbb{R}^N), \quad H_n \to H, \quad K_n \to K \quad \text{a.e. as } n \to +\infty,
\]
and
\begin{equation} \label{eq:cutoff-bounds}
|H_n| \leq |H|, \quad |K_n| \leq |K| \quad \text{for every } n \in \mathbb{N}.
\end{equation}

We define the bilinear form
\begin{equation} \label{eq:bilinear-form}\begin{split}
a_n(\varphi, \psi) := &\int_{[0,1]}\left(\int_{\mathbb{R}^N} (-\Delta)^{s/2} \varphi \, (-\Delta)^{s/2} \psi \, dx\right)
\, d\mu(s) \\ & \qquad
+ \lambda \int_{\mathbb{R}^N} \varphi \psi \, dx - \int_{\mathbb{R}^N} \left( I_\alpha * (H_n \varphi) \right) K_n \psi \, dx,
\end{split}\end{equation}
for every $\varphi, \psi \in H^\mu(\mathbb{R}^N)$.

We claim that
\begin{equation}\label{coiercotr86igjdm}
{\mbox{the bilinear form $a_n$ is coercive in $\H$, provided that~$\gamma$ is sufficiently small}}.
\end{equation}
Indeed, by~\eqref{eq:stima-lemma42} and Proposition~\ref{Prop H mu in H s},
we have that, for all~$ \varphi \in \H$,
\begin{equation*}\begin{split}
 a_n(\varphi, \varphi) &=\int_{[0,1]}\|(-\Delta)^{s/2} \varphi\|^2_2\, d\mu(s) 
+ \lambda \int_{\mathbb{R}^N} \varphi^2 \, dx - \int_{\mathbb{R}^N} \left( I_\alpha * (H_n \varphi) \right) K_n \varphi \, dx\\
&\geq \int_{[0,1]}\|(-\Delta)^{s/2} \varphi\|^2_2\, d\mu(s)   
+ \lambda\| \varphi\|_2^2 -  \frac{\mu([s_\sharp,1])}{2} 
 \|(-\Delta)^{s_\sharp/2} \varphi\|_{2}^2 -\frac{\lambda}{2}\|\varphi\|_{2}^2\\
 &\geq  \int_{[0,1]}\|(-\Delta)^{s/2} \varphi\|^2_2\, d\mu(s) 
+ \frac{\lambda}2\| \varphi\|_2^2 -  \frac{\mu([s_\sharp,1])}{2} \|\varphi\|^2_2-
 \frac{1}{2}\int_{[0,1]}\|(-\Delta)^{s/2} \varphi\|^2_2\, d\mu(s).
\end{split}
\end{equation*}
Hence, recalling~\eqref{second condition of lambda}, and taking~$\gamma$ sufficiently small,
\begin{equation*}\begin{split}
 a_n(\varphi, \varphi) 
&\geq \frac12\int_{[0,1]}\|(-\Delta)^{s/2} \varphi\|^2_2\, d\mu(s) +\frac{ \mu([ s_\sharp, 1])}2 \|\varphi\|^2_2\\&\geq\frac12
\min\left\{1, \mu([ s_\sharp, 1]) \right\}\|\varphi\|^2_\mu,
\end{split}
\end{equation*} which establishes~\eqref{coiercotr86igjdm}.

Now, we set
\[
f := (\lambda - 1)u \in H^\mu(\mathbb{R}^N).
\]
Then, by Lax-Milgram Theorem, for each $n \in \mathbb{N}$ there exists a unique $u_n \in H^\mu(\mathbb{R}^N)$ such that
\begin{equation*} 
a_n(u_n, \varphi) = \int_{\R^N} f \varphi\,dx, \quad {\mbox{for all }} \varphi \in H^\mu(\mathbb{R}^N),
\end{equation*}
which is equivalent to the weak formulation of the equation
\begin{equation} \label{eq:choquard-approx}
\mathcal{L} u_n + \lambda u_n - (I_\alpha * (H_n u_n)) K_n = (\lambda - 1)u \quad \text{in } \mathbb{R}^N,
\end{equation} namely, for all $ \varphi \in \H$,
\begin{align*}
&\int_{[0,1]}\int_{\mathbb{R}^N} (-\Delta)^{s/2} u_n \, (-\Delta)^{s/2} \varphi \, dx\, d\mu(s) \\ &+ \lambda \int_{\mathbb{R}^N} u_n \varphi \, dx - \int_{\mathbb{R}^N} \left(I_\alpha * (H_n u_n)\right) K_n \varphi \, dx = (\lambda - 1) \int_{\mathbb{R}^N} u \varphi \, dx.
\end{align*}

Now, from \eqref{coiercotr86igjdm} we infer that
\begin{equation} \label{eq:un-bound}
\|u_n\|_{\mu}^2 \leq C a_n(u_n,u_n)\leq  C\|f\|_{2}\|u_n\|_{\mu}
=C
(\lambda - 1) \|u\|_{2}\|u_n\|_{\mu},
\end{equation}
hence the sequence $\{u_n\}$ is bounded in $H^\mu(\mathbb{R}^N)$. Therefore, up to a subsequence (and recalling
Proposition~\ref{Prop H mu in H s}), we have
\[
u_n \rightharpoonup \bar{u} \quad \text{weakly in } H^\mu(\mathbb{R}^N) \qquad{\mbox{and}}\qquad u_n \to \bar{u} \quad \text{a.e. in } \mathbb{R}^N.
\]

We now check that, for all~$\varphi\in\H$,
\begin{equation}\label{vcx904376t9ewigpwkjgf}
    \int_{\mathbb{R}^N} (I_\alpha * (H_n u_n)) K_n \varphi \, dx \longrightarrow \int_{\mathbb{R}^N} (I_\alpha * (H \bar{u})) K \varphi \, dx.
\end{equation}
By the continuous embedding in~\eqref{aggcompatto00} and Lemma \ref{regularity of f}, we have 
\begin{equation}\label{bvewioutry9843guyrs8ut}
    u_n \rightharpoonup \bar{u} \quad \text{in } L^q(\mathbb{R}^N), \quad \text{for all } q \in [2, 2^*_{s_\sharp}],
\end{equation}
where \( 2^*_{s_\sharp} = \frac{2N}{N - 2s_\sharp} \) is the fractional critical exponent.

We proceed by decomposing the functions $H$ and $K$ into their singular and decay components, as follows:
\begin{eqnarray*}
    H = H^* + H_* \in L^{\frac{2N}{\alpha}}(\mathbb{R}^N) + L^{\frac{2N}{\alpha+2s_\sharp}}(\mathbb{R}^N)\quad 
{\mbox{and }}\quad
    K = K^* + K_* \in L^{\frac{2N}{\alpha}}(\mathbb{R}^N) + L^{\frac{2N}{\alpha + 2s_\sharp}}(\mathbb{R}^N).
\end{eqnarray*}
To handle the convolution term, we analyze separately all four possible combinations. Let us fix generic components \( \widetilde{H} \in \{H^*, H_*\} \) and \( \widetilde{K} \in \{K^*, K_*\} \), with
\[
    \widetilde{H} \in L^\beta(\mathbb{R}^N), \quad \widetilde{K} \in L^\gamma(\mathbb{R}^N), \quad \text{where } \beta, \gamma \in \left\{ \frac{2N}{\alpha}, \frac{2N}{\alpha + 2s_\sharp} \right\}.
\]
We use the same notation for~$H_n$ and~$K_n$.

By the weak convergence in~\eqref{bvewioutry9843guyrs8ut}, one can prove that
\begin{equation}\label{csklfroi43y854yh9bijhfdrigoiher}
    \widetilde{H}_n u_n \rightharpoonup \widetilde{H} \bar{u} \quad \text{ in } L^r(\mathbb{R}^N), \quad \textit{with} \quad \frac{1}{r}=\frac{1}{\beta}+\frac{1}{q}.
\end{equation} 

Hence, by~\eqref{csklfroi43y854yh9bijhfdrigoiher}, and exploiting the continuity and linearity of the Riesz potential, we see that 
\begin{align}\label{weak potential}
    I_\alpha * (\widetilde{H}_n u_n) \rightharpoonup I_\alpha * (\widetilde{H} \bar{u}) \quad \text{ in } L^h(\mathbb{R}^N), \quad \text{with } \frac{1}{h} = \frac{1}{r} - \frac{\alpha}{N}.
\end{align}

We also claim that
\[
    (I_\alpha * (\widetilde{H}_n u_n)) \widetilde{K}_n \rightharpoonup (I_\alpha * (\widetilde{H} \bar{u})) \widetilde{K} \quad \text{weakly in } L^k(\mathbb{R}^N), \quad \text{with } \frac{1}{k} = \frac{1}{\gamma} + \frac{1}{h}.
\]
Indeed, let~$ \varphi \in L^{k'}(\R^N)$, where $k'$ is the conjugate exponent of $k$. Then,
\begin{equation}\label{jwr43986toiewhigfSDFGHu65}\begin{split}
    &\left|\int_{\R^N}  (I_\alpha * (\widetilde{H}_n u_n)) \widetilde{K}_n \varphi- \int_{\R^N} (I_\alpha * (\widetilde{H} \bar{u})) \widetilde{K} \varphi \right|\\ &\quad\leq \left |\int_{\R^N}\big(I_\alpha*(\widetilde H_n u_n)-I_\alpha*(\widetilde H\bar u)\big)\widetilde K \varphi\right|+\left|\int_{\R^N} (I_\alpha* (\widetilde H_n u_n))(\widetilde K_n -\widetilde K)\varphi\right|.
\end{split}\end{equation}

Since~$\widetilde K \varphi \in L^{h'}(\R^N)$, with $\frac{1}{h'}=\frac{1}{\gamma}+\frac{1}{k'}$, we can employ~\eqref{weak potential}
and conclude that
$$ \lim_{n\to+\infty} \left |\int_{\R^N}\big(I_\alpha*(\widetilde H_n u_n)-I_\alpha*(\widetilde H\bar u)\big)\widetilde K \varphi\right|=0.$$
Moreover, using the generalized H\"older inequality with $\frac{1}{h}+\frac{1}{\gamma}+ \frac{1}{k'}=1$, we have 
\begin{align}\label{(2) weak convergence final }
   \left| \int_{\R^N} (I_\alpha* (\widetilde H_n u_n))(\widetilde K_n -\widetilde K)\varphi\right|\leq \|I_\alpha * \widetilde H_n u_n\|_{h}\|\widetilde K_n-\widetilde K\|_{\gamma} \|\varphi\|_{k'}
    =0,
\end{align}
since $\widetilde K_n\to \widetilde K$ in $L^\gamma(\R^N)$ and $ \|I_\alpha * \widetilde H_n u_n\|_{h}$, $\|\varphi\|_{k'}$ are bounded.

These pieces of information, together with~\eqref{jwr43986toiewhigfSDFGHu65}, entail that
\begin{equation*}
\lim_{n\to+\infty}\left|\int_{\R^N}  (I_\alpha * (\widetilde{H}_n u_n)) \widetilde{K}_n \varphi- \int_{\R^N} (I_\alpha * (\widetilde{H} \bar{u})) \widetilde{K} \varphi \right|=0.
\end{equation*}
Hence, to complete the proof of~\eqref{vcx904376t9ewigpwkjgf} it suffices to check that
\begin{equation}\label{bvewugoh65748kjdcshfDFGH} H^\mu(\mathbb{R}^N)\subset L^{k'}(\mathbb{R}^N).\end{equation}

To this end, we check the four possible combinations of $\beta$ and $\gamma$ and we find that
\begin{enumerate}
    \item If $\beta = \gamma = \frac{2N}{\alpha}$ and $q=2$, then $k' = 2$.
    \item If $\beta = \frac{2N}{\alpha}$, $\gamma = \frac{2N}{\alpha+2s_\sharp}$ (or viceversa) and $q=2$, then $k' = 2^*_{s_\sharp}$.
    \item If $\beta=\gamma=\frac{2N}{\alpha+ 2s_\sharp}$  and $q=2^*_{s_\sharp}$, then $k'=2^*_{s_\sharp}$.
\end{enumerate}
In any case, we conclude that~\eqref{bvewugoh65748kjdcshfDFGH} holds true, and this allows us to complete the proof of~\eqref{vcx904376t9ewigpwkjgf}, as desired.

Gathering these pieces of information, we can pass the equation in~\eqref{eq:choquard-approx} and
conclude that~$\bar u$ is a weak solution of \begin{align*}
    \mathcal{L}_\mu \bar u+\lambda \bar u -(I_{\alpha} *(H\bar u))K=(\lambda -1)u\,\,\,\textit{in}\,\,\,\, \R^N
.\end{align*}
Then by Lax-Milgram Theorem, we have that $u=\bar u$ and hence
\begin{equation*}
 u_n \rightharpoonup u \quad \text{in } H^\mu(\mathbb{R}^N), 
 \qquad{\mbox{and}}\qquad u_n \to {u} \quad \text{a.e. in } \mathbb{R}^N.
\end{equation*}

Now, let $k \geq 0$ and define
\begin{equation*}
u_{n,k} := T_k(u_n) \in L^2(\mathbb{R}^N) \cap L^\infty(\mathbb{R}^N),
\end{equation*}
where $T_k$ is the truncation operator introduced in Lemma \ref{truncation_inequality}. Let $r \geq 2$. Then~$|u_{n,k}|^{r/2} \in H^\mu(\mathbb{R}^N)$, since~$h(t) := (T_k(t))^{r/2}$ is a Lipschitz function with $h(0) = 0$ and 
\begin{align*}
    &\int_{[0,1]}\|(-\Delta)^{\frac{s}{2}} |u_{n,k}|^{\frac{r}{2}}\|^2_{2}\, d\mu(s)\\&=\int_{[0,1]}\left(
    \iint_{\R^{2N}}\frac{||u_{n,k}(x)|^{\frac{r}{2}}-|u_{n,k}(y)|^{\frac{r}{2}}|^2}{|x-y|^{N+2s}}\, dx\, dy\right)\, d\mu(s)\\&\leq L\int_{[0,1]}\left(\iint_{\R^{2N}}\frac{||u_{n,k}(x)|-|u_{n,k}(y)||^2}{|x-y|^{N+2s}}\, dx\, dy\right)\, d\mu(s)\\
    &< +\infty.
\end{align*}

Let us define
\begin{equation*}
\varphi := u_{n,k} |u_{n,k}|^{r-2}.
\end{equation*}
Then $\varphi \in H^\mu(\mathbb{R}^N)$, since again  $h(t) := T_k(t) |T_k(t)|^{r-2}$ is a Lipschitz function and vanishes at $0$. Thus, we can choose~$\varphi$ as a test function in the weak formulation of~\eqref{eq:choquard-approx}, and
we get, by Lemma~\ref{truncation_inequality} that
\begin{align*}
& \frac{4(r-1)}{r^2} \int_{[0,1]}\left(
\int_{\mathbb{R}^N} \left| (-\Delta)^{s/2} \left( |u_{n,k}|^{r/2} \right) \right|^2 dx \right)\, d\mu(s)\\ &
=\int_{[0,1]} \left(C_{N,s} \iint_{\mathbb{R}^{2N}}\frac{4(r-1)}{r^2} \frac{\left| |u_{n,k}(x)|^{r/2} - |u_{n,k}(y)|^{r/2} \right|^2}{|x - y|^{N+2s}} dx\,dy\right)\, d\mu(s)\\ & 
\leq \int_{[0,1]}\left(
C_{N,s}\int \int_{\mathbb{R}^{2N}} \frac{(u_n(x) - u_n(y))(\varphi(x)-\varphi(y))  }{|x - y|^{N+2s}}\,dx\,dy\right)\,d\mu(s) \\ &=-\lambda \int_{\mathbb{R}^N} u_n \varphi \, dx + \int_{\mathbb{R}^N} (I_\alpha * (H_n u_n)) K_n \varphi \, dx + (\lambda - 1) \int_{\mathbb{R}^N} u \varphi \, dx. 
\end{align*}
Since $u_n \varphi \geq |u_{n,k}|^r$, we deduce that
\begin{equation}
\label{eq: u_nphimaggu_n,k}\begin{split}
&\frac{4(r-1)}{r^2} \int_{[0,1]}\int_{\mathbb{R}^N} \left| (-\Delta)^{s/2} \left( |u_{n,k}|^{r/2} \right) \right|^2 dx\, d\mu(s) 
\\&\quad \leq -\lambda \int_{\mathbb{R}^N} |u_{n,k}|^r dx + \int_{\mathbb{R}^N} (I_\alpha * (H_n u_n)) K_n \varphi \, dx + (\lambda - 1) \int_{\mathbb{R}^N} u \varphi \, dx.
\end{split}\end{equation}

Let us now focus on the Choquard-type term on the right-hand side. We estimate, using \eqref{eq:cutoff-bounds},
\begin{equation}
\label{eq:choquard-split-eng}\begin{split}
&\int_{\mathbb{R}^N} (I_\alpha * (H_n u_n)) K_n \varphi \, dx  \\ & \leq
\int_{\mathbb{R}^N} (I_\alpha * (|H_n||u_n| \chi_{\{|u_n| \leq k\}})) |K_n||u_{n,k}|^{r-1} \, dx \\ &\qquad   +
\int_{\mathbb{R}^N} (I_\alpha * (|H_n||u_n| \chi_{\{|u_n| > k\}})) |K_n||u_{n,k}|^{r-1} \, dx\\ &
 \leq \int_{\mathbb{R}^N} \left( I_\alpha * (|H_n||u_{n,k}|) \right) |K_n||u_{n,k}|^{r-1} \, dx 
\\ & \qquad+ \int_{\mathbb{R}^N} \left( I_\alpha * (|H_n||u_n|\chi_{\{|u_n| > k\}}) \right) |K_n||u_n|^{r-1} \, dx   \\
&  \leq \int_{\mathbb{R}^N} \left( I_\alpha * (|H||u_{n,k}|) \right) |K||u_{n,k}|^{r-1} \, dx 
\\ &\qquad+ \int_{\mathbb{R}^N} \left( I_\alpha * (|H_n||u_n|\chi_{\{|u_n| \geq k\}}) \right) |K_n||u_n|^{r-1} \, dx  \\
&=: (A)+(B).
\end{split}\end{equation}

We focus first on~$(A)$.  
Let $r \in \left[2, \frac{2N}{\alpha} \right)$, so that 
$
\theta := \frac{2}{r} \in \left( \frac{\alpha}{N}, 2 - \frac{\alpha}{N} \right)$.
In this way, we can use Lemma~\ref{nonlocal_estimate} with~$u$ replaced by~$|u_{n,k}|^{r/2}$
and~$\varepsilon:=\frac{\sqrt{2(r-1)\mu([s_\sharp,1])}}{r}$. Thus, recalling also
Proposition~\ref{Prop H mu in H s}, we have 
that
\begin{equation}\label{stima A}\begin{split}
  (A)& \leq \frac{2(r - 1)\mu([s_\sharp,1])}{r^2} \|(-\Delta)^{s_\sharp /2}(|u_{n,k}|^{r/2})\|_{2}^2 
+ C(r) \||u_{n,k}|^{r/2}\|_{2}^2\\&  \leq \frac{2(r - 1)}{r^2} \int_{[0,1]} \|(-\Delta)^{s /2}(|u_{n,k}|^{r/2})\|_{2}^2\, d\mu(s) +\mu([s_\sharp,1])
\||u_{n,k}|^{r/2}\|_2^2
+ C(r) \||u_{n,k}|^{r/2}\|_{2}^2.\end{split}\end{equation}

Let us now check~$(B)$.  
Assuming $
r < \min\left\{ \frac{2N}{\alpha}, \frac{2N}{N - 2s_{\sharp}} \right\}$ we have that~$u_n \in L^r(\R^N)$ and
\begin{align*}
    &|H_n||u_n| \in L^a(\mathbb{R}^N), \quad \text{with}\,\, \frac{1}{a} = \frac{\alpha}{2N} + \frac{1}{r},\\ & 
|K_n||u_n|^{r-1} \in L^b(\mathbb{R}^N), \quad \text{with } \frac{1}{b} = \frac{\alpha}{2N} + 1 - \frac{1}{r}.
\end{align*}Since
$\frac{1}{a} + \frac{1}{b} = \frac{N+\alpha}{N }$,
by the Hardy–Littlewood–Sobolev inequality in Proposition~\ref{prop:HLS},
\begin{equation}\label{bvc9854yhfahfalhlkjsas987643}\begin{split}
&\int_{\mathbb{R}^N} \left( I_\alpha * (|H_n||u_n|\chi_{\{|u_n| > k\}}) \right) |K_n||u_n|^{r-1} \, dx\\  
&\quad\leq C \left( \int_{\{|u_n| > k\}} ||H_n||u_n||^a \, dx \right)^{1/a}
\left( \int_{\mathbb{R}^N} ||K_n||u_n|^{r-1}|^b \, dx \right)^{1/b}. 
\end{split}\end{equation}

We observe that the last integral in~\eqref{bvc9854yhfahfalhlkjsas987643} is bounded
uniformly in~$k$, while
the second last is infinitesimal
thanks to the Dominated Convergence Theorem. Accordingly, we find that~$
(B) = o_k(1)$ as~$k\to+\infty$. 

Plugging this and~\eqref{stima A} into~\eqref{eq:choquard-split-eng}, we obtain that
\begin{equation*}\begin{split}
&\int_{\mathbb{R}^N} (I_\alpha * (H_n u_n)) K_n \varphi \, dx  \\ & \leq
 \frac{2(r - 1)}{r^2} \int_{[0,1]} \|(-\Delta)^{s /2}(|u_{n,k}|^{r/2})\|_{2}^2\, d\mu(s) +\mu([s_\sharp,1])
\||u_{n,k}|^{r/2}\|_2^2
+ C(r) \||u_{n,k}|^{r/2}\|_{2}^2+o_k(1).
\end{split}\end{equation*}

From this and~\eqref{eq: u_nphimaggu_n,k}, we thus conclude that
\begin{eqnarray*}
&&\int_{[0,1]}\frac{4(r - 1)}{r^2} \|(-\Delta)^{s/2}(|u_{n,k}|^{r/2})\|_{2}^2 \, d\mu(s)
\\ &&
\leq  -\lambda \int_{\mathbb{R}^N} |u_{n,k}|^r \, dx 
+ C(r) \int_{\mathbb{R}^N} |u_{n,k}|^r \, dx+\mu([s_\sharp,1])
\int_{\R^N}|u_{n,k}|^{r}\, dx \\ & &\qquad + \frac{2(r-1)}{r^2} \int_{[0,1]}\|(-\Delta)^{s/2} |u_{n,k}|^{r/2}\|^2_{2} \, d\mu(s)   + (\lambda - 1) \int_{\mathbb{R}^N} u\varphi \, dx   + o_k(1).
\end{eqnarray*} 
Therefore, rearranging terms,
\begin{eqnarray*}
&&\int_{[0,1]}\frac{2(r - 1)}{r^2} \|(-\Delta)^{s/2}(|u_{n,k}|^{r/2})\|_{2}^2 \, d\mu(s)
\\ &&  \leq \Big( C(r)- \lambda+\mu([s_\sharp,1])
\Big) \int_{\mathbb{R}^N} |u_{n,k}|^r \, dx  + (\lambda - 1) \int_{\mathbb{R}^N} u\varphi \, dx + o_k(1).
\end{eqnarray*}

Since $H^{\mu}(\R^N) \hookrightarrow H^{s_\sharp}(\R^N) \hookrightarrow L^{2^*_{s_\sharp}}(\R^N)$ (recall Proposition~\ref{Prop H mu in H s}),
we gather that
$$ C'(r) \left( \int_{\mathbb{R}^N} |u_{n,k}|^{\frac{r}{2} 2^*_{s_\sharp}} \, dx \right)^{\frac{2}{2^*_{s_\sharp}}}
\leq  \Big( C(r)- \lambda+\mu([s_\sharp,1])
\Big) \int_{\mathbb{R}^N} |u_{n,k}|^r \, dx  + (\lambda - 1) \int_{\mathbb{R}^N} u\varphi \, dx + o_k(1),
$$ for some~$C'(r)>0$.

Now,
since the sequence~$u_{n,k}$ is monotone in $k$ and converges to~$ u_n$ pointwise,
we can employ the Monotone Convergence Theorem and conlcude that
  \begin{equation}\label{u in Lr/2 2*}
 C'(r)\left(\int_{\R^N} |u_n|^{\frac{r}{2} 2^*_{s_\sharp}}\, dx\right)^{\frac{2}{2^*_{s_\sharp}}} \leq \Big( C(r)- \lambda+\mu([s_\sharp,1])
\Big) \int_{\mathbb{R}^N} |u_{n}|^r \, dx  + (\lambda - 1) \int_{\mathbb{R}^N} u\varphi \, dx .
\end{equation}
This entails that~$u_n \in L^{\frac{r}{2} 2^*_{s_\sharp}}(\mathbb{R}^N)$, with
$
\frac{r}{2} \in \left[1, \min\left\{ \frac{N}{\alpha}, \frac{N}{N - 2s_\sharp} \right\} \right).
$

Now, if $N - 2s_\sharp < \alpha$, the proof is complete.  
Otherwise, define~$r_1 := r$ and iterate with
$$
r_2 \in \left( \frac{2N}{N - 2s_\sharp}, \min\left\{ \frac{2N}{\alpha}, 2\left(\frac{N}{N - 2s_\sharp}\right)^2 \right\} \right).
$$ Again, if $\frac{2N}{\alpha}< 2 \Big( \frac{N}{N-2s_\sharp}\Big)$ we are done, otherwise we repeat the argument. Inductively, we have
\begin{equation*}
\left( \frac{N}{N - 2s_\sharp} \right)^m \to +\infty \quad \text{as } m \to +\infty,
\end{equation*} thus
$
\frac{2N}{\alpha} < 2 \left( \frac{N}{N - 2s_\sharp} \right)^m
$ after a finite number of steps. For such \( r = r_m \), consider again \eqref{u in Lr/2 2*}: by the almost everywhere convergence \( u_n \to u \) and Fatou's Lemma, we get
\begin{eqnarray*}
&&C''(r) \left( \int_{\mathbb{R}^N} |u|^{\frac{r}{2} 2^*_{s_\sharp}} \, dx \right)^{\frac{2}{2^*_{s_\sharp}}} 
\\ &&\leq \liminf_{n} C''(r) \left( \int_{\mathbb{R}^N} |u_n|^{\frac{r}{2} 2^*_{s_\sharp}} \, dx \right)^{\frac{2}{2^*_{s_\sharp}}} \\
&&\leq \liminf_{n}\Big( C(r)- \lambda+\mu([s_\sharp,1])
\Big)  \int_{\mathbb{R}^N} |u_n|^r \, dx + (\lambda - 1) \int_{\mathbb{R}^N} |u|\, |u_n|^{r-1} \, dx \Big) \\
&&\leq \Big( C(r)- \lambda+\mu([s_\sharp,1])
\Big)  \limsup_{n} \int_{\mathbb{R}^N} |u_n|^r \, dx + (\lambda - 1) \limsup_{n} \int_{\mathbb{R}^N} |u|\, |u_n|^{r-1} \, dx.
\end{eqnarray*}

Being \( u_n \) equibounded in \( H^{s_\sharp}(\mathbb{R}^N) \), and thus in \( L^{2^*_{s_\sharp}}(\mathbb{R}^N) \), by the iteration argument we deduce that it is also equibounded in \( L^r(\mathbb{R}^N) \), in particular, the boundness is given by $\|u\|^2_2$ times a constant $C(r)$. 
Thus, the right-hand side is a finite quantity, and we conclude that
\begin{equation*}
u \in L^{\frac{r}{2} 2^*_{s_\sharp}}(\mathbb{R}^N),
\end{equation*}
which is the claim.\end{proof} 

We also recall the following statement, which can be seen as a generalization of Proposition~\ref{prop:HLS} and whose proof, once Proposition~\ref{L^r regularity} is established, follows as in~\cite[Proposition 4.5]{CGT}.

\begin{proposition} \label{regularity convol} Assume that~$\mu^-=0$.
Let $u \in \H$ be a solution of \eqref{problem}. 

Then, $u \in L^q(\R^N)$ for $q \in \left[2, \frac{N}{\alpha}\frac{2N}{N-2s}\right]$ and \begin{align*}
       I_\alpha * F(u) \in C_0(\R^N),
   \end{align*} that is,
   \begin{equation*}
       I_{\alpha} *F(u)\,\,\, \in \,\,\, C(\R^N)\cap L^{\infty}(\R^N)
   \end{equation*}and 
   \begin{align*}
     \lim_{|x|\to+\infty}  (I_\alpha *F(u))(x)=0.
   \end{align*}
\end{proposition}

\subsection{A Kato-type inequality (in the weak sense)}
In order to complete the proof of Theorem~\ref{THM:LINFTY}, we will
also need a version of Kato's inequality in the general setting of this paper. The precise statement that we will use is the following:

\begin{proposition}\label{lem:kato}
Assume that~$\mu^-=0$.
Let~$\phi:\R\to[0,+\infty)$ be a~$C^\infty$ and convex function
such that~$\phi(t)=0$ for all~$t\le0$
and~$|\phi'(t)|+|\phi''(t)| \le L$ for every~$t\in\R$, for some~$L\in[0,+\infty)$.

Then, for all nonnegative functions~$v\in\H$, we have that~$\phi(v)\in\H$ and, for all nonnegative functions~$\varphi\in \H$,
\begin{equation}\label{katoineq00}\begin{split}&
\int_{[0,1]}\left(C_{N,s}\iint_{\R^{2N}}\frac{\big(\phi(v(x))-\phi(v(y))\big)
(\varphi(x)-\varphi(y))}{|x-y|^{N+2s}}\,dx\,dy\right)\,d\mu(s)\\&\qquad\le
\int_{[0,1]}\left(C_{N,s}\iint_{\R^{2N}}\frac{(v(x)-v(y))
\big(\phi'(v(x))\varphi(x)-\phi'(v(y))\varphi(y)\big)}{|x-y|^{N+2s}}\,dx\,dy\right)\,d\mu(s).\end{split}
\end{equation}
\end{proposition}

To establish Proposition~\ref{lem:kato}, we will use multiple times
the following observation:
 
\begin{lemma}\label{udwkjbdkjkjkju5r7u58700}
Let~$w_\varepsilon$, $\psi_\varepsilon$ be sequences of functions
in~$\H$ such that
$$ \sup_{\varepsilon} \int_{[0,1]} [w_\varepsilon]^2_s\,d\mu(s)<+\infty
$$ and
$$ \lim_{\varepsilon\searrow0}
\int_{[0,1]}\,[\psi_\varepsilon-\psi]_s^2\,d\mu(s)=0,$$
for some~$\psi\in\H$.

Then,
\begin{equation*}
\lim_{\varepsilon\searrow0}\int_{[0,1]}\left(C_{N,s}\iint_{\R^{2N}}\frac{ (w_\varepsilon(x)- w_\varepsilon(y))
\big((\psi_\varepsilon(x)-\psi(x))-(\psi_\varepsilon(y)
-\psi(y)\big)}{|x-y|^{N+2s}}\,dx\,dy\right)\,d\mu(s)=0.
\end{equation*}
\end{lemma}

\begin{proof}
Let~$\eta\in(0,1)$.
We observe that, by the Cauchy-Schwarz inequality,
\begin{eqnarray*}&&\lim_{\varepsilon\searrow0}
\left|\int_{[0,1]}\left(C_{N,s}\iint_{\R^{2N}}\frac{(w_\varepsilon(x)-w_\varepsilon(y))\, \big((\psi_\varepsilon(x)-\psi(x))- (\psi_\varepsilon(y)-\psi(y))\big)}{|x-y|^{N+2s}}\,dx\,dy\right)\,d\mu(s)\right|
\\&&\qquad\leq\lim_{\varepsilon\searrow0}
\int_{[0,1]}\left(C_{N,s}\iint_{\R^{2N}}\frac{|w_\varepsilon(x)-w_\varepsilon(y)|\, \big|(\psi_\varepsilon(x)-\psi(x))- (\psi_\varepsilon(y)-\psi(y))\big|}{|x-y|^{N+2s}}\,dx\,dy\right)\,d\mu(s)
\\&&\qquad\leq\lim_{\varepsilon\searrow0} \eta
\int_{[0,1]}[w_\varepsilon]^2_s\,d\mu(s)
+\frac1{\eta}\int_{[0,1]}\,[\psi_\varepsilon-\psi]_s^2\,d\mu(s)\\&&\qquad
\leq
\eta\sup_\varepsilon
\int_{[0,1]}[w_\varepsilon]^2_s\,d\mu(s).
\end{eqnarray*}
Thus, sending~$\eta\searrow0$, we obtain the desired result.
\end{proof}

\begin{proof}[Proof of Proposition~\ref{lem:kato}]
Thanks to the assumptions on~$\phi$, we have that, for all~$x$, $y\in\R^N$,
$$ |\phi(v(x))-\phi(v(y))|\le |\phi'(v(x))|\,| v(x)-v(y)|\le L| v(x)-v(y)|,$$
and therefore if~$v\in\H$ then~$\phi(v)\in\H$ as well.

We now establish the inequality in~\eqref{katoineq00}. We divide the proof into several steps:
\smallskip

\noindent {\bf{Step 1: proof of~\eqref{katoineq00} when~$v$, $\varphi\in C^\infty_0(\R^N,[0,+\infty))$.}}

In this case, we can write, in the principal value sense,
\begin{eqnarray*}
&&\iint_{\R^{2N}}\frac{\big(\phi(v(x))-\phi(v(y))\big)
(\varphi(x)-\varphi(y))}{|x-y|^{N+2s}}\,dx\,dy=
2\iint_{\R^{2N}}\frac{\big(\phi(v(x))-\phi(v(y))\big)
\varphi(x)}{|x-y|^{N+2s}}\,dx\,dy\\&&\qquad
\leq2\iint_{\R^{2N}}\frac{\phi'(v(x))( v(x)- v(y))
\varphi(x)}{|x-y|^{N+2s}}\,dx\,dy
\\&&\qquad=\iint_{\R^{2N}}\frac{( v(x)- v(y))\big(\phi'(v(x))
\varphi(x)-\phi'(v(y))\varphi(y)\big)}{|x-y|^{N+2s}}\,dx\,dy.
\end{eqnarray*}
Accordingly, the desired inequality in~\eqref{katoineq00}
follows by integrating over~$[0,1]$ with respect to~$\mu$.
\smallskip

\noindent {\bf{Step 2: proof of~\eqref{katoineq00} when~$v\in \H$ and~$\varphi\in C^\infty_0(\R^N,[0,+\infty))$.}}

Let~$v_\varepsilon\in C^{\infty}_0(\R^N,[0,+\infty))$
be such that~$v_\varepsilon\to v$ in~$\H$ (and a.e. in~$\R^N$, thanks to~\eqref{aggcompatto00}).
Then, we know that
\begin{equation}\label{SPK:DLHS1}\begin{split}&
\int_{[0,1]}\left(C_{N,s}\iint_{\R^{2N}}\frac{\big(\phi(v_\varepsilon(x))-\phi(v_\varepsilon(y))\big)
(\varphi(x)-\varphi(y))}{|x-y|^{N+2s}}\,dx\,dy\right)\,d\mu(s)\\&\qquad\le
\int_{[0,1]}\left(C_{N,s}\iint_{\R^{2N}}\frac{(v_\varepsilon(x)-v_\varepsilon(y))
\big(\phi'(v_\varepsilon(x))\varphi(x)-\phi'(v_\varepsilon(y))\varphi(y)\big)}{|x-y|^{N+2s}}\,dx\,dy\right)\,d\mu(s).\end{split}
\end{equation}

By Lemma~\ref{udwkjbdkjkjkju5r7u58700} (used here with~$w_\varepsilon:=\varphi$, $\psi_\varepsilon:=\phi(v_\varepsilon)$
and~$\psi:=\phi(v)$) we have that
\begin{equation}\label{SPK:DLHS2}
\begin{split}
&\lim_{\varepsilon\searrow0}\int_{[0,1]}\left(C_{N,s}\iint_{\R^{2N}}\frac{\big(\phi(v_\varepsilon(x))-\phi(v_\varepsilon(y))\big)
(\varphi(x)-\varphi(y))}{|x-y|^{N+2s}}\,dx\,dy\right)\,d\mu(s)\\&\qquad\qquad=
\int_{[0,1]}\left(C_{N,s}\iint_{\R^{2N}}\frac{\big(\phi(v (x))-\phi(v (y))\big)
(\varphi(x)-\varphi(y))}{|x-y|^{N+2s}}\,dx\,dy\right)\,d\mu(s).
\end{split}\end{equation}

Moreover, we know that the function
$$ \Psi_\varepsilon(x,y,s):=\sqrt{C_{N,s}}\;\frac{v_\varepsilon(x)-v_\varepsilon(y)}{|x-y|^{\frac{N+2s}2}}$$
converges as~$\varepsilon\searrow0$ in~$L^2(\R^{2N}\times[0,1],\,dx\,dy\,d\mu(s))$
and therefore (see e.g.~\cite[Theorem~4.9(b)]{brezis}) there exists~$ \Psi\in L^2(\R^{2N}\times[0,1],\,dx\,dy\,d\mu(s))$ such that~$|\Psi_\varepsilon|\le\Psi$.

As a result,
\begin{eqnarray*}&& C_{N,s}\frac{
|v_\varepsilon(x)-v_\varepsilon(y)|
\big|\phi'(v_\varepsilon(x))-\phi'(v_\varepsilon(y)) \big|\varphi(x)}{|x-y|^{N+2s}}
\le
C_{N,s}\frac{L
|v_\varepsilon(x)-v_\varepsilon(y)|^2\varphi(x)}{|x-y|^{N+2s}}\\&&\quad
\le L\|\varphi\|_{L^\infty(\R^N)}| \Psi_\varepsilon(x,y,s)|^2
\le L\|\varphi\|_{L^\infty(\R^N)}| \Psi(x,y,s)|^2\in L^1(\R^{2N}\times[0,1],\,dx\,dy\,d\mu(s)).\end{eqnarray*}

This observation allows us to use the Dominated Convergence Theorem and conclude that
\begin{equation}\label{PKJSMPD5m386b0-29mn098mnG.7bX2v6b}
\begin{split}&\lim_{\varepsilon\searrow0}
\int_{[0,1]}\left(C_{N,s}\iint_{\R^{2N}}\frac{(v_\varepsilon(x)-v_\varepsilon(y))
\big(\phi'(v_\varepsilon(x))-\phi'(v_\varepsilon(y))\big)\varphi(x)}{|x-y|^{N+2s}}\,dx\,dy\right)\,d\mu(s)
\\&\qquad=
\int_{[0,1]}\left(C_{N,s}\iint_{\R^{2N}}\frac{(v(x)-v(y))
\big(\phi'(v(x))-\phi'(v(y))\big)\varphi(x)}{|x-y|^{N+2s}}\,dx\,dy\right)\,d\mu(s).
\end{split}
\end{equation}

Additionally,
\begin{eqnarray*}
&&
C_{N,s}\;\frac{|v_\varepsilon(x)-v_\varepsilon(y)|
|\varphi(x)-\varphi(y)|\,|\phi'(v_\varepsilon(y))|}{|x-y|^{N+2s}}\le\sqrt{C_{N,s}}\,L
|\Psi_\varepsilon(x,y,s)|
\;\frac{
|\varphi(x)-\varphi(y)|}{|x-y|^{\frac{N+2s}2}}\\&&\qquad\le\sqrt{C_{N,s}}\,L
|\Psi(x,y,s)|
\;\frac{
|\varphi(x)-\varphi(y)|}{|x-y|^{\frac{N+2s}2}}\in L^1(\R^{2N}\times[0,1],\,dx\,dy\,d\mu(s)).
\end{eqnarray*}

Hence, the Dominated Convergence Theorem in this case yields that
\begin{equation*}
\begin{split}&\lim_{\varepsilon\searrow0}
\int_{[0,1]}\left(C_{N,s}\iint_{\R^{2N}}\frac{(v_\varepsilon(x)-v_\varepsilon(y))
(\varphi(x)-\varphi(y)) \phi'(v_\varepsilon(y))}{|x-y|^{N+2s}}\,dx\,dy\right)\,d\mu(s)
\\&\qquad=
\int_{[0,1]}\left(C_{N,s}\iint_{\R^{2N}}\frac{(v(x)-v(y))
(\varphi(x)-\varphi(y)) \phi'(v(y))}{|x-y|^{N+2s}}\,dx\,dy\right)\,d\mu(s)
.\end{split}
\end{equation*}

{F}rom this and~\eqref{PKJSMPD5m386b0-29mn098mnG.7bX2v6b} we arrive at\begin{equation*}
\begin{split}&\lim_{\varepsilon\searrow0}
\int_{[0,1]}\left(C_{N,s}\iint_{\R^{2N}}\frac{(v_\varepsilon(x)-v_\varepsilon(y))
\big(\phi'(v_\varepsilon(x))\varphi(x)-\phi'(v_\varepsilon(y))\varphi(y)\big)}{|x-y|^{N+2s}}\,dx\,dy\right)\,d\mu(s)\\&\qquad=\lim_{\varepsilon\searrow0}
\int_{[0,1]}\left(C_{N,s}\iint_{\R^{2N}}\frac{(v_\varepsilon(x)-v_\varepsilon(y))
\big(\phi'(v_\varepsilon(x))-\phi'(v_\varepsilon(y))\big)\varphi(x)}{|x-y|^{N+2s}}\,dx\,dy\right)\,d\mu(s)
\\&\qquad\qquad+
\int_{[0,1]}\left(C_{N,s}\iint_{\R^{2N}}\frac{(v_\varepsilon(x)-v_\varepsilon(y))
(\varphi(x)-\varphi(y)) \phi'(v_\varepsilon(y))}{|x-y|^{N+2s}}\,dx\,dy\right)\,d\mu(s)
\\&\qquad=
\int_{[0,1]}\left(C_{N,s}\iint_{\R^{2N}}\frac{(v(x)-v(y))
\big(\phi'(v(x))-\phi'(v(y))\big)\varphi(x)}{|x-y|^{N+2s}}\,dx\,dy\right)\,d\mu(s)
\\&\qquad\qquad+
\int_{[0,1]}\left(C_{N,s}\iint_{\R^{2N}}\frac{(v(x)-v(y))
(\varphi(x)-\varphi(y)) \phi'(v(y))}{|x-y|^{N+2s}}\,dx\,dy\right)\,d\mu(s)
\\&\qquad=\int_{[0,1]}\left(C_{N,s}\iint_{\R^{2N}}\frac{(v(x)-v(y))
\big(\phi'(v(x))\varphi(x)-\phi'(v(y))\varphi(y)\big)}{|x-y|^{N+2s}}\,dx\,dy\right)\,d\mu(s)
.\end{split}
\end{equation*}

We combine this information and~\eqref{SPK:DLHS2} to pass~\eqref{SPK:DLHS1}
to the limit, thus completing the proof of this step.

\smallskip

\noindent {\bf{Step 3: proof of~\eqref{katoineq00} when~$v\in \H$ and~$\varphi\in \H\cap L^\infty(\R^N)$.}}

In this case, by density we can find~$\varphi_\varepsilon\in C^\infty_0(\R^N,[0,+\infty))$
such that~$\|\varphi_\varepsilon\|_{L^\infty(\R^N)}\leq \|\varphi\|_{L^\infty(\R^N)}$ and~$\varphi_\varepsilon$ converges to~$\varphi$ in~$\H$ as~$\varepsilon\searrow0$. 
By Step~2, we know that~\eqref{katoineq00} holds true for~$\varphi_\varepsilon$.

Moreover, we use Lemma~\ref{udwkjbdkjkjkju5r7u58700} with~$w_\varepsilon:=\phi(v)$, $\psi_\varepsilon:=\varphi_\varepsilon$
and~$\psi:=\varphi$ and we infer that
\begin{equation}\label{vbcxnui37654yudhgxbgwsdf}\begin{split}&\lim_{\varepsilon\searrow0}
\int_{[0,1]}\left(C_{N,s}\iint_{\R^{2N}}\frac{\big(\phi(v(x))-\phi(v(y))\big)
(\varphi_\varepsilon(x)-\varphi_\varepsilon(y))}{|x-y|^{N+2s}}\,dx\,dy\right)\,d\mu(s)
\\&\qquad\qquad
=\int_{[0,1]}\left(C_{N,s}\iint_{\R^{2N}}\frac{\big(\phi(v(x))-\phi(v(y))\big)
(\varphi(x)-\varphi(y))}{|x-y|^{N+2s}}\,dx\,dy\right)\,d\mu(s).
\end{split}\end{equation}

Furthermore,
\begin{equation}\label{udwkjbdkjkjkju5r7u5870}\begin{split}&
\Big|(v(x)-v(y))
\big(\phi'(v(x))\varphi(x)-\phi'(v(y))\varphi(y)\big)
-(v(x)-v(y))
\big(\phi'(v(x))\varphi_\varepsilon(x)-\phi'(v(y))\varphi_\varepsilon(y)\big)\Big|\\
&\leq|v(x)-v(y)|\,
\Big|\phi'(v(x))(\varphi(x)-\varphi_\varepsilon(x))-\phi'(v(y))(\varphi(y)-\varphi_\varepsilon(y))\Big|\\
&\leq|v(x)-v(y)|\,|\phi'(v(x))|\Big|(\varphi(x)-\varphi_\varepsilon(x))- (\varphi(y)-\varphi_\varepsilon(y))\Big|
\\&\qquad+|v(x)-v(y)|\,|
\phi'(v(x))-\phi'(v(y))|\, |\varphi(y)-\varphi_\varepsilon(y)|\\&\leq L
|v(x)-v(y)|\, \Big|(\varphi(x)-\varphi_\varepsilon(x))- (\varphi(y)-\varphi_\varepsilon(y))\Big|
+L|v(x)-v(y)|^2\,|\varphi(y)-\varphi_\varepsilon(y)|.
\end{split}\end{equation}

We also remark that~$|\varphi(y)-\varphi_\varepsilon(y)|\le 2\|\varphi\|_{L^\infty(\R^N)}$,
hence, by the Dominated Convergence Theorem, we infer that
$$ \lim_{\varepsilon\searrow0}
\int_{[0,1]}\left(C_{N,s}\iint_{\R^{2N}}\frac{L|v(x)-v(y)|^2\,|\varphi(y)-\varphi_\varepsilon(y)|}{
|x-y|^{N+2s}}\,dx\,dy\right)\,d\mu(s)=0.$$
Also, exploiting Lemma~\ref{udwkjbdkjkjkju5r7u58700} with~$w_\varepsilon:=v$ we obtain that
\begin{equation*} \begin{split}&\lim_{\varepsilon\searrow0}
\int_{[0,1]}\left(C_{N,s}\iint_{\R^{2N}}\frac{ | v(x)- v(y)|\,\Big| (\varphi(x)-\varphi_\varepsilon(x))- (\varphi(y)-\varphi_\varepsilon(y))\Big|}{|x-y|^{N+2s}}\,dx\,dy\right)\,d\mu(s)
=0.\end{split}\end{equation*}
The last two displays, together with~\eqref{udwkjbdkjkjkju5r7u5870}, give that
\begin{equation*}\begin{split}&\lim_{\varepsilon\searrow0}
\int_{[0,1]}\left(C_{N,s}\iint_{\R^{2N}}\frac{(v(x)-v(y))
\big(\phi'(v(x))\varphi_\varepsilon(x)-\phi'(v(y))\varphi_\varepsilon(y)\big)}{|x-y|^{N+2s}}\,dx\,dy\right)\,d\mu(s)
\\&\qquad\qquad
=\int_{[0,1]}\left(C_{N,s}\iint_{\R^{2N}}\frac{\big(\phi(v(x))-\phi(v(y))\big)
(\varphi(x)-\varphi(y))}{|x-y|^{N+2s}}\,dx\,dy\right)\,d\mu(s).
\end{split}\end{equation*}

This and~\eqref{vbcxnui37654yudhgxbgwsdf} entail that we can pass to the limit the inequality for~$\varphi_\varepsilon$ and
complete the proof of this step.

\smallskip

\noindent {\bf{Step 4: proof of~\eqref{katoineq00} when~$v$, $\varphi\in \H$.}}

For all~$m\in\mathbb{N}$, let
$$ \varphi_m(x):=\begin{cases}
\varphi(x) &{\mbox{ if }} \varphi(x)\le m,\\
m &{\mbox{ if }} \varphi(x)> m,\end{cases}$$ and we observe that~$\varphi_m$ converges to~$\varphi$
a.e. in~$\R^N$ as~$m\to+\infty$.

Furthermore, we point out that
$$ |\varphi_m(x)-\varphi_m(y)|\le
|\varphi(x)-\varphi(y)|.$$
Therefore, by the Dominated Convergence Theorem,
\begin{equation}\label{bvncxui58743uygfsjhg1234567jgkfd}\begin{split}&\lim_{m\to+\infty}
\int_{[0,1]}\left(C_{N,s}\iint_{\R^{2N}}\frac{\big(\phi(v(x))-\phi(v(y))\big)
(\varphi_m(x)-\varphi_m(y))}{|x-y|^{N+2s}}\,dx\,dy\right)\,d\mu(s)\\&\qquad\qquad=
\int_{[0,1]}\left(C_{N,s}\iint_{\R^{2N}}\frac{\big(\phi(v(x))-\phi(v(y))\big)
(\varphi(x)-\varphi(y))}{|x-y|^{N+2s}}\,dx\,dy\right)\,d\mu(s)
\end{split}\end{equation}
and
\begin{equation}
\label{bvncxui58743uygfsjhg1234567jgkfd3}
\begin{split}&\lim_{m\to+\infty}
\int_{[0,1]}\left(C_{N,s}\iint_{\R^{2N}}\frac{(v(x)-v(y)) \phi'(v(x))(\varphi_m(x)-\varphi_m (y))}{|x-y|^{N+2s}}\,dx\,dy\right)\,d\mu(s)\\&\qquad\qquad=
\int_{[0,1]}\left(C_{N,s}\iint_{\R^{2N}}\frac{(v(x)-v(y)) \phi'(v(x))(\varphi(x)-\varphi(y))}{|x-y|^{N+2s}}\,dx\,dy\right)\,d\mu(s).
\end{split}\end{equation}

Moreover, by the convexity of~$\phi$, we know that, for all~$a$, $b\in\R$,
$$ (a-b)\big(\phi'(a)-\phi'(b)\big)\geq0$$
and thus
\begin{eqnarray*}&&
\int_{[0,1]}\left(C_{N,s}\iint_{\R^{2N}}\frac{(v(x)-v(y))
\big(\phi'(v(x)) -\phi'(v(y))\big)\varphi_m(y) }{|x-y|^{N+2s}}\,dx\,dy\right)\,d\mu(s)
\\&&\qquad\leq 
\int_{[0,1]}\left(C_{N,s}\iint_{\R^{2N}}\frac{(v(x)-v(y))
\big(\phi'(v(x)) -\phi'(v(y))\big)\varphi(y) }{|x-y|^{N+2s}}\,dx\,dy\right)\,d\mu(s)
.
\end{eqnarray*}

As a consequence of this fact and~\eqref{bvncxui58743uygfsjhg1234567jgkfd3}, we have that
\begin{equation}
\label{bvncxui58743uygfsjhg1234567jgkfd2}
\begin{split}&\lim_{m\to+\infty}
\int_{[0,1]}\left(C_{N,s}\iint_{\R^{2N}}\frac{(v(x)-v(y))
\big(\phi'(v(x))\varphi_m(x)-\phi'(v(y))\varphi_m(y)\big)}{|x-y|^{N+2s}}\,dx\,dy\right)\,d\mu(s)
\\&\qquad=\lim_{m\to+\infty}
\int_{[0,1]}\left(C_{N,s}\iint_{\R^{2N}}\frac{(v(x)-v(y))
\phi'(v(x)) (\varphi_m(x)- \varphi_m(y) )}{|x-y|^{N+2s}}\,dx\,dy\right)\,d\mu(s)
\\&\qquad\qquad+
\int_{[0,1]}\left(C_{N,s}\iint_{\R^{2N}}\frac{(v(x)-v(y))
\big(\phi'(v(x)) -\phi'(v(y))\big)\varphi_m(y) }{|x-y|^{N+2s}}\,dx\,dy\right)\,d\mu(s)
\\&\qquad\leq \lim_{m\to+\infty}
\int_{[0,1]}\left(C_{N,s}\iint_{\R^{2N}}\frac{(v(x)-v(y))
\phi'(v(x)) (\varphi_m(x)- \varphi_m(y) )}{|x-y|^{N+2s}}\,dx\,dy\right)\,d\mu(s)
\\&\qquad\qquad+
\int_{[0,1]}\left(C_{N,s}\iint_{\R^{2N}}\frac{(v(x)-v(y))
\big(\phi'(v(x)) -\phi'(v(y))\big)\varphi(y) }{|x-y|^{N+2s}}\,dx\,dy\right)\,d\mu(s)\\
&\qquad=
\int_{[0,1]}\left(C_{N,s}\iint_{\R^{2N}}\frac{(v(x)-v(y))
\phi'(v(x)) (\varphi(x)- \varphi(y) )}{|x-y|^{N+2s}}\,dx\,dy\right)\,d\mu(s)
\\&\qquad\qquad+
\int_{[0,1]}\left(C_{N,s}\iint_{\R^{2N}}\frac{(v(x)-v(y))
\big(\phi'(v(x)) -\phi'(v(y))\big)\varphi(y) }{|x-y|^{N+2s}}\,dx\,dy\right)\,d\mu(s)\\
\\&\qquad=\int_{[0,1]}\left(C_{N,s}\iint_{\R^{2N}}\frac{(v(x)-v(y))
\big(\phi'(v(x))\varphi(x)-\phi'(v(y))\varphi(y)\big)}{|x-y|^{N+2s}}\,dx\,dy\right)\,d\mu(s).
\end{split}\end{equation}

By Step~3, we know that~\eqref{katoineq00} holds true for~$v$ and~$\varphi_m$. This fact, together with~\eqref{bvncxui58743uygfsjhg1234567jgkfd} and~\eqref{bvncxui58743uygfsjhg1234567jgkfd2}, gives the desired result.
\end{proof}

\subsection{A useful approximation result}
In order to prove Theorem~\ref{THM:LINFTY}, one of the main ingredients is to exploit
Proposition~\ref{lem:kato} with a specific choice of~$\phi$. Namely,
given~$\beta > 1$ and~$T > 1$, we define the function
\begin{equation}\label{defofut98jkgfmnv9876}
\phi(t) := 
\begin{cases} 
0, & \text{if } t \le 0, \\
t^\beta, & \text{if } 0 < t < T, \\
\beta T^{\beta-1}(t - T) + T^\beta, & \text{if } t \ge T.
\end{cases}
\end{equation}
We observe that~$\phi$ is convex, but only~$C^1$, thus we cannot employ
Proposition~\ref{lem:kato} straight away.
To overcome this issue, we now present an approximation argument.

\begin{lemma}\label{lemma:p488approx}
Let~$\phi$ be as in~\eqref{defofut98jkgfmnv9876}.

Then, there exists a sequence of~$C^\infty$ and convex functions~$\phi_\delta:\R\to[0,+\infty)$ such that
\begin{eqnarray*}
&&\phi_\delta(t)=0  \quad\text{for all } t \le 0,\\
&&{\mbox{the support of~$\phi''_\delta$ is compact,}}\\
&&\phi_\delta\le\phi\qquad{\mbox{and}}\qquad 0\le \phi_\delta'\le\phi'\\
&&{\mbox{and }} \quad \lim_{\delta\searrow0}\phi_\delta(t)=\phi(t). 
\end{eqnarray*}
\end{lemma}

\begin{proof}
Let~$\eta\in C^\infty_0((-1,1), [0,+\infty))$ be even and such that
$$ \int_{-1}^1\eta(r)\,dr=1.$$
For any~$\delta\in(0,1)$, let
\begin{eqnarray*} \eta_\delta(x)&:=&\frac1{\delta}\eta\left(\frac{x}{\delta}\right)\\
{\mbox{and }}\quad
\phi_\delta(t)&:=&\int_{-\infty}^{+\infty}\phi(r-2\delta)\eta_\delta(t-r)\,dr
=\int_{-\delta}^\delta \phi(t-r-2\delta)\eta_\delta(r)\,dr.\end{eqnarray*}
Then, $\phi_\delta\in C^\infty(\R,[0+\infty))$, $\phi_\delta$ is convex and
$$ \lim_{\delta\searrow0}\phi_\delta(t)=\phi(t).$$

Now, we observe that, if~$t\leq0$ and~$r\in(-\delta,\delta)$,
$$ t-r-2\delta\leq -r-2\delta\leq -\delta$$
and therefore~$\phi(t-r-2\delta)=0$. As a consequence, for all~$t\leq0$,
\begin{equation}\label{ncer9784234567}
\phi_\delta(t)=\int_{-\delta}^\delta \phi(t-r-2\delta)\eta_\delta(r)\,dr=0.
\end{equation}

Furthermore, if~$t\geq 4T$ and~$r\in(-\delta,\delta)$
$$ t-r-2\delta\geq 4T-3\delta\geq T,$$
thus~$\phi( t-r-2\delta)=\beta T^{\beta-1}(t-r-2\delta - T) + T^\beta$. As a result, for all~$t\geq 4T$, 
\begin{eqnarray*}
\phi_\delta(t)&=&\int_{-\delta}^\delta \Big( \beta T^{\beta-1}(t-r-2\delta - T) + T^\beta\Big)\eta_\delta(r)\,dr\\&=&
\beta T^{\beta-1} (t-2\delta - T)+
T^\beta.
\end{eqnarray*}
This and~\eqref{ncer9784234567} entails that the support of~$\phi''_\delta$ is contained
in~$[0,\max\{4T,4\}]$, and therefore it is compact, as desired.

We also point out that, since~$\phi$ is monotone non-decreasing,
\begin{eqnarray*}
\phi_\delta(t)=\int_{-\delta}^\delta \phi(t-r-2\delta)\eta_\delta(r)\,dr\le 
\int_{-\delta}^\delta \phi(t-\delta)\eta_\delta(r)\,dr= \phi(t-\delta)\le  \phi(t).
\end{eqnarray*}
Also, since~$\phi'$ is monotone non-decreasing,
\begin{eqnarray*}
\phi'_\delta(t)=\int_{-\delta}^\delta \phi'(t-r-2\delta)\eta_\delta(r)\,dr\le 
\int_{-\delta}^\delta \phi'(t-\delta)\eta_\delta(r)\,dr= \phi'(t-\delta)\le  \phi'(t).
\end{eqnarray*}
These considerations complete the proof.
\end{proof}

\subsection{Completion of the proof of Theorem~\ref{THM:LINFTY}}
With the work done so far, we
can now complete the proof of the boundness of non-negative weak solutions. 

\begin{proof}[Proof of Theorem~\ref{THM:LINFTY}]
By Proposition~\ref{regularity convol}, we know that~$a := I_\alpha * F(u) \in L^\infty(\mathbb{R}^N)$.
Thus, $u$ satisfies the following non-autonomous problem with a local nonlinearity
\[
\mathcal{L}_\mu u + \mu u = a(x) f(u)\quad \text{in } \mathbb{R}^N,
\]
with~$a$ bounded. 

In particular, we can write
\begin{equation}\label{cnwejkh3725fguescbj}
\mathcal{L}_\mu u = g(x, u) := -\mu u + a(x) f(u) \quad \text{in } \mathbb{R}^N,
\end{equation}
where, in light of~$(F_2)$,
\[
|g(x, t)| \leq \mu |t| + C \|a\|_\infty \left( |t|^{\frac{\alpha}{N}} + |t|^{\frac{\alpha + 2s_\sharp}{N - 2s_\sharp}} \right).
\]

Now we set
\begin{equation}\label{csfeuwytuigejlsw765}
\gamma := \max\left\{ 1, \frac{\alpha + 2s_\sharp}{N - 2s_\sharp} \right\} \in [1, 2^*_{s_\sharp}-1)
\end{equation}
and we point out that
\begin{equation}\label{csfeuwytuigejlsw7652}
|g(x, t)| \leq C(1 + |t|^\gamma).
\end{equation}

Let~$\beta > 1$, $T > 1$, and~$\phi$ be as in~\eqref{defofut98jkgfmnv9876}.
Thanks to Lemma~\ref{lemma:p488approx},
we can find a convex
function~$\phi_\delta\in C^\infty(\R)$ 
that satisfies the assumptions
of Proposition~\ref{lem:kato}.

Accordingly, by Proposition~\ref{lem:kato},
if~$v\in\H$ then also~$\phi_\delta(v)\in\H $
and, for any nonnegative function~$\varphi\in\H$,
\begin{equation}\label{12345qwerxcv5tr6y3uktfkejwjiuh}\begin{split}&
\int_{[0,1]}\left(\iint_{\R^{2N}}\frac{\big(\phi_\delta(u(x))-\phi_\delta(u(y)) \big)(\varphi(x)-\varphi(y))}{|x-y|^{N+2s}}\,dx\,dy
\right)\,d\mu(s)\\&\quad\le \int_{[0,1]}\left(\iint_{\R^{2N}}\frac{(u(x)-u(y))\big(\phi'_\delta(u(x))\varphi(x)
-\phi'_\delta(u(y))\varphi(y) \big)}{|x-y|^{N+2s}}\,dx\,dy
\right)\,d\mu(s).\end{split}
\end{equation}

Now, by the Sobolev embedding (see e.g.~\cite[Theorem~1.1]{CT2004}),
we have that, for all~$s\in[s_\sharp, \overline{s}_\sharp]$,
$$ \|\phi_\delta(u)\|_{2^*_{s}}^2 
\le C \,C_{N,s}\iint_{\R^{2N}}\frac{\big(\phi_\delta(u(x))-\phi_\delta(u(y)) \big)^2}{|x-y|^{N+2s}}\,dx\,dy, $$ where~$C>0$ depends only on~$N$.

Therefore, thanks to the assumption\footnote{We point out that if we knew the stronger assumption \label{foriuermnv876543}
that~$\mu(\{s_\sharp\})>0$, then~\eqref{76546tJJJbvncmxiuw54398438yt} would simplify to
\begin{equation*}\begin{split}
\|\phi_\delta(u)\|_{2^*_{s_\sharp}}^2& \le 
C \iint_{\R^{2N}}\frac{\big(\phi_\delta(u(x))-\phi_\delta(u(y)) \big)^2}{|x-y|^{N+2s_\sharp}}\,dx\,dy
\\&
\le \frac{C}{ \mu( \{s_\sharp\})}
\int_{[0,1]}C_{N,s}
\left(\iint_{\R^{2N}}\frac{\big(\phi_\delta(u(x))-\phi_\delta(u(y)) \big)^2}{|x-y|^{N+2s}}\,dx\,dy\right)\,d\mu(s).\end{split}\end{equation*}
Therefore, from now on the proof of Theorem~\ref{THM:LINFTY} would carry over
essentially as the ones of~\cite[Proposition~2.2]{Barrios} or~\cite[Proposition~5.1.1]{DMV}}
in~\eqref{ass9567daskjytrTRYJ},
\begin{equation}\label{76546tJJJbvncmxiuw54398438yt}\begin{split}
&\inf_{s\in [s_\sharp, \overline{s}_\sharp]}
\|\phi_\delta(u)\|_{2^*_{s}}^2 \le 
C \inf_{s\in [s_\sharp, \overline{s}_\sharp]} C_{N,s}
 \iint_{\R^{2N}}\frac{\big(\phi_\delta(u(x))-\phi_\delta(u(y)) \big)^2}{|x-y|^{N+2s}}\,dx\,dy
\\&\qquad
\le \frac{C}{ \mu( [s_\sharp, \overline{s}_\sharp])}
\int_{[0,1]}C_{N,s}
\left(\iint_{\R^{2N}}\frac{\big(\phi_\delta(u(x))-\phi_\delta(u(y)) \big)^2}{|x-y|^{N+2s}}\,dx\,dy\right)\,d\mu(s),\\
\\&\qquad
\le C
\int_{[0,1]}C_{N,s}
\left(\iint_{\R^{2N}}\frac{\big(\phi_\delta(u(x))-\phi_\delta(u(y)) \big)^2}{|x-y|^{N+2s}}\,dx\,dy\right)\,d\mu(s),
\end{split}\end{equation} up to renaming~$C$, that now depends on~$N$,
$s_\sharp$, $\overline{s}_\sharp$, and~$\mu$.

As a consequence, using~\eqref{12345qwerxcv5tr6y3uktfkejwjiuh}
with~$\varphi:=\phi_\delta (u)$, we conclude that
\begin{eqnarray*}&&\inf_{ s\in [s_\sharp, \overline{s}_\sharp]}
\|\phi_\delta(u)\|_{2^*_{s}}^2\\&&\;
\leq C\int_{[0,1]}C_{N,s}\left(
\iint_{\R^{2N}}\frac{(u(x)-u(y))\big(\phi'_\delta(u(x))\phi_\delta(x)
-\phi'_\delta(u(y))\phi_\delta(y) \big)}{|x-y|^{N+2s}}\,dx\,dy
\right)\,d\mu(s)
.\end{eqnarray*}
We now use the fact that~$u$ is a weak solution of~\eqref{cnwejkh3725fguescbj} and we obtain that
$$\inf_{s\in [s_\sharp, \overline{s}_\sharp]}
\|\phi_\delta(u)\|_{2^*_{s}}^2\leq C
\int_{\R^{N}} g(x,u(x))\,\phi'_\delta(u(x))\phi_\delta(u(x))\,dx.$$
Thus, recalling~\eqref{csfeuwytuigejlsw765} and~\eqref{csfeuwytuigejlsw7652},
$$\inf_{s\in [s_\sharp, \overline{s}_\sharp]}
\|\phi_\delta(u)\|_{2^*_{s}}^2\leq C
\int_{\R^{N}} \left(1+u^{2^*_{s_\sharp}-1}(x)\right)\,\phi'_\delta(u(x))\phi_\delta(u(x))\,dx.$$


Now, in light of Lemma~\ref{lemma:p488approx}, we know that~$0\leq\phi_\delta\leq\phi$
and~$0\leq\phi_\delta'\leq\phi'$. Therefore,
\begin{equation}\label{ncxmueiwrhfsdhvh7654300}
\inf_{s\in [s_\sharp, \overline{s}_\sharp]}
\|\phi_\delta(u)\|_{2^*_{s}}^2\leq C
\int_{\R^{N}} \left(1+u^{2^*_{s_\sharp}-1}(x)\right)\,\phi' (u(x))\phi(u(x))\,dx.
\end{equation}

We claim that
\begin{equation}\label{ncxmueiwrhfsdhvh76543}
\inf_{s\in [s_\sharp, \overline{s}_\sharp]}
\|\phi(u)\|_{2^*_{s}}^2
\le
\lim_{\delta\searrow0}
\inf_{s\in [s_\sharp, \overline{s}_\sharp]}
\|\phi_\delta(u)\|_{2^*_{s}}^2
.\end{equation}
Indeed, we know that for all~$j\in\mathbb{N}$ there exists~$s_j\in [s_\sharp, \overline{s}_\sharp]$ such that
$$ 
\|\phi_\delta(u)\|_{2^*_{s_j}}^2
-\frac1j\leq \inf_{s\in [s_\sharp, \overline{s}_\sharp]}
\|\phi_\delta(u)\|_{2^*_{s}}^2.$$
Hence,
from the pointwise convergence of~$\phi_\delta$ to~$\phi$ (recall Lemma~\ref{lemma:p488approx})
and the Fatou's Lemma, we conclude that
\begin{eqnarray*}&& \lim_{\delta\searrow0}\inf_{s\in [s_\sharp, \overline{s}_\sharp]}
\|\phi_\delta(u)\|_{2^*_{s}}^2\ge 
\lim_{\delta\searrow0}\|\phi_\delta(u)\|_{2^*_{s_j}}^2
-\frac1j\ge
\|\phi(u)\|_{2^*_{s_j}}^2
-\frac1j
\ge \inf_{s\in [s_\sharp, \overline{s}_\sharp]}
\|\phi(u)\|_{2^*_{s}}^2-\frac1j.
\end{eqnarray*}
Sending~$j\to+\infty$ gives the claim in~\eqref{ncxmueiwrhfsdhvh76543}.

From~\eqref{ncxmueiwrhfsdhvh7654300} and~\eqref{ncxmueiwrhfsdhvh76543},
we infer that
$$\inf_{s\in [s_\sharp, \overline{s}_\sharp]}
\|\phi(u)\|_{2^*_{s}}^2\leq C
\int_{\R^{N}} \left(1+u^{2^*_{s_\sharp}-1}(x)\right)\,\phi' (u(x))\phi(u(x))\,dx.$$

Hence, using the properties~$\phi'(u)\phi(u) \le \beta u^{2\beta-1}$ and $u\phi'(u) \le \beta \phi(u)$, we have
\begin{equation}
\label{eq:5.1.3}\begin{split}\inf_{s\in [s_\sharp, \overline{s}_\sharp]}
\left( \int_{\mathbb{R}^N} (\phi(u))^{2^*_{  s }} \, dx \right)^{2/2^*_{ s }}& \le C\beta\left( \int_{\mathbb{R}^N} u^{2\beta-1} \, dx + \int_{\mathbb{R}^N} (\phi(u))^2 u^{2^*_{s_\sharp}-2} \, dx \right)
.\end{split}
\end{equation}
Notice that the last integral is finite for every $T\ge1$, since
\begin{equation}\label{nbvcx58934fusdk3edfv}\begin{split}&
\int_{\mathbb{R}^N} (\phi(u))^2 u^{2^*_{s_\sharp}-2} \, dx
=\int_{\{u<T\}} u^{2\beta} u^{2^*_{s_\sharp}-2} \, dx
+\beta^2 T^{2\beta-2}\int_{\{u\geq T\}} u^2 u^{2^*_{s_\sharp}-2} \, dx\\&\qquad\leq
T^{2\beta-2}\int_{\{u<T\}} u^{2^*_{s_\sharp}} \, dx
+\beta^2 T^{2\beta-2}\int_{\{u\geq T\}}u^{2^*_{s_\sharp}} \, dx<+\infty.
\end{split}\end{equation}

We now use~\eqref{eq:5.1.3} with~$\beta:=\beta_1$, where
\begin{equation}\label{beayujd31}\beta_{1} := \frac{2^*_{s_\sharp} + 1}{2}>1
.\end{equation}
In this way,
\begin{equation}
\label{pqecm230n4ymp76ik203rptghln}
\begin{split}\inf_{s\in [s_\sharp, \overline{s}_\sharp]}
\left( \int_{\mathbb{R}^N} (\phi(u))^{2^*_{  s }} \, dx \right)^{2/2^*_{ s }}& \le C\beta_{1}\left( \int_{\mathbb{R}^N} u^{2^*_{s_\sharp}} \, dx + \int_{\mathbb{R}^N} (\phi(u))^2 u^{2^*_{s_\sharp}-2} \, dx \right)
.\end{split}
\end{equation}
From this and~\eqref{nbvcx58934fusdk3edfv} it follows that
\begin{equation*}
\inf_{s\in [s_\sharp, \overline{s}_\sharp]}
\left( \int_{\mathbb{R}^N} (\phi(u))^{2^*_{  s }} \, dx \right)^{2/2^*_{ s }}<+\infty.
\end{equation*}

Let now~$s_j\in[s_\sharp, \overline{s}_\sharp]$ be a minimizing sequence for the term on the left-hand side
of~\eqref{eq:5.1.3}, say such that
\begin{equation}\label{qoeuirj2ptyh5y}
\inf_{s\in [s_\sharp, \overline{s}_\sharp]}
\left( \int_{\mathbb{R}^N} (\phi(u))^{2^*_{  s }} \, dx \right)^{2/2^*_{ s }}\ge
\left( \int_{\mathbb{R}^N} (\phi(u))^{2^*_{  s_j }} \, dx \right)^{2/2^*_{ s_j }}-\frac1j.
\end{equation}

Also, let~$R>1$ and use the H\"older inequality with exponents~$r:= 2^*_{s_j }/2$
and~$ r':=2^*_{ s_j}/(2^*_{s_j}-2)$ to see that
\begin{equation}\label{oq0852dh293bv5n}
\begin{split}
&\int_{\mathbb{R}^N} (\phi(u))^2 u^{2^*_{s_\sharp}-2} \, dx
=\int_{\{u\le R\}} (\phi(u))^2 u^{2^*_{s_\sharp}-2} \, dx
+\int_{\{u> R\}} (\phi(u))^2 u^{2^*_{s_\sharp}-2} \, dx\\&\qquad\leq
R^{2^*_{s_\sharp}-1}
\int_{\{u\le R\}}\frac{ (\phi(u))^2 }{u} \, dx
+\left(\int_{\{u> R\}} (\phi(u))^{2^*_{ s_j}} \, dx\right)^{2/2^*_{ s_j }}
\left(\int_{\{u> R\}}  
u^{ \frac{2^*_{s_\sharp}s_\sharp}{s_j}} \, dx\right)^{ (2^*_{s_j }-2)/2^*_{ s_j} }.
\end{split}\end{equation}

In addition, if~$u> R>1$ then~$u^{ \frac{2^*_{s_\sharp}s_\sharp}{s_j}}\le
u^{2^*_{s_\sharp}}$ and therefore
$$ \int_{\{u> R\}}  
u^{ \frac{2^*_{s_\sharp}s_\sharp}{s_j}} \, dx\le\int_{\{u> R\}}  
u^{ 2^*_{s_\sharp}} \, dx.$$
Since the latter quantity is the tail of a convergent integral, we can suppose that
it is smaller than~$1$,
provided~$R$ is large enough, and therefore,
since~$(2^*_{s_j }-2)/2^*_{ s_j}\ge(2^*_{s_\sharp }-2)/2^*_{ s_\sharp}$,
$$\left(\int_{\{u> R\}}  
u^{ \frac{2^*_{s_\sharp}s_\sharp}{s_j}} \, dx\right)^{ (2^*_{s_j }-2)/2^*_{ s_j} }\le
\left(\int_{\{u> R\}}  
u^{2^*_{s_\sharp}} \, dx\right)^{ (2^*_{s_j }-2)/2^*_{ s_j} }
\le\left(\int_{\{u> R\}}  
u^{2^*_{s_\sharp}} \, dx\right)^{ (2^*_{s_\sharp }-2)/2^*_{ s_\sharp} }.$$
We now take~$R$ sufficiently large such that the latter quantity is less than~$\frac1{2C\beta_1}$ and we plug this information into~\eqref{oq0852dh293bv5n},
finding that
$$\int_{\mathbb{R}^N} (\phi(u))^2 u^{2^*_{s_\sharp}-2} \, dx\leq
R^{2^*_{s_\sharp}-1}
\int_{\{u\le R\}}\frac{ (\phi(u))^2 }{u} \, dx
+
\frac1{2C\beta_1}\left(\int_{\{u> R\}} (\phi(u))^{2^*_{ s_j}} \, dx\right)^{2/2^*_{ s_j }}.$$

We insert this into~\eqref{pqecm230n4ymp76ik203rptghln} and we conclude that
\begin{equation*}
\begin{split}&\inf_{s\in [s_\sharp, \overline{s}_\sharp]}
\left( \int_{\mathbb{R}^N} (\phi(u))^{2^*_{  s }} \, dx \right)^{2/2^*_{ s }}\\&\qquad \le C\beta_{1}\int_{\mathbb{R}^N} u^{2^*_{s_\sharp}} \, dx + 
C\beta_1R^{2^*_{s_\sharp}-1}
\int_{\{u\le R\}}\frac{ (\phi(u))^2 }{u} \, dx
+
\frac1{2}\left(\int_{\{u> R\}} (\phi(u))^{2^*_{ s_j}} \, dx\right)^{2/2^*_{ s_j }}
.\end{split}
\end{equation*}

We now use~\eqref{qoeuirj2ptyh5y} to deduce that
\begin{equation*}
\begin{split}&\inf_{s\in [s_\sharp, \overline{s}_\sharp]}
\left( \int_{\mathbb{R}^N} (\phi(u))^{2^*_{  s }} \, dx \right)^{2/2^*_{ s }}\\&\qquad \le C\beta_{1}\int_{\mathbb{R}^N} u^{2^*_{s_\sharp}} \, dx + 
C\beta_1R^{2^*_{s_\sharp}-1}
\int_{\{u\le R\}}\frac{ (\phi(u))^2 }{u} \, dx
+
\frac1{2}\inf_{s\in [s_\sharp, \overline{s}_\sharp]}
\left( \int_{\mathbb{R}^N} (\phi(u))^{2^*_{  s }} \, dx \right)^{2/2^*_{ s }}+\frac1{2j}.
\end{split}
\end{equation*}
Hence, reabsorbing a term to the left-hand side,
\begin{equation*}
\begin{split}&\frac12\inf_{s\in [s_\sharp, \overline{s}_\sharp]}
\left( \int_{\mathbb{R}^N} (\phi(u))^{2^*_{  s }} \, dx \right)^{2/2^*_{ s }}\\&\qquad \le C\beta_{1}\int_{\mathbb{R}^N} u^{2^*_{s_\sharp}} \, dx + 
C\beta_1R^{2^*_{s_\sharp}-1}
\int_{\{u\le R\}}\frac{ (\phi(u))^2 }{u} \, dx
+\frac1{2j}.
\end{split}
\end{equation*}
As a result, sending~$j\to+\infty$,
\begin{equation*}
\begin{split}&\frac12\inf_{s\in [s_\sharp, \overline{s}_\sharp]}
\left( \int_{\mathbb{R}^N} (\phi(u))^{2^*_{  s }} \, dx \right)^{2/2^*_{ s }} \le C\beta_{1}\int_{\mathbb{R}^N} u^{2^*_{s_\sharp}} \, dx + 
C\beta_1R^{2^*_{s_\sharp}-1}
\int_{\{u\le R\}}\frac{ (\phi(u))^2 }{u} \, dx.
\end{split}
\end{equation*}

Consequently, since~$\phi(u)\le u^\beta$ (and then, in this case,
by~\eqref{beayujd31}, $\phi(u)\le u^{\frac{2^*_{s_\sharp} + 1}{2}}$), we find that
\begin{equation}\label{sjoqwlneVISKNDpqwXt34-a}
\begin{split}&\frac12\inf_{s\in [s_\sharp, \overline{s}_\sharp]}
\left( \int_{\mathbb{R}^N} (\phi(u))^{2^*_{  s }} \, dx \right)^{2/2^*_{ s }} \le C\beta_{1}\int_{\mathbb{R}^N} u^{2^*_{s_\sharp}} \, dx + 
C\beta_1R^{2^*_{s_\sharp}-1}
\int_{\R^N}u^{2^*_{s_\sharp}}\, dx.
\end{split}
\end{equation}

We now claim that, for all~$\beta>1$,
\begin{equation}\label{vcxkueriusdfj78043}
\lim_{T\to+\infty}\inf_{s\in [s_\sharp, \overline{s}_\sharp]}
\left( \int_{\mathbb{R}^N} (\phi(u))^{2^*_{  s }} \, dx \right)^{2/2^*_{ s }}\ge
\inf_{s\in [s_\sharp, \overline{s}_\sharp]}\left( \int_{\mathbb{R}^N} u^{2^*_{  s }\beta} \, dx \right)^{2/2^*_{ s}}.\end{equation}
Indeed, we recall that
$$ \lim_{T\to+\infty}\phi(u)=u^{\beta}$$
and therefore, by Fatou's Lemma,
\begin{eqnarray*}
&&\lim_{T\to+\infty}\inf_{s\in [s_\sharp, \overline{s}_\sharp]}
\left( \int_{\mathbb{R}^N} (\phi(u))^{2^*_{  s }} \, dx \right)^{2/2^*_{ s }}\ge
\lim_{T\to+\infty}
\left( \int_{\mathbb{R}^N} (\phi(u))^{2^*_{  s_j }} \, dx \right)^{2/2^*_{ s_j }}-\frac1j\\
&&\qquad\ge\left( \int_{\mathbb{R}^N} u^{2^*_{  s_j }\beta} \, dx \right)^{2/2^*_{ s_j }}-\frac1j\ge
\inf_{s\in [s_\sharp, \overline{s}_\sharp]}\left( \int_{\mathbb{R}^N} u^{2^*_{  s }\beta} \, dx \right)^{2/2^*_{ s}}-\frac1j.
\end{eqnarray*}
Sending~$j\to\infty$ we thus obtain{vcxkueriusdfj78043}.

The observation in~\eqref{vcxkueriusdfj78043} (used with~$\beta:=\beta_1$), in tandem with~\eqref{sjoqwlneVISKNDpqwXt34-a}, yields that
\begin{equation*}
\begin{split}&\inf_{s\in [s_\sharp, \overline{s}_\sharp]}\left( \int_{\mathbb{R}^N} u^{2^*_{  s }\beta_1} \, dx \right)^{2/2^*_{ s}} \le 2C\beta_{1}\int_{\mathbb{R}^N} u^{2^*_{s_\sharp}} \, dx + 2C\beta_1R^{2^*_{s_\sharp}-1}
\int_{\R^N}u^{2^*_{s_\sharp}}\, dx<+\infty.
\end{split}
\end{equation*}

Let now~$\beta>\beta_1$ in~\eqref{eq:5.1.3}.
Using the fact that~$\phi(u)\le u^\beta$ and sending~$T\to+\infty$ (recall~\eqref{vcxkueriusdfj78043}), we obtain that
\begin{equation}\label{r438whasjbhszmSDFGHJ000}
\inf_{s\in [s_\sharp, \overline{s}_\sharp]}
\left( \int_{\mathbb{R}^N} u^{2^*_{  s }\beta} \, dx \right)^{2/2^*_{ s }} \le C\beta\left( \int_{\mathbb{R}^N} u^{2\beta-1} \, dx + \int_{\mathbb{R}^N}  u^{2\beta+2^*_{s_\sharp}-2} \, dx \right)
.\end{equation}

Furthermore, we can write $u^{2\beta-1} = u^a u^b$ with 
\[
a := \frac{2^*_{s_\sharp}(2^*_{s_\sharp}-1)}{2(\beta-1)}  \qquad \text{and} \qquad b := 2\beta - 1 - a.
\]
Notice that, since $\beta > \beta_1$, then $0 < a, b < 2^*_{s_\sharp}$. Hence, applying the
Young's inequality with exponents $r := 2^*_{s_\sharp}/a$ and $r' := 2^*_{s_\sharp}/(2^*_{s_\sharp} - a)$,
we find that
\begin{eqnarray*}
\int_{\mathbb{R}^N} u^{2\beta-1} \, dx &\le& \frac{a}{2^*_{s_\sharp}} \int_{\mathbb{R}^N} u^{2^*_{s_\sharp}} \, dx + \frac{2^*_{s_\sharp}-a}{2^*_{s_\sharp}} \int_{\mathbb{R}^N} u^{\frac{2^*_{s_\sharp} b}{2^*_{s_\sharp}-a}} \, dx \\
&\le&  \int_{\mathbb{R}^N} u^{2^*_{s_\sharp}} \, dx + \int_{\mathbb{R}^N} u^{2\beta + 2^*_{s_\sharp} - 2} \, dx \\
&\le &
C \left( 1 + \int_{\mathbb{R}^N} u^{2\beta + 2^*_{s_\sharp} - 2} \, dx \right),
\end{eqnarray*}
with $C > 0$ independent of $\beta$. 

Plugging this into \eqref{r438whasjbhszmSDFGHJ000}, we infer that
\begin{equation*}
\inf_{s\in [s_\sharp, \overline{s}_\sharp]}
\left( \int_{\mathbb{R}^N} u^{2^*_{  s }\beta} \, dx \right)^{2/2^*_{ s }} \le 
C\beta \left( 1 + \int_{\mathbb{R}^N} u^{2\beta + 2^*_{s_\sharp} - 2} \, dx \right),
\end{equation*}
with $C$ changing from line to line, but remaining independent of~$\beta$.
Therefore, there exists~$s_\beta\in [s_\sharp, \overline{s}_\sharp]$ such that
\begin{equation*}
\left( \int_{\mathbb{R}^N} u^{2^*_{  s_0 }\beta} \, dx \right)^{2/2^*_{ s_0 }} \le 
C\beta \left( 1 + \int_{\mathbb{R}^N} u^{2\beta + 2^*_{s_\sharp} - 2} \, dx \right),
\end{equation*} up to renaming~$C$ once more.

As a consequence, for all~$\beta>\beta_1$,
\begin{equation}\label{5.1.8} 
\left( 
1 + \int_{\mathbb{R}^N} u^{2^*_{s_\beta} \beta} \, dx \right)^{\frac{1}{2^*_{s_\beta }(\beta-1)}} \le (C\beta)^{\frac{1}{2(\beta-1)}} \left( 1 + \int_{\mathbb{R}^N} u^{2\beta + 2^*_{s_\sharp} - 2} \, dx \right)^{\frac{1}{2(\beta-1)}}.
\end{equation}

We now perform an iterative argument, 
by defining a sequence~$\beta_{m}$ such that~$\beta_1$ is as in~\eqref{beayujd31} and, for all~$m\ge1$,
\begin{equation*}
2\beta_{m+1} + 2^*_{s_\sharp} - 2 = 2^*_{s_{\beta_m}} \beta_m.
\end{equation*}
In this way,
$$ 2\beta_{m+1} + 2^*_{s_\sharp} - 2 \ge 2^*_{s_\sharp} \beta_m$$and therefore
\begin{equation}\label{APIskjqw-0pri34-p5y56y}
\beta_{m+1} - 1 \ge \left( \frac{2^*_{s_\sharp}}{2} \right)^m (\beta_1 - 1).
\end{equation}

With this notation, we use~\eqref{5.1.8} with~$\beta:=\beta_{m+1}$ and we see that
\begin{equation}\label{qowufj340yujh6py.e3r4t}
\left( 
1 + \int_{\mathbb{R}^N} u^{2^*_{s_{\beta_{m+1}}} \beta_{m+1}} \, dx\right)^{\frac{1}{2^*_{s_{\beta_{m+1}}}(\beta_{m+1}-1)}} 
\le (C\beta_{m+1})^{\frac{1}{2(\beta_{m+1}-1)}} \left( 1 + \int_{\mathbb{R}^N} u^{2^*_{s_{\beta_m}} \beta_m} \, dx \right)^{\frac{1}{2(\beta_{m+1}-1)}}.
\end{equation}

It is also useful to observe that
$$ 2(\beta_{m+1}-1)= 2^*_{s_{\beta_m}} \beta_m-2^*_{s_\sharp}
\ge2^*_{s_{\beta_m}} \beta_m-2^*_{s_{\beta_m}}=2^*_{s_{\beta_m}}( \beta_m-1)
$$
and thus
$$\frac1{2(\beta_{m+1}-1)}\le\frac1{2^*_{s_{\beta_m}}( \beta_m-1)}.$$

On this account,
$$ \left( 1 + \int_{\mathbb{R}^N} u^{2^*_{s_{\beta_m}} \beta_m} \, dx \right)^{\frac{1}{2(\beta_{m+1}-1)}}\le
 \left( 1 + \int_{\mathbb{R}^N} u^{2^*_{s_{\beta_m}} \beta_m} \, dx \right)^{\frac{1}{2^*_{s_{\beta_m}}( \beta_m-1)}}.$$
By combining this inequality and~\eqref{qowufj340yujh6py.e3r4t}, we conclude that
\begin{equation*}
\left( 
1 + \int_{\mathbb{R}^N} u^{2^*_{s_{\beta_{m+1}}} \beta_{m+1}} \, dx\right)^{\frac{1}{2^*_{s_{\beta_{m+1}}}(\beta_{m+1}-1)}} 
\le (C\beta_{m+1})^{\frac{1}{2(\beta_{m+1}-1)}} \left( 1 + \int_{\mathbb{R}^N} u^{2^*_{s_{\beta_m}} \beta_m} \, dx \right)^{\frac{1}{2^*_{s_{\beta_m}}( \beta_m-1)}}.
\end{equation*}

Hence, defining $C_{m+1} := C\beta_{m+1}$ and
\begin{equation*}
A_m := \left( 1 + \int_{\mathbb{R}^N} u^{2^*_{s_{\beta_m}} \beta_m} \, dx \right)^{\frac{1}{2^*_{s_{\beta_m}}( \beta_m-1)}},
\end{equation*}
we have that
$$ A_{m+1}\le C_{m+1}^{\frac{1}{2(\beta_{m+1}-1)}}\,A_m.$$
We can thereby iterate this inequality and obtain that 
\begin{equation}\label{0qoudjweoti0-45pui-04qvn65o}
A_{m+1} \le \prod_{k=2}^{m+1} C_k^{\frac{1}{2(\beta_k-1)}} A_1 .
\end{equation}

We also remark that, in light of~\eqref{APIskjqw-0pri34-p5y56y},
there exists~$m_0\ge2$ such that, for all~$m\ge m_0$,
$$
\beta_{m} - 1 \ge \frac{\beta_m}2.
$$
As a result,
$$\sum_{k=m_0}^m\frac{\ln(C\beta_k)}{2(\beta_k-1)}
\le\sum_{k=m_0}^m\frac{\ln(C\beta_k)}{\beta_k}\le\sum_{k=m_0}^{+\infty}\frac{\ln(C\beta_k)}{\beta_k}=:C',
$$and we stress that~$C'$ is finite, thanks to~\eqref{APIskjqw-0pri34-p5y56y}.

On this account,
$$\sum_{k=2}^m\frac{\ln(C\beta_k)}{2(\beta_k-1)}\le C'',$$
for some~$C''\in(0,+\infty)$ and thus
\begin{eqnarray*}
\prod_{k=2}^{m+1} C_k^{\frac{1}{2(\beta_k-1)}} =
\prod_{k=2}^{m+1} (C\beta_k)^{\frac{1}{2(\beta_k-1)}} =
\prod_{k=2}^{m+1} \exp\left(\frac{\ln(C\beta_k)}{2(\beta_k-1)}\right)=
\exp\left(\sum_{k=2}^{m+1} \frac{\ln(C\beta_k)}{2(\beta_k-1)}\right)\le e^{C''}=:C_0,
\end{eqnarray*}
with~$C_0\in(0,+\infty)$.

This and~\eqref{0qoudjweoti0-45pui-04qvn65o} give that
$$ A_{m+1} \le C_0 A_1=:C_1 < +\infty.$$
In particular,
\begin{equation}\label{0wepiou23p4oyi-=65koplo.iy} \left( \int_{\mathbb{R}^N} u^{2^*_{s_{\beta_{m+1}}} \beta_{m+1}} \, dx \right)^{\frac{1}{2^*_{s_{\beta_{m+1}}}( \beta_{m+1}-1)}}\le C_1.\end{equation}

We also stress that~$2^*_{s_{\beta_{m+1}}}\in[2^*_{s_\sharp},2^*_{\overline{s}_\sharp}]$ and therefore, in view of~\eqref{APIskjqw-0pri34-p5y56y},
$$ \lim_{m\to+\infty}2^*_{s_{\beta_{m+1}}} \beta_{m+1}=+\infty.$$
This allows us to pass to the limit as~$m\to+\infty$ in~\eqref{0wepiou23p4oyi-=65koplo.iy} and obtain that~$\|u\|_{L^\infty(\R^N)}\le C_1$,
as desired.
\end{proof}

\section{The Pohozaev identity and proof of Theorem \ref{pohozaev theorem}} 
\label{sec_Pohozaev}

The goal of this section is to establish Theorem~\ref{pohozaev theorem}.
For this, we recall the following useful notation and statements.

For any~$ s\in\R$ and any vector field~$X \in C^1_0 (\R^N,\R^N)$, 
we define, for $x, y\in \R^N$, with~$x \neq y$, 
\begin{equation}\label{DFGHJbcxm84thdkslkfi7uit}
{\mathcal{K}_{X}^s(x,y)}:= \frac{\big(\mathrm{div} (X)\big)(x) + \big(\mathrm{div} (X)\big)(y)}{2}- {\frac{N+2s}{2}}\frac{(X(x)-X(y))\cdot(x-y)}{|x-y|^2}.\end{equation}


\begin{proposition}\cite[Proposition 6.5]{CGT1}\label{Pohozaev right term}
Let $\alpha \in (0,N)$ and let $H \in \operatorname{Lip}_{\mathrm{loc}}(\mathbb{R}^N) \cap L^\infty(\mathbb{R}^N)$ be such that
\[
(I_\alpha \ast |H|)\,|H| \in L^1(\mathbb{R}^N) \qquad \text{and} \qquad (I_\alpha \ast |H|)\,|\nabla H| \in L^1_{\mathrm{loc}}(\mathbb{R}^N).
\]
Let $X \in C_0^1(\mathbb{R}^N, \mathbb{R}^N)$.

Then,
\begin{equation*}
\iint_{\mathbb{R}^{2N}} I_\alpha(x - y)\, H(x)\, H(y)\, \mathcal{K}^{-\frac{\alpha}{2}}_X(x,y)\, dx\, dy = -\int_{\mathbb{R}^N} (I_\alpha \ast H)(x)\, \nabla H(x) \cdot X(x)\, dx.
\end{equation*}
\end{proposition}


We also discuss
the analytical properties of the normalization constant~$C_{N,s}$, defined in~\eqref{constant_definition}:

\begin{lemma}\label{rem:normalization_constant_analysis}
$C_{N,s}$ is a smooth, non-negative and uniformly bounded function for all~$s \in [0, 1]$.
\end{lemma}

\begin{proof}
We recall from~\eqref{constant_definition} that \begin{equation*}
    C_{N,s} = \frac{2^{2s}s \Gamma \left(\frac{N+2s}{2}\right)}{\pi^{N/2} \Gamma(1-s)}.
\end{equation*}

The Euler Gamma function $\Gamma (z)$ is  holomorphic in the right half-plane $\{\text{Re}(z) > 0\}$, with simple poles located at non-positive integers~$z \in \{0, -1, -2, \dots\}$.

We observe that, for all~$s \in [0, 1]$, 
$$ \frac{N+2s}{2}\in\left[\frac{N}{2}, \frac{N+2}{2}\right],$$
which is strictly contained in the region of holomorphy. 

Moreover, when~$s$ is close to~$1$,
$$ \Gamma(1-s)=\frac1{1-s}-\gamma+O(1-s),$$ being~$\gamma$ the Euler-Mascheroni constant.

These considerations complete the proof of Lemma~\ref{rem:normalization_constant_analysis}.
\end{proof}

The following result is pivotal towards the proof of the desired Pohozaev equality
in Theorem~\ref{pohozaev theorem}.

\begin{proposition}
\label{pohozaev superposition}
Assume that~$\mu_-=0$. 
Let $u \in \H $ be such that, for all~$R>0$,
\begin{equation}\label{bcwiytr84y8tyihsWYU8765}
 {\iiint_{[0,1]\times B_R\times\R^N}} C_{N,s}\frac{|2u(x)-u(x+y)-u(x-y)|}{|y|^{N+2s}}\,d\mu(s)\,dx\,dy<+\infty.
\end{equation}

Suppose also that~$u\in W^{1,p}_{loc}(\R^N)$ for some~$p\in(1,+\infty)$ and that
\begin{equation}\label{bcwiytr84y8tyihsWYU876500}
\mathcal{L}_{\mu}u\in L^q_{loc}(\R^N),\end{equation} being~$q$ the conjugate exponent of~$p$, namely~$\frac1p+\frac1q=1$.

Let~$X \in C^1_0 (\R^N,\R^N)$.

Then,
\begin{equation}\label{Pohozaev on superposition operator}  \begin{split}&
\int_{(0,1)}\left(C_{N,s}\iint_{\R^{2N}}
\frac{|u(x)-u(y)|^2}{|x-y|^{N+2s}} \mathcal{K}_{X}^s(x,y)\, dx\, dy\right)\,d\mu(s) 
+ \frac{\mu(\{0\})}2\int_{\R^N} u^2 (x)\,\mathrm{div} X(x)\, dx\\&\qquad+\mu(\{1\})\left(\frac12\int_{\R^N} |\nabla u(x)|^2 \,\mathrm{div} X(x)\, dx
-\int_{\R^N} \partial_j u(x)\partial_i u(x) \partial_j X_i(x)\, dx
\right)\\&\qquad\qquad
= - \int_{\R^N} \mathcal{L}_\mu u(x)\, \nabla u \cdot X\, dx,\end{split}\end{equation}
 where~$\mathcal{K}_{X}^s$ is defined in~\eqref{DFGHJbcxm84thdkslkfi7uit}.
 \end{proposition}
 
Proposition~\ref{pohozaev superposition} may be viewed as an extension of~\cite[Proposition~6.3]{CGT1} to the present setting. We emphasize, however, that the proof given below does not follow from that of~\cite[Proposition~6.3]{CGT1}. Indeed, our regularity assumptions on~$u$ are substantially more general. In fact, the case~$\mu:=\delta_s$ for some~$s\in(0,1)$, which is treated in~\cite[Proposition~6.3]{CGT1}, arises as a direct corollary of Proposition~\ref{pohozaev superposition} here, replacing the assumption that $u \in L^2(\R^N)$ with the summability condition $(2.10)$
in~\cite[Proposition~2.1]{CGT1}.

\begin{remark} We point out that if
\begin{equation*}
\iiint_{[0,1]\times B_R\times\R^N} \left(C_{N,s}\frac{|2u(x)-u(x+y)-u(x-y)|}{|y|^{N+2s}}\right)^q\,d\mu(s)\,dx\,dy<+\infty.
\end{equation*}then both conditions~\eqref{bcwiytr84y8tyihsWYU8765}
and~\eqref{bcwiytr84y8tyihsWYU876500} would be satisfied.
\end{remark}

\begin{proof}[Proof of Proposition~\ref{pohozaev superposition}]
We first prove~\eqref{Pohozaev on superposition operator}
assuming that~$u\in C^\infty_0 (\R^N)$.

For this, we recall the definition in~\eqref{DFGHJbcxm84thdkslkfi7uit} and we use the Dominated Convergence Theorem to see that
\begin{eqnarray*}&&
\int_{(0,1)}\left(C_{N,s} \iint_{\R^{2N}} \frac{|u(x) - u(y)|^2}{|x - y|^{N+2s}} \mathcal{K}^s_X(x, y) \, dx\, dy\right)\, d\mu(s) \\&=&
\int_{(0,1)}\left(C_{N,s}\iint_{\R^{2N}} \frac{|u(x) - u(y)|^2}{|x - y|^{N+2s}}\right.\\ &&\qquad\qquad\left.
\left(\frac{
\big(\mathrm{div}(X)\big)(x)+\big(\mathrm{div}(X)\big)(y)}2-\frac{N+2s}2\frac{(X(x)-X(y))\cdot(x-y)}{|x-y|^2}
\right)\, dx\, dy\right)\, d\mu(s)\\&=&
\int_{(0,1)}\left(C_{N,s}
\lim_{\varepsilon\searrow0}\iint_{{\R^{2N}}\atop{\{|x-y|>\varepsilon\}}} \frac{|u(x) - u(y)|^2}{|x - y|^{N+2s}}\right.\\ &&\qquad\qquad\left.
\left(\frac{
\big(\mathrm{div}(X)\big)(x)+\big(\mathrm{div}(X)\big)(y)}2-\frac{N+2s}2\frac{(X(x)-X(y))\cdot(x-y)}{|x-y|^2}
\right)\, dx\, dy\right)\, d\mu(s)
\\&=&\lim_{\varepsilon\searrow0}
\int_{(0,1)}\left(C_{N,s} \iint_{{\R^{2N}}\atop{\{|x-y|>\varepsilon\}}} \frac{|u(x) - u(y)|^2}{|x - y|^{N+2s}}\right.\\ &&\qquad\qquad\left.
\left(\frac{
\big(\mathrm{div}(X)\big)(x)+\big(\mathrm{div}(X)\big)(y)}2-\frac{N+2s}2\frac{(X(x)-X(y))\cdot(x-y)}{|x-y|^2}
\right)\, dx\, dy\right)\, d\mu(s)
\\&=&
\lim_{\varepsilon\searrow0}
\int_{(0,1)}\left(C_{N,s}\iint_{{\R^{2N}}\atop{\{|x-y|>\varepsilon\}}} 
\frac{|u(x) - u(y)|^2}{|x - y|^{N+2s}}\right.\\ &&\qquad\qquad\left.
\left(\big(\mathrm{div}(X)\big)(x)-(N+2s)\frac{X(x)\cdot(x-y)}{|x-y|^2}
\right)\, dx\, dy\right)\, d\mu(s).
\end{eqnarray*}

We now observe that
\begin{eqnarray*} \mathrm{div}_x\left(\frac{X(x)}{|x-y|^{N+2s}}\right)
= \frac{ \big( \mathrm{div}(X)\big)(x)}{|x-y|^{N+2s}} -(N+2s)\frac{X(x)\cdot (x-y)}{|x-y|^{N+2s+2}}
\end{eqnarray*}
and therefore
\begin{equation}\label{PREWQERT90843tihlkgewsgt9p54}\begin{split}&
\int_{(0,1)}\left(C_{N,s}
\iint_{\R^{2N}} \frac{|u(x) - u(y)|^2}{|x - y|^{N+2s}} \mathcal{K}^s_X(x, y) \, dx\, dy\right)\, d\mu(s) \\&=
\lim_{\varepsilon\searrow0}
\int_{(0,1)}\left(C_{N,s}\iint_{{\R^{2N}}\atop{\{|x-y|>\varepsilon\}}} 
|u(x) - u(y)|^2 \mathrm{div}_x\left(\frac{X(x)}{|x-y|^{N+2s}}\right)\, dx\, dy\right)\, d\mu(s)\\
&=\lim_{\varepsilon\searrow0}
\int_{(0,1)}\left(C_{N,s}\iint_{{\R^{2N}}\atop{\{|x-y|>\varepsilon\}}} \mathrm{div}_x\left(
\frac{|u(x) - u(y)|^2  X(x)}{|x-y|^{N+2s}}\right)\, dx\, dy\right)\, d\mu(s)\\&\qquad
-\int_{(0,1)}\left( 2C_{N,s}\iint_{{\R^{2N}}\atop{\{|x-y|>\varepsilon\}}} 
\frac{(u(x) - u(y))\nabla u(x)\cdot  X(x)}{|x-y|^{N+2s}}\, dx\, dy\right)\, d\mu(s)\\
&=\lim_{\varepsilon\searrow0}
\int_{(0,1)}\left(C_{N,s}\iint_{\R^{N}\times\partial B_\varepsilon(y)}
\frac{|u(x) - u(y)|^2  X(x)\cdot(x-y)}{\varepsilon^{N+2s+1}}\, d{\mathcal{H}}^{n-1}_x\, dy\right)\, d\mu(s)\\&\qquad
-\int_{(0,1)}\left( 2C_{N,s}\iint_{{\R^{2N}}\atop{\{|x-y|>\varepsilon\}}} 
\frac{(u(x) - u(y))\nabla u(x)\cdot  X(x)}{|x-y|^{N+2s}}\, dx\, dy\right)\, d\mu(s)
.\end{split}
\end{equation}

Moreover, for any~$\varepsilon\in(0,1)$, \begin{equation}\label{vnklewr8uy4987ty84wygtuhedwsPOIU}
{\mbox{the map $(s,x,y)\longmapsto
2C_{N,s}\chi_{\R^N\setminus B_\varepsilon(x)}(y)\displaystyle \frac{(u(x) - u(y))\nabla u(x)\cdot  X(x)}{|x-y|^{N+2s}}$ belongs to~$ L^1\big((0,1)\times\R^{2N}\big)$.}}\end{equation}
Indeed, we suppose that the supports of~$u$ and~$X$ are contained in~$B_R$ for some~$R>0$ and we point out that, for all~$x\in B_R$,
\begin{eqnarray*}&&\left|
\int_{\R^N\setminus B_\varepsilon(x)} 
\frac{(u(x) - u(y)) }{|x-y|^{N+2s}}\,dy\right|\le
\int_{\R^N\setminus B_\varepsilon(x)} 
\frac{|u(x) - u(y)|}{|x-y|^{N+2s}}\,dy
\\&&\qquad\le 2\| u\|_{L^\infty(\R^N)} {{\mathcal{H}}^{n-1}(\partial B_1)}\int_\varepsilon^{+\infty}\frac{d\rho}{\rho^{1+2s}}
=\frac{\| u\|_{L^\infty(\R^N)}{\mathcal{H}}^{n-1}(\partial B_1)}{s\varepsilon^{2s}}\le\frac{\| u\|_{L^\infty(\R^N)}{\mathcal{H}}^{n-1}(\partial B_1)}{s\varepsilon^{2}}
.
\end{eqnarray*}
Consequently,
\begin{eqnarray*}&&\left|
\iint_{{\R^{2N}}\atop{\{|x-y|>\varepsilon\}}} 
\frac{(u(x) - u(y))\nabla u(x)\cdot  X(x)}{|x-y|^{N+2s}}\,dx\,dy\right|\\&&\quad\le
\|\nabla u\|_{L^\infty(\R^N,\R^N)} \|X\|_{L^\infty(\R^N,\R^N)}
\iint_{B_R\times (\R^N\setminus B_\varepsilon(x))} 
\frac{|u(x) - u(y)|}{|x-y|^{N+2s}}\,dx\,dy\\&&\quad\le \frac{C R^N
\| u\|^2_{C^1(\R^N)} \|X\|_{L^\infty(\R^N,\R^N)}}{s\varepsilon^{2}},
\end{eqnarray*} for some dimensional constant~$C>0$.

Accordingly, recalling
the definition of~$C_{N,s}$ in~\eqref{constant_definition} and Lemma~\ref{rem:normalization_constant_analysis},
we gather that
\begin{eqnarray*}&&\iiint_{(0,1)\times\R^{2N}}
2C_{N,s}\chi_{\R^N\setminus B_\varepsilon(x)}(y)\displaystyle \frac{(u(x) - u(y))\nabla u(x)\cdot  X(x)}{|x-y|^{N+2s}}\, dx\,dy\,d\mu(s)
\\&&\qquad\le 
\frac{C R^N
\| u\|^2_{C^1(\R^N)} \|X\|_{L^\infty(\R^N,\R^N)}}{\varepsilon^{2}},
\end{eqnarray*} up to renaming~$C$.
This establishes~\eqref{vnklewr8uy4987ty84wygtuhedwsPOIU}.

Hence, thanks to~\eqref{vnklewr8uy4987ty84wygtuhedwsPOIU} we can employ Fubini-Tonelli Theorem and, exploiting also a symmetrization argument and the Dominated Convergence Theorem, we find that
\begin{equation}\label{WQERT90843tihlkgewsgt9p54}\begin{split}&\lim_{\varepsilon\searrow0}
\int_{(0,1)}\left( 2C_{N,s}\iint_{{\R^{2N}}\atop{\{|x-y|>\varepsilon\}} } 
\frac{(u(x) - u(y))\nabla u(x)\cdot  X(x)}{|x-y|^{N+2s}}\, dx\, dy\right)\, d\mu(s)
\\&=\lim_{\varepsilon\searrow0}\int_{\R^N}\left(
\int_{(0,1)}\left( 2C_{N,s}\int_{\R^N\setminus B_\varepsilon(x)} 
\frac{u(x) - u(y)}{|x-y|^{N+2s}}\, dy\right)\, d\mu(s)\right)\nabla u(x)\cdot  X(x)\,dx\\
&= \lim_{\varepsilon\searrow0}\int_{\R^N}\left(
\int_{(0,1)}\left( C_{N,s}\int_{\R^N\setminus B_\varepsilon(x)} 
\frac{2u(x) - u(x+y)-u(x-y)}{|x-y|^{N+2s}}\, dy\right)\, d\mu(s)\right)\nabla u(x)\cdot  X(x)\,dx\\
&=\int_{\R^N}\left(
\int_{(0,1)}\left( C_{N,s}\int_{\R^N } 
\frac{2u(x) - u(x+y)-u(x-y)}{|x-y|^{N+2s}}\, dy\right)\, d\mu(s)\right)\nabla u(x)\cdot  X(x)\,dx
\\&=\int_{\R^N}\left(
\int_{(0,1)}(-\Delta)^s u(x)\, d\mu(s)\right)\nabla u(x)\cdot  X(x)\,dx
.\end{split}
\end{equation}

Furthermore, changing variable~$z:=x-y$, we write\begin{eqnarray*}&&
\iint_{\R^{N}\times\partial B_\varepsilon(y)}
\frac{|u(x) - u(y)|^2  X(x)\cdot(x-y)}{\varepsilon^{N+2s+1}}\, d{\mathcal{H}}^{n-1}_x\, dy
\\&=&\iint_{\R^{N}\times\partial B_\varepsilon}
\frac{|u(y+z) - u(y)|^2  X(y+z)\cdot z}{{\varepsilon^{N+2s+1}}}\, d{\mathcal{H}}^{n-1}_z\, dy\\
&=&\frac12\iint_{\R^{N}\times\partial B_\varepsilon}
\frac{|u(y+z) - u(y)|^2  X(y+z)\cdot z}{{\varepsilon^{N+2s+1}}}\, d{\mathcal{H}}^{n-1}_z\, dy
\\&&\qquad-\frac12\iint_{\R^{N}\times\partial B_\varepsilon}
|u(y-z) - u(y)|^2  X(y-z)\cdot z\, d{\mathcal{H}}^{n-1}_z\, dy\\
&=&\frac12\iint_{\R^{N}\times\partial B_\varepsilon}
\frac{|u(y+z) - u(y)|^2 \big( X(y+z)-X(y-z)\big)\cdot z}{{\varepsilon^{N+2s+1}}}\, d{\mathcal{H}}^{n-1}_z\, dy
\\&&\qquad-\frac12\iint_{\R^{N}\times\partial B_\varepsilon}
\frac{\Big(|u(y-z) - u(y)|^2  -|u(y+z) - u(y)|^2\Big)
X(y-z)\cdot z}{{\varepsilon^{N+2s+1}}}\, d{\mathcal{H}}^{n-1}_z\, dy
.
\end{eqnarray*}
We notice that
\begin{eqnarray*}
|u(y-z) - u(y)|^2  -|u(y+z) - u(y)|^2=
\big(u(y-z) +u(y+z) -2 u(y)\big)\big(u(y-z)-u(y+z) \big),
\end{eqnarray*} hence
\begin{eqnarray*}&&\left|
\iint_{\R^{N}\times\partial B_\varepsilon(y)}
\frac{|u(x) - u(y)|^2  X(x)\cdot(x-y)}{\varepsilon^{N+2s+1}}\, d{\mathcal{H}}^{n-1}_x\, dy\right|\\
&&\quad\le CR^N \|u\|_{C^2(\R^N)}^2\|X\|_{L^\infty(\R^N,\R^N)}
\int_{\partial B_\varepsilon}| z|^4\, d{\mathcal{H}}^{n-1}_z\le
CR^N \|u\|_{C^2(\R^N)}^2\|X\|_{L^\infty(\R^N,\R^N)}{\varepsilon^{2-2s}},
\end{eqnarray*}
for some dimensional constant~$C>0$ (possibly changing step after step).

As a result,
\begin{equation}\label{nvkjew8to76439823456gerhg}\begin{split}&\left|
\int_{(0,1)}\left(C_{N,s}\iint_{\R^{N}\times\partial B_\varepsilon(y)}
\frac{|u(x) - u(y)|^2  X(x)\cdot(x-y)}{\varepsilon^{N+2s+1}}\, d{\mathcal{H}}^{n-1}_x\, dy\right)\, d\mu(s)\right|
\\&\qquad\qquad\le C\int_{(0,1)}C_{N,s} \varepsilon^{2-2s}\, d\mu(s),
\end{split}\end{equation}
where now~$C$ may also depend on~$u$ and~$X$.

In light of Lemma~\ref{rem:normalization_constant_analysis}, we see that
$$ \left|C_{N,s} \varepsilon^{2-2s}\right|\le C,$$
and also, for any~$s\in(0,1)$,
$$ \lim_{\varepsilon\searrow0}C_{N,s} \varepsilon^{2-2s}=0.$$
We can thereby make use of the Dominated Convergence Theorem and infer from~\eqref{nvkjew8to76439823456gerhg}
that
$$ \lim_{\varepsilon\searrow0}
\int_{(0,1)}\left(C_{N,s}\iint_{\R^{N}\times\partial B_\varepsilon(y)}
\frac{|u(x) - u(y)|^2  X(x)\cdot(x-y)}{\varepsilon^{N+2s+1}}\, d{\mathcal{H}}^{n-1}_x\, dy\right)\, d\mu(s)=0.$$

Using this and~\eqref{WQERT90843tihlkgewsgt9p54} into~\eqref{PREWQERT90843tihlkgewsgt9p54} we thus conclude that
\begin{equation}\label{cbnxwuwri4itk764583fjdsm}\begin{split}&
\int_{(0,1)}\left(C_{N,s}
\iint_{\R^{2N}} \frac{|u(x) - u(y)|^2}{|x - y|^{N+2s}} \mathcal{K}^s_X(x, y) \, dx\, dy\right)\, d\mu(s) 
\\&\qquad=
-\int_{\R^N}\left(
\int_{(0,1)}(-\Delta)^s u(x)\, d\mu(s)\right)\nabla u(x)\cdot  X(x)\,dx
.\end{split}
\end{equation}

We now deal with the cases~$s=1$ and~$s=0$. 
In the first case, we have that
\begin{equation}\label{456vcnmxrui345iDFGHJ65432q}\begin{split}&
\int_{\R^N} \Delta u(x)\, \nabla u (x)\cdot X(x)\, dx
=-\int_{\R^N} \nabla u(x)\cdot\nabla\big( \nabla u \cdot X\big)(x)\, dx
\\&\qquad=-\int_{\R^N} \partial_j u(x)\partial_j\big( \partial_i u X_i\big)(x)\, dx
\\&\qquad
=-\int_{\R^N} \partial_j u(x) \partial_{ij} u(x) X_i(x)\, dx
-\int_{\R^N} \partial_j u(x)\partial_i u(x) \partial_j X_i(x)\, dx\\&\qquad
=-\frac12\int_{\R^N} \partial_i |\nabla u|^2(x) X_i(x)\, dx
-\int_{\R^N} \partial_j u(x)\partial_i u(x) \partial_j X_i(x)\, dx\\&\qquad=
\frac12\int_{\R^N} |\nabla u(x)|^2 \,\mathrm{div} X(x)\, dx
-\int_{\R^N} \partial_j u(x)\partial_i u(x) \partial_j X_i(x)\, dx
.\end{split}\end{equation}

In the case~$s=0$, we have
\begin{eqnarray*}
&&\int_{\R^N} u(x)\, \nabla u (x)\cdot X(x)\, dx=
\frac12\int_{\R^N}  \nabla (u^2) (x)\cdot X(x)\, dx=
-\frac12\int_{\R^N} u^2 (x)\,\mathrm{div} X(x)\, dx.
\end{eqnarray*}
Gathering this, \eqref{cbnxwuwri4itk764583fjdsm} and~\eqref{456vcnmxrui345iDFGHJ65432q}, we obtain~\eqref{Pohozaev on superposition operator} for~$u\in C^\infty_0(\R^N)$.

We will now complete the proof of~\eqref{Pohozaev on superposition operator} via a density argument.
For this, let~$\varphi\in C^\infty_0(B_1,[0,1])$ be such that~$\|\varphi\|_{L^1(\R^N)}=1$ and, for
any~$\varepsilon\in(0,1)$, let~$\varphi_\varepsilon(x):=\frac1{\varepsilon^N}\varphi\left(\frac{x}{\varepsilon}\right)$.
Let also~$\tau \in C^\infty_0(B_2,[0,1])$ be such that~$\tau=1$ in~$B_1$ and let~$\tau_\varepsilon(x):=\tau(\varepsilon x)$.

We now take~$u$ as in the statement of Proposition~\ref{pohozaev superposition} and define~$u_\varepsilon:=(u*\varphi_\varepsilon)\tau_\varepsilon$. We point out that~$u_\varepsilon\in C^\infty_0(\R^N)$ and therefore, by~\eqref{Pohozaev on superposition operator} applied to~$u_\varepsilon$,
\begin{equation}\label{bvcwiuteuiw4yt798kjhbfdskSDFGH8765} \begin{split}&
\int_{(0,1)}\left(C_{N,s} \iint_{\R^{2N}}
\frac{|u_\varepsilon(x)-u_\varepsilon(y)|^2}{|x-y|^{N+2s}} \mathcal{K}_{X}^s(x,y)\, dx\, dy\right)\,d\mu(s) 
+ \frac{\mu(\{0\})}2\int_{\R^N} u_\varepsilon^2 (x)\,\mathrm{div} X(x)\, dx\\&\qquad+\mu(\{1\})\left(\frac12\int_{\R^N} |\nabla u_\varepsilon(x)|^2 \,\mathrm{div} X(x)\, dx
-\int_{\R^N} \partial_j u_\varepsilon(x)\partial_i u_\varepsilon(x) \partial_j X_i(x)\, dx
\right)\\&\qquad\qquad
= - \int_{\R^N} \mathcal{L}_\mu u_\varepsilon(x)\, \nabla u_\varepsilon \cdot X\, dx.\end{split}\end{equation}

Moreover, in light of the construction of the approximating sequence in~\cite[Appendix~B]{DPLSV}, we know that
$$ \lim_{\varepsilon\searrow0} \|u-u_\varepsilon\|_{\mu}=0.$$
As a result,
\begin{eqnarray*}&&\lim_{\varepsilon\searrow0}
\left|\int_{(0,1)}\left(C_{N,s} \iint_{\R^{2N}}
\frac{|u(x)-u(y)|^2-|u_\varepsilon(x)-u_\varepsilon(y)|^2}{|x-y|^{N+2s}} \mathcal{K}_{X}^s(x,y)\, dx\, dy\right)\,d\mu(s) 
\right|\\&&\;\le
C\|X\|_{C^1(\R^N,\R^N)}
\lim_{\varepsilon\searrow0}
\int_{(0,1)}\left(C_{N,s}\iint_{\R^{2N}}
\frac{\big||u(x)-u(y)|^2-|u_\varepsilon(x)-u_\varepsilon(y)|^2\big|}{|x-y|^{N+2s}} \, dx\, dy\right)\,d\mu(s) \\&&\;\le
C\|X\|_{C^1(\R^N,\R^N)}
\lim_{\varepsilon\searrow0}\|u-u_\varepsilon\|_{\mu}^2=0.
\end{eqnarray*}
Also, if~$\mu(\{0\})\neq0$,
\begin{eqnarray*}&&\lim_{\varepsilon\searrow0}
\left|\int_{\R^N}\big( u^2 (x)-u_\varepsilon^2(x)\big)\,\mathrm{div} X(x)\, dx\right|\le \|X\|_{C^1(\R^N,\R^N)}
\lim_{\varepsilon\searrow0}
\int_{\R^N}\big|u^2 (x)-u_\varepsilon^2(x)\big|\, dx\\&&\qquad\le
\|X\|_{C^1(\R^N,\R^N)}
\lim_{\varepsilon\searrow0}\|u-u_\varepsilon\|_{L^2(\R^N)}^2=0.
\end{eqnarray*}
Furthermore, if~$\mu(\{1\})\neq0$,
\begin{eqnarray*}&&\lim_{\varepsilon\searrow0}\left|
\int_{\R^N}\big( |\nabla u(x)|^2-|\nabla u_\varepsilon(x)|^2\big) \,\mathrm{div} X(x)\, dx\right|\le\|X\|_{C^1(\R^N,\R^N)}
\lim_{\varepsilon\searrow0}\int_{\R^N}\big| |\nabla u(x)|^2-|\nabla u_\varepsilon(x)|^2\big|\, dx\\&&\qquad\le
\|X\|_{C^1(\R^N,\R^N)}\|\nabla (u-u_\varepsilon)\|_{L^2(\R^N)}^2=0
\end{eqnarray*}
and
\begin{eqnarray*}&&\lim_{\varepsilon\searrow0}\left|
\int_{\R^N} \Big(\partial_j u(x)\partial_i u(x) -\partial_j u_\varepsilon(x)\partial_i u_\varepsilon(x)
\Big)\partial_j X_i(x)\, dx\right|\\&&\;\le\|X\|_{C^1(\R^N,\R^N)}
\lim_{\varepsilon\searrow0}
\int_{\R^N} \Big|\partial_j u(x)\partial_i u(x) -\partial_j u_\varepsilon(x)\partial_i u_\varepsilon(x)
\Big|\,dx\\&&\;\le\|X\|_{C^1(\R^N,\R^N)}
\lim_{\varepsilon\searrow0}
\int_{\R^N}|\partial_j u(x)|\,\big|\partial_i u(x) - \partial_i u_\varepsilon(x)
\big|\,dx+\int_{\R^N} |\partial_i u_\varepsilon(x) |\,\big|\partial_j u(x)-\partial_j u_\varepsilon(x)\big|\,dx
\\&&\;\le\|X\|_{C^1(\R^N,\R^N)}
\lim_{\varepsilon\searrow0}
\int_{\R^N}|\nabla  u(x)|\,\big|\nabla u(x) - \nabla u_\varepsilon(x)
\big|\,dx+\int_{\R^N} |\nabla u_\varepsilon(x) |\,\big|\nabla u(x)-\nabla u_\varepsilon(x)\big|\,dx
\\&&\;\le\|X\|_{C^1(\R^N,\R^N)}
\lim_{\varepsilon\searrow0}\|\nabla u\|_{L^2(\R^N)}\|\nabla (u-u_\varepsilon)\|_{L^2(\R^N)}
+\|\nabla u_\varepsilon\|_{L^2(\R^N)}\|\nabla (u-u_\varepsilon)\|_{L^2(\R^N)}\\&&\;=0.
\end{eqnarray*}

Gathering these pieces of information, we obtain the desired convergence of the left-hand side of~\eqref{bvcwiuteuiw4yt798kjhbfdskSDFGH8765}.
Hence, we now focus on establishing the desired convergence of the right-hand side.

To this end, we claim that
\begin{equation}\label{9843uigSDFGHJ043-ujgthtrh65oewjh4hgfdwef}
{\mbox{$\mathcal{L}_\mu u_\varepsilon$ converges to~$\mathcal{L}_\mu u$ in~$L^q_{loc}(\R^N)$ as~$\varepsilon\searrow0$.}}
\end{equation}
Indeed, we know that
\begin{equation}\label{bclw85t74hfkjhfajkbjbvbjdmslk34567hblks}
{\mbox{$(\mathcal{L}_\mu u)*\varphi_\varepsilon$ converges to~$\mathcal{L}_\mu u$ in~$L^q_{loc}(\R^N)$ as~$\varepsilon\searrow0$,}}
\end{equation} see e.g.~\cite[Theorem~9.6]{WZ2015}.

Moreover, thanks to the assumption in~\eqref{bcwiytr84y8tyihsWYU8765}, we can employ Fubini-Tonelli Theorem and obtain that
\begin{eqnarray*}&&
(\mathcal{L}_\mu u)*\varphi_\varepsilon(x)=
\int_{B_\varepsilon}\left( \int_{[0,1]}(-\Delta)^s u(x-y) \, d\mu(s)\right)\varphi_\varepsilon(y)\,dy\\&&\quad
= \iint_{[0,1]\times \R^N}C_{N,s}\left(\int_{B_\varepsilon}\big(2u(x-y)-u(x-y+z)-u(x-y-z)\big) \, \varphi_\varepsilon(y)\,dy\right)
\,\frac{dz}{|z|^{N+2s}}\,d\mu(s)\\
&&\quad=\iint_{[0,1]\times \R^N}C_{N,s}\frac{2u*\varphi_\varepsilon(x)-u*\varphi_\varepsilon(x+z)-u*\varphi_\varepsilon(x-z)}{|z|^{N+2s}}\,dz\,d\mu(s)\\
&&\quad=\mathcal{L}_\mu (u*\varphi_\varepsilon)(x).
\end{eqnarray*}
From this and~\eqref{bclw85t74hfkjhfajkbjbvbjdmslk34567hblks} we deduce that
\begin{equation}\label{questanoelanhnbfiuert765}
{\mbox{$\mathcal{L}_\mu (u*\varphi_\varepsilon)$ converges to~$\mathcal{L}_\mu u$ in~$L^q_{loc}(\R^N)$ as~$\varepsilon\searrow0$,}}
\end{equation}

Furthermore, given~$R>0$, we have that~$u_\varepsilon(x)=u*\varphi_\varepsilon(x)$
for all~$x\in B_R$ and~$\varepsilon$ sufficiently small, and therefore 
\begin{eqnarray*}&&\left|\mathcal{L}_\mu u_\varepsilon(x)-
\mathcal{L}_\mu (u*\varphi_\varepsilon)(x)\right|
\\&&\quad=\left|\int_{[0,1]}
C_{N,s} \left(\int_{\R^N}\frac{ u*\varphi_\varepsilon(x+y)-u*\varphi_\varepsilon(x-y)-u_\varepsilon(x+y)-u_\varepsilon(x-y) }{|y|^{N+2s} }\,dy \right)\,d\mu(s)\right|\\&&\quad\le
\int_{[0,1]}
C_{N,s} \left(\int_{\R^N}\frac{ |u*\varphi_\varepsilon(x+y)(1-\tau_\varepsilon(x+y))|+|u*\varphi_\varepsilon(x-y)(1-\tau_\varepsilon(x-y)| }{|y|^{N+2s} }\,dy \right)\,d\mu(s)\\
&&\quad=2\int_{[0,1]}
C_{N,s} \left(\int_{\R^N}\frac{ |u*\varphi_\varepsilon(z)(1-\tau_\varepsilon(z))|}{|x-z|^{N+2s} }\,dz \right)\,d\mu(s)\\
&&\quad\le 2\int_{[0,1]}
C_{N,s} \left(\int_{\R^N\setminus B_{1/\varepsilon}}\frac{ |u*\varphi_\varepsilon(z)|}{|x-z|^{N+2s} }\,dz \right)\,d\mu(s).
\end{eqnarray*}
Now, if~$x\in B_R$ and~$z\in\R^N\setminus B_{1/\varepsilon}$, if~$\varepsilon$ is sufficiently small, we have that~$ |x-z|\ge |z|-|x|\ge |z|/2$,
thus
\begin{eqnarray*}\left|\mathcal{L}_\mu u_\varepsilon(x)-
\mathcal{L}_\mu (u*\varphi_\varepsilon)(x)\right|&\le& 2\int_{[0,1]}
C_{N,s} 2^{N+2s}\left(\int_{\R^N\setminus B_{1/\varepsilon}}\frac{ |u*\varphi_\varepsilon(z)|}{|z|^{N+2s} }\,dz \right)\,d\mu(s).
\end{eqnarray*}

Hence, by the H\"older inequality and the Young's Convolution Theorem (see e.g.~\cite[Theorem~9.2]{WZ2015}),
\begin{eqnarray*}&&\left|\mathcal{L}_\mu u_\varepsilon(x)-
\mathcal{L}_\mu (u*\varphi_\varepsilon)(x)\right|\\&&\quad\le 2^{N+3}\int_{[0,1]}
C_{N,s}\left(\int_{\R^N\setminus B_{1/\varepsilon}}|u*\varphi_\varepsilon(z)|^2\,dz\right)^{1/2}\left(
\int_{\R^N\setminus B_{1/\varepsilon}} \frac{dz}{|z|^{2(N+2s)} }\right)^{1/2}\,d\mu(s)\\&&\quad\le
C\int_{[0,1]}
\frac{C_{N,s}\, \varepsilon^{\frac{N+4s}2 }}{\sqrt{N+4s}}\left(\int_{\R^N\setminus B_{1/\varepsilon}}|u*\varphi_\varepsilon(z)|^2\,dz\right)^{1/2}
\,d\mu(s)\\
&&\quad\le
C \varepsilon^{\frac{N}2 }\|u\|_{L^2(\R^N)} \int_{[0,1]}C_{N,s}\,d\mu(s),
\end{eqnarray*}
for some dimensional constant~$C>0$ (possibly changing from line to line).

Accordingly, recalling also Lemma~\ref{rem:normalization_constant_analysis}, we conclude that, for all~$R>0$,
$$\lim_{\varepsilon\searrow0}\sup_{x\in B_R}\left| \mathcal{L}_\mu u_\varepsilon(x)-
\mathcal{L}_\mu (u*\varphi_\varepsilon)(x)\right|=0.$$
From this observation and~\eqref{questanoelanhnbfiuert765}, we gather that, for all~$R>0$,
\begin{eqnarray*}\lim_{\varepsilon\searrow0}
\| \mathcal{L}_\mu u-\mathcal{L}_\mu u_\varepsilon\|_{L^q(B_R)}\le \lim_{\varepsilon\searrow0}
\| \mathcal{L}_\mu u-\mathcal{L}_\mu (u*\varphi_\varepsilon)\|_{L^q(B_R)}+\| \mathcal{L}_\mu (u*\varphi_\varepsilon) -\mathcal{L}_\mu u_\varepsilon\|_{L^q(B_R)}=0.
\end{eqnarray*}
This establishes~\eqref{9843uigSDFGHJ043-ujgthtrh65oewjh4hgfdwef}.

We now suppose that the support of~$X$ is contained in~$B_R$ for some~$R>0$.
Using~\eqref{9843uigSDFGHJ043-ujgthtrh65oewjh4hgfdwef}, we thus conclude that
\begin{eqnarray*}&&\lim_{\varepsilon\searrow0}
\left|\int_{\R^N} \mathcal{L}_\mu u (x)\, \nabla u  \cdot X\, dx-
\int_{\R^N} \mathcal{L}_\mu u_\varepsilon(x)\, \nabla u_\varepsilon \cdot X\, dx\right|\\&&\quad
\le \lim_{\varepsilon\searrow0}
\left|\int_{\R^N} \mathcal{L}_\mu u (x)\, \nabla (u-u_\varepsilon)  \cdot X\, dx\right|+\left|
\int_{\R^N} \mathcal{L}_\mu (u-u_\varepsilon)(x)\, \nabla u_\varepsilon \cdot X\, dx\right|\\&&\quad\le\|X\|_{L^\infty(\R^N,\R^N)}\left(
\lim_{\varepsilon\searrow0}
\int_{B_R} |\mathcal{L}_\mu u (x)|\, |\nabla (u-u_\varepsilon) |\, dx+
\int_{B_R} |\mathcal{L}_\mu (u-u_\varepsilon)(x)|\,| \nabla u_\varepsilon|\, dx\right)\\&&\quad\le
\|X\|_{L^\infty(\R^N,\R^N)}\left(
\lim_{\varepsilon\searrow0}\|\mathcal{L}_\mu u \|_{L^q(B_R)}\|\nabla (u-u_\varepsilon) \|_{L^p(B_R)}+
\|\mathcal{L}_\mu (u-u_\varepsilon)\|_{L^q(B_R)}\|\nabla u_\varepsilon\|_{L^p(B_R)}\right)\\&&\quad=0.
\end{eqnarray*}
This entails the desired convergence also of the right-hand side of~\eqref{bvcwiuteuiw4yt798kjhbfdskSDFGH8765}, and 
the proof of~\eqref{Pohozaev on superposition operator} is thereby complete.
\end{proof}

With this preliminary work, we can complete the proof of Theorem~\ref{pohozaev theorem}.

\begin{proof}[Proof of Theorem~\ref{pohozaev theorem}]
Let~\( X \in C^1_0(\mathbb{R}^N, \mathbb{R}^N) \).
By Proposition \ref{Pohozaev right term} (used here with~\( H := F(u) \)) and
Proposition \ref{pohozaev superposition}, we obtain that 
\begin{equation*} \begin{split}&
\int_{(0,1)}\left( C_{N,s} \iint_{\R^{2N}}
\frac{|u(x)-u(y)|^2}{|x-y|^{N+2s}} \mathcal{K}_{X}^s(x,y)\, dx\, dy\right)\,d\mu(s) 
+ \frac{\mu(\{0\})}2\int_{\R^N} u^2 (x)\,\mathrm{div} (X)(x)\, dx\\&\qquad+\mu(\{1\})
\left(\frac12\int_{\R^N} |\nabla u(x)|^2 \,\mathrm{div} (X)(x)\, dx
-\int_{\R^N} \partial_j u(x)\partial_i u(x) \partial_j X_i(x)\, dx
\right)\\&\qquad\qquad
= - \int_{\R^N} \mathcal{L}_\mu u\, \nabla u \cdot X\, dx\\
&\qquad\qquad=
 \int_{\R^N} \Big(\beta u-
 (I_\alpha \ast F(u))f(u)\Big) \nabla u \cdot X\, dx\\
 &\qquad\qquad=\frac{\beta}{2}
 \int_{\R^N} \nabla (u^2) \cdot X\, dx-\int_{\R^N}
 (I_\alpha \ast F(u)) \nabla F(u) \cdot X\, dx\\
 &\qquad\qquad=-\frac{\beta}{2}
\int_{\R^N} u^2\,{\mathrm{div}}( X)\, dx+
\iint_{\mathbb{R}^{2N}} I_\alpha(x - y)\,F(u(x))\, F(u(y))\, \mathcal{K}^{-\frac{\alpha}{2}}_X(x,y)\, dx\, dy.
\end{split}\end{equation*}

We now specialize the choice of the vector field, by taking a function~$\varphi\in C^\infty_0(B_2,[0,1])$ with~$\varphi=1$ in~$B_1$ and defining~$\varphi_n(x) := \varphi\left( \frac{x}{n} \right)$ and~$X_n(x) := \varphi_n(x) x$.

We point out that~\( \varphi_n \in C^\infty_0(B_{2n},[0,1])\), \( \varphi_n =1 \) in~\( B_n\) and~\( |x|\,|\nabla \varphi_n(x)| \leq C \) for each~\( x \in \mathbb{R}^N \) and~\( n \in \mathbb{N} \). Also, $\varphi_n\to1$ and~$\nabla\varphi_n\to0$ as~$n\to+\infty$. Hence,
$X_n$ converges to~$x$ as~$n\to+\infty$.

Moreover, we have that~$ \operatorname{div}(X_n) = N \varphi_n + \nabla \varphi_n \cdot x$
and~$\partial_j X_i=\partial_j\varphi_n x_i+\varphi_n\delta_{ij}$.

Therefore, recalling the definition of~$\mathcal{K}_{X}^s$
in~\eqref{DFGHJbcxm84thdkslkfi7uit},
\begin{eqnarray*}&&
\int_{(0,1)}\left(C_{N,s} \iint_{\R^{2N}}
\frac{|u(x)-u(y)|^2}{|x-y|^{N+2s}} \mathcal{K}_{X}^s(x,y)\, dx\, dy\right)\,d\mu(s) \\
&&\quad= \int_{(0,1)}\left(\frac{C_{N,s}N}{2} \iint_{\mathbb{R}^{2N}}\frac{|u(x) - u(y)|^2}{|x - y|^{N + 2s}} \big(\varphi_n(x) + \varphi_n(y)\big)\, dx\,dy\right)\,d\mu(s)  \\
    &&\qquad\quad + \int_{(0,1)}\left( \frac{C_{N,s}}{2} \iint_{\mathbb{R}^{2N}}
    \frac{|u(x) - u(y)|^2}{|x - y|^{N + 2s}} \Big(\nabla \varphi_n(x) \cdot x + \nabla \varphi_n(y) \cdot y\Big) \, dx\,dy\right)\, d\mu (s)\\
    &&\qquad\quad- \int_{(0,1)}\left(  \frac{C_{N,s}(N+2s)}{2} \iint_{\mathbb{R}^{2N}}
    \frac{|u(x) - u(y)|^2}{|x - y|^{N + 2s}}  \frac{(\varphi_n(x)x - \varphi_n(y)y) \cdot (x - y)}{|x - y|^2} \, dx\,dy \right)\,d\mu(s).
\end{eqnarray*}
We can now pass the above equation to the limit as~$n\to+\infty$ using the Dominated Convergence Theorem and we obtain that
\begin{eqnarray*}&
    &\lim_{n \to \infty} \int_{(0,1)}\left(C_{N,s} \iint_{\R^{2N}}
\frac{|u(x)-u(y)|^2}{|x-y|^{N+2s}} \mathcal{K}_{X}^s(x,y)\, dx\, dy\right)\,d\mu(s)\\&&\quad
  =   \int_{(0,1)}\left(  C_{N,s}N \iint_{\mathbb{R}^{2N}}\frac{|u(x) - u(y)|^2}{|x - y|^{N + 2s}}\, dx\,dy\right)\,d\mu(s)  \\
    &&\qquad\quad - \int_{(0,1)}\left(  \frac{C_{N,s}(N+2s)}{2} \iint_{\mathbb{R}^{2N}}
    \frac{|u(x) - u(y)|^2}{|x - y|^{N + 2s}} \, dx\,dy \right)\,d\mu(s)\\&&\quad=
      \int_{(0,1)}\frac{N-2s}{2} [u]_s^2\,d\mu(s).
    \end{eqnarray*}
    
Also, if~$\mu(\{0\})\neq0$,
 \begin{eqnarray*}\lim_{n\to+\infty}  
\frac12\int_{\R^N} u^2 (x)\,\mathrm{div} (X_n)(x)\, dx
=\frac12\int_{\R^N} u^2 (x)\,\Big(N \varphi_n + \nabla \varphi_n \cdot x\Big)\, dx=\frac{N}2\|u\|^2_{L^2(\R^N)}.
\end{eqnarray*}  
Furthermore, if~$\mu(\{1\})\neq0$,
\begin{eqnarray*}&&
\lim_{n\to+\infty}  \frac12\int_{\R^N} |\nabla u(x)|^2 \,\mathrm{div} (X_n)(x)\, dx
-\int_{\R^N} \partial_j u(x)\partial_i u(x) \partial_j X_{n,i}(x)\, dx\\
&&\quad=\lim_{n\to+\infty}  \frac12\int_{\R^N} |\nabla u(x)|^2 \big(N \varphi_n + \nabla \varphi_n \cdot x\big)\, dx
-\int_{\R^N} \partial_j u(x)\partial_i u(x) \big(\partial_j\varphi_n x_i+\varphi_n\delta_{ij}\big)\, dx\\&&\quad=
\frac{N}2\int_{\R^N} |\nabla u(x)|^2\, dx-\int_{\R^N} |\nabla u(x)|^2\, dx\\
&&\quad=\frac{N-2}2\int_{\R^N} |\nabla u(x)|^2\, dx.
\end{eqnarray*}
Gathering these pieces of information, we thereby find that
\begin{eqnarray*}&&\lim_{n\to+\infty}
\int_{(0,1)}\left(C_{N,s} \iint_{\R^{2N}}
\frac{|u(x)-u(y)|^2}{|x-y|^{N+2s}} \mathcal{K}_{X_n}^s(x,y)\, dx\, dy\right)\,d\mu(s) 
+ \frac{\mu(\{0\})}2\int_{\R^N} u^2 (x)\,\mathrm{div} (X_n)(x)\, dx\\&&\qquad+\mu(\{1\})
\left(\frac12\int_{\R^N} |\nabla u(x)|^2 \,\mathrm{div} (X_n)(x)\, dx
-\int_{\R^N} \partial_j u(x)\partial_i u(x) \partial_j X_{n,i}(x)\, dx
\right)\\&&\quad=  \int_{[0,1]}\frac{N-2s}{2} [u]_s^2\,d\mu(s).
\end{eqnarray*}
  
  Similarly (see~\cite[Proposition~6.5]{CGT1}), we see that
\begin{align*}
    &\lim_{n \to \infty}\iint_{\mathbb{R}^{2N}} I_\alpha(x - y)\,F(u(x))\, F(u(y))\, \mathcal{K}^{-\frac{\alpha}{2}}_{X_n}(x,y)\, dx\, dy
    \\ &\quad=  \frac{2N - (N -\alpha)}{2} \iint_{\mathbb{R}^{2N}}I_\alpha(x - y) F(u(x)) F(u(y)) \, dx\,dy \\
    &\quad= \frac{N+\alpha}{2} \int_{\mathbb{R}^N} (I_\alpha * F(u)) F(u) \, dx
\end{align*}
and $$
    \lim_{n \to \infty} \frac{\beta}{2}
\int_{\R^N} u^2\,{\mathrm{div}}( X_n)\, dx=  \frac{N\beta}{2} \|u\|^2_{2}. 
$$

Putting all the pieces together, we obtain the desired identity.
\end{proof}

\section*{Acknowledgments}
Alessandro Cannone is supported by D.M. 2023 
n. 118- PNRR, ``PDEs from Quantum Science", 
CUP: H91I23000500007 and by GNAMPA research project 
``Aspetti qualitativi e analisi di blow-up per problemi 
differenziali ellittici" CUP:  E53C25002010001. Silvia Cingolani is supported by INdAM-GNAMPA.
Serena Dipierro is supported by the Australian Future Fellowship FT230100333 ``New perspectives on nonlocal equations''.
Part of this work was carried out during a visit of SD to the University of Bari, which we thank for the warm hospitality.

\end{document}